\pdfoutput=1
\documentclass[11pt]{article} 

\newcommand{\siam}[1]{foo}
\newcommand{\arxiv}[1]{bar}
\ifdefined \DoSIREV
\renewcommand{\siam}[1]{#1}%
\renewcommand{\arxiv}[1]{\ignorespaces}%
\else%
\renewcommand{\siam}[1]{\ignorespaces}%
\renewcommand{\arxiv}[1]{#1}%
\fi%

\usepackage{amssymb}
\usepackage{amsbsy}
\usepackage{amsmath}
\usepackage{amsfonts}
\usepackage{latexsym}
\usepackage{graphicx}
\usepackage{color}
\usepackage{xifthen}
\usepackage{xspace}
\usepackage{mathtools}
\usepackage{enumerate}

\usepackage{algorithm}
\usepackage[noend]{algpseudocode}

\usepackage{multirow}
\usepackage[normalem]{ulem}
\usepackage{accents}

\usepackage{array}
\usepackage[T1]{fontenc}
\usepackage{thmtools}

\usepackage{enumitem}

\usepackage{etoolbox}
\newtoggle{heavyplots}

\usepackage{xfrac}

\arxiv{%
  \usepackage[top=1in, right=3cm, left=3cm, bottom=1in]{geometry}
  \usepackage{amsthm}
  \usepackage{xcolor}
  \definecolor{darkblue}{rgb}{0,0,.5}
  \usepackage[colorlinks=true,allcolors=darkblue]{hyperref}

  \usepackage{thmtools}
  \usepackage{thm-restate}

  \theoremstyle{plain}
  \newtheorem{theorem}{Theorem}
  \newtheorem{lemma}{Lemma}
  \newtheorem{proposition}{Proposition}
  \newtheorem{corollary}[theorem]{Corollary}

  \let\oldparagraph=\paragraph
  \renewcommand\paragraph[1]{\oldparagraph{#1.}}

}

\ifdefined\DoSIREV
\else
\makeatletter
\long\def\@makecaption#1#2{
  \vskip 0.8ex
  \setbox\@tempboxa\hbox{\small {\bf #1:} #2}
  \parindent 1.5em  %
  \dimen0=\hsize
  \advance\dimen0 by -3em
  \ifdim \wd\@tempboxa >\dimen0
  \hbox to \hsize{
    \parindent 0em
    \hfil 
    \parbox{\dimen0}{\def\baselinestretch{0.96}\small
      {\bf #1.} #2
    } 
    \hfil}
  \else \hbox to \hsize{\hfil \box\@tempboxa \hfil}
  \fi
}
\makeatother
\fi

\newtheorem{assumption}{Assumption}

\newcommand{\mc}[1]{\mathcal{#1}}

\newcommand{\wt}[1]{\widetilde{#1}}  %
\newcommand{\wb}[1]{\overline{#1}} %

\newcommand{\norm}[1]{\left\|{#1}\right\|} %
\newcommand{\normbigg}[1]{\bigg\|{#1}\bigg\|} %
\newcommand{\norms}[1]{\|{#1}\|} %

\newcommand{\opnorm}[1]{\norm{#1}}  %
\newcommand{\R}{\mathbb{R}} %
\newcommand{\N}{\mathbb{N}} %
\renewcommand{\P}{\mathbb{P}}	%
\newcommand{\I}{\mathbb{I}} %
\newcommand{\sphere}{\mathbb{S}} %
\providecommand{\argmin}{\mathop{\rm argmin}}

\providecommand{\diag}{\mathop{\rm diag}}

\newcommand{\hinge}[1]{\left({#1}\right)_+} %

\providecommand{\minimize}{\mathop{\rm minimize}}

\providecommand{\subjectto}{\mathop{\rm subject\;to}}
\newcommand{\ceil}[1]{\left\lceil{#1}\right\rceil}

\newcommand{\half}{\frac{1}{2}}

\newcommand{\defeq}{\coloneqq}
\newcommand{\eqdef}{\eqqcolon}

\newcommand{\grad}{\nabla}
\newcommand{\hess}{\nabla^2}

\def\ie{\emph{i.e.\ }}

\newcommand{\Otil}[1]{\widetilde{O}\left( #1 \right)}

\newcommand{\T}{T} %
\newcommand{\eps}{\varepsilon}

\newcommand{\del}{\partial}

\newcommand{\Ls}{L\subopt}

\newcommand{\IndNC}{}

\newcommand{\tgrow}{\tau_{\mathrm{grow}}}
\newcommand{\tconv}{\tau_{\mathrm{converge}}}
\newcommand{\tiltgrow}{\tilde{\tau}_{\mathrm{grow}}}
\newcommand{\tiltconv}{\tilde{\tau}_{\mathrm{converge}}}

\newcommand{\sigbar}{\wb{\sigma}}

\newcommand{\smallval}{\nu}

\newcommand{\CP}{Cauchy-point}
\newcommand{\SP}{Find-SOSP}
\newcommand{\SSP}{Solve-CR}
\newcommand{\SFSP}{Solve-Quadratic}

\newcommand{\slb}{r}
\newcommand{\siter}{\Delta^\star}

\newcommand{\glb}{g\subopt}

\newcommand{\epsgrad}{\eps_{\mathrm{g}}}

\newcommand{\subfin}{_{\textup{final}}}
\newcommand{\Tinner}{T_{\textup{inner}}}
\newcommand{\Touter}{T_{\textup{outer}}}
\newcommand{\Tfinal}{T\subfin}

\newcommand{\callCP}[1]{\hyperref[func:CP]{\Call{\CP}{#1}}}
\newcommand{\callSP}[1]{\hyperref[func:SP]{\Call{\SP}{#1}}}
\newcommand{\callSSP}[1]{\hyperref[func:SSP]{\Call{\SSP}{#1}}}
\newcommand{\callSFSP}[1]{\hyperref[func:SFSP]{\Call{\SFSP}{#1}}}

\newcommand{\f}[1][A,b]{f_{#1}}
\newcommand{\fcu}[1][A,b,\rho]{f_{#1}}

\newcommand{\As}{A\subopt}

\newcommand{\scu}{x_{\star}^{\mathsf{cr}}}  %
\newcommand{\sol}{x_{\star}^{\mathsf{cr}}}  %
\newcommand{\soltilde}{\tilde{x}_{\star}^{\mathsf{cr}}}  %
\newcommand{\Rupper}{R_\rho}
\newcommand{\Rcauchy}{R_{\mathsf{cr}}}  %
\newcommand{\xcauchy}{x_{\mathsf{cr}}}  %

\newcommand{\str}[1][]{{x\subopt^{{\mathsf{tr}}#1}}} %
\newcommand{\ltr}{\lambda_{\mathsf{tr}}}
\newcommand{\Atr}{A_{\ltr}}

\newcommand{\lmin}{\lambda_{\min}}
\newcommand{\lmax}{\lambda_{\max}}
\newcommand{\vmin}{v_{\min}}
\newcommand{\vmax}{v_{\max}}

\newcommand{\itertr}{x^{\mathsf{tr}}}
\newcommand{\gradmap}{\mathsf{G}}

\newcommand{\rs}{\rho \norm{\scu}}

\newcommand{\jittertr}{\hat{x}^{\mathsf{tr}}}  %
\newcommand{\jittercu}{\hat{x}^{\mathsf{cr}}}  %

\newcommand{\itercu}{x^{\mathsf{cr}}}

\DeclareDocumentCommand{\Krylov}{ O{t} O{A,b} }{\mc{K}_{#1}(#2)}

\DeclareDocumentCommand{\minmaxU}{ O{n} O{\kappa} }{
  \mathfrak{U}_{#1}\left(#2\right)}

\newcommand{\supind}[1]{^{\left({#1}\right)}}

\newcommand{\opt}{^\star}
\newcommand{\subopt}{_\star}

\newcommand{\uniform}{\mathsf{Uni}} %
\newcommand{\univar}{u}  %
\newcommand{\betadist}{\mathsf{Beta}} %

\definecolor{innerboxcolor}{rgb}{.9,.95,1}
\definecolor{outerlinecolor}{rgb}{.6,0,.2}

 \usepackage[numbers,square]{natbib}

\newcommand\blfootnote[1]{%
	\begingroup
	\renewcommand\thefootnote{}\footnote{#1}%
	\addtocounter{footnote}{-1}%
	\endgroup
}

\numberwithin{theorem}{section}
\numberwithin{lemma}{section}
\numberwithin{proposition}{section}

\siam{
\headers{First-order methods for nonconvex quadratic minimization}{Y. 
Carmon \and J.C. Duchi}
}

\siam{
\title{{First-order methods for nonconvex \\ quadratic
    minimization}\thanks{February 19, 2020.  \funding{YC and JCD were
      partially supported by the SAIL-Toyota Center for AI Research and the
      Office of Naval Research award N00014-19-2288. YC was partially
      supported by the Stanford Graduate Fellowship and the Numerical
      Technologies Fellowship. JCD was partially supported by the National
      Science Foundation award NSF-CAREER-1553086}}}
 \author{
 	Yair Carmon\thanks{Department of Electrical Engineering, Stanford 
 		University
 		(\email{yairc@stanford.edu}).}
 	\and
 	John C. Duchi\thanks{Departments of Statistics and Electrical 
 	Engineering, 
 		Stanford University (\email{jduchi@stanford.edu}).}
 }
}

\arxiv{
  \title{First-Order Methods for Nonconvex Quadratic Minimization}

  \author{Yair Carmon ~~~ John C.\ Duchi\\   
    \texttt{\{\href{mailto:yairc@stanford.edu}{yairc},%
      \href{mailto:jduchi@stanford.edu}{jduchi}\}@stanford.edu}}
  \date{}
}

\toggletrue{heavyplots}

\begin{document}
\maketitle

\begin{abstract}
  We consider minimization of indefinite quadratics with either trust-region 
  (norm) constraints or cubic regularization. Despite the nonconvexity of 
  these problems we prove that, under mild assumptions, gradient descent 
  converges to their global solutions, and give a non-asymptotic rate of 
  convergence for the cubic variant. We also consider Krylov subspace 
  solutions and establish sharp convergence guarantees to the solutions of 
  both trust-region and cubic-regularized problems.
  Our rates mirror the behavior of these methods on convex 
  quadratics and eigenvector problems, highlighting their scalability. 
  When we use Krylov subspace solutions to approximate the 
  cubic-regularized Newton step, our results recover the strongest known 
  convergence guarantees to approximate second-order stationary points of 
  general smooth nonconvex functions.
  \blfootnote{%
    This is a SIAM Review preprint covering our papers~\cite{CarmonDu19} 
    and~\cite{CarmonDu18}; some materials in 
    Section~\ref{sec:majorization} are new.
  }
\end{abstract}

\siam{

}

\newcommand{\regpenalty}{\mathsf{reg}}

\section{Introduction}

Consider the potentially nonconvex quadratic function
\begin{equation*}
  \f(x) \defeq \half x^T A x + b^T x,
\end{equation*}
where $A \in \R^{d \times d}$ is symmetric and possibly indefinite and $b 
\in \R^d$. 
We wish to solve the 
problems
\renewcommand{\theequation}{\textsc{P.tr}}
\begin{equation}
  \label{eqn:problem-tr}
  \minimize_{x\in\R^d}
  \f(x) ~ \subjectto \norm{x} \le R
\end{equation}
and
\renewcommand{\theequation}{\textsc{P.cu}}
\begin{equation}
  \label{eqn:problem-cubic}
  \minimize_{x\in\R^d}   \fcu(x) \defeq \f(x) + \frac{\rho}{3} \norm{x}^3,
\end{equation}
where $R$ and $\rho \ge 0$ are regularization parameters.  These problems
arise primarily in the family of trust-region and cubic-regularized Newton
methods for general nonlinear 
optimization~\cite{ConnGoTo00,NesterovPo06, Griewank81,CartisGoTo11},
which optimize a smooth function $g$ by iteratively
minimizing second-order models of $g$ centered at
an iterate $x$, which take the form
\begin{equation*}
  g(y) \approx g_x(y) \defeq
  g(x) +
  \underbrace{\grad g(x)^T(y - x) + \half (y - x)^T \grad^2 g(x) (y - x)}_{
    = \f[\hess g(x), \grad g(x)](y-x)}.
\end{equation*}
Such models tend to be unreliable when $y$ is far from $x$, particularly
in the nonconvex setting when it is possible that $\grad^2 g(x) \not\succ 
0$.
Trust-region and cubic regularization models address this by
instead (approximately) iterating
\renewcommand{\theequation}{\arabic{equation}}
\setcounter{equation}{0}
\begin{equation}
  \label{eqn:generic-problem-iteration}
  y_{k + 1} \approx \argmin_\Delta
  \left\{\f[\hess g(y_k), \grad g(y_k)](\Delta)
  + \regpenalty(\norm{\Delta})\right\},
\end{equation}
where $\regpenalty$ regularizes large $\norm{\Delta}$; in trust-region
methods by a hard constraint so that the model $\f[\hess g(y_k), \grad
  g(y_k)]$ is accurate, and in cubic-regularization methods by
$\norm{\Delta}^3$ so that the penalized model is a locally accurate upper
bound on $g$~\cite{ConnGoTo00, NesterovPo06, CartisGoTo11}.  Trust-region
and cubic-regularized model-based
methods offer a principled and
powerful platform for integrating second-order information into the 
optimization procedure.

The centrality of these methods motivates considerable interest in solving
their corresponding subproblems~\cite{ConnGoTo00, NocedalWr06, CartisGoTo11,
  HazanKo16, Ho-NguyenKi17, ZhangShLi17}. This becomes computationally
challenging in high-dimensional settings, where direct decomposition (or
even storage) of the matrix $A$ is infeasible. In many scenarios, however,
computing matrix-vector products $v \mapsto Av$ is feasible.
As particular examples, when $A$ is sparse or given explicitly by a 
low-rank factorization, this is feasible;
if $A = \hess g(x)$ for a smooth function $g$, then $A v
= \frac{\grad g(x + tv) - g(x)}{t} + O(t)$ is approximable to arbitrary
accuracy by finite differences;
if $A$ is the Hessian of a neural network, we can compute Hessian-vector 
products
efficiently on batches of training data~\cite{Pearlmutter94,Schraudolph02} 
via back-propagation.

\subsection{Outline of methods and our contribution} %

We study first-order methods for solving problems~\eqref{eqn:problem-tr}
and~\eqref{eqn:problem-cubic} that access the matrix $A$ only through
matrix-vector product evaluations. Our main goal is to characterize the
number of evaluations these methods require to reach a desired accuracy
$\eps$ in the regime where the problem dimension $d$ is very high.  We
establish nearly dimension-free bounds---depending at most 
logarithmically on
$d$---highlighting the scalability of first-order methods.  In particular,
we study gradient descent and Krylov subspace methods, lynchpins of 
optimization and in frequent use for problems~\eqref{eqn:problem-tr}
and~\eqref{eqn:problem-cubic}. %

\paragraph{Gradient descent}
For the trust-region problem~\eqref{eqn:problem-tr}, gradient descent 
iterates
\begin{equation}
  \label{eq:grad-iter-tr}
  x_{t + 1} =
  \Pi_R(x_t - \eta \grad \f(x_t))
  = \Pi_R(x_t - \eta (Ax_t + b)),
\end{equation}
where $\eta >0$ is a step size parameter and 
$\Pi_R$ is the Euclidean projection to the ball of radius $R$. For the
cubic-regularized problem~\eqref{eqn:problem-cubic}, it is simply 
\begin{equation}
  \label{eq:grad-iter}
  x_{t + 1} =
  x_t - \eta \grad \fcu(x_t)
  = (I - \eta A - \rho \eta \norm{x_t} I) x_t - \eta b.
\end{equation}
In neither case is $f$ necessarily convex, so it is not \emph{a-priori}
clear that gradient descent even converges to global subproblem solutions;
we establish such global convergence under standard and weak assumptions on
$\eta$ and the initialization $x_0$.

For the cubic regularized problem we prove that the number of 
steps to reach accuracy $\eps$ scales at most as
$\min\{\kappa, \frac{1}{\eps}\} \log \frac{1}{\alpha \eps}$, where 
$\kappa$ 
is a problem-dependent condition number and $\alpha = |\vmin^T 
b|/\norm{b}$ is the normalized 
inner product between $b$ and the eigenvector of $A$ corresponding to its 
smallest eigenvalue. We establish these rates by breaking the gradient 
descent trajectory into two phases and bounding their durations; the first 
stage consists of the iterate norm rapidly growing away from the origin 
(thereby escaping all saddle points), while the second stage consists of 
contraction towards the global solution.

\paragraph{Krylov subspace methods}
Krylov subspace methods 
iterate for $t = 1, 2, \ldots$ by solving
the problems~\eqref{eqn:problem-tr} and~\eqref{eqn:problem-cubic} 
over the \emph{Krylov subspaces}
\begin{equation}
  \label{eqn:krylov-subspace}
  \Krylov[t][A,b]
  \defeq \mathrm{span}\{b, Ab, \ldots, A^{t-1}b\}
  = \{p(A)b\mid \text{$p$ is a degree $t-1$ polynomial}\},
\end{equation}
iteratively setting
\begin{equation}
  \label{eqn:krylov-iterates}
  \itertr_t = \argmin_{x \in \Krylov} \{\f(x)
  \mid \norm{x} \le R\}
  ~~ \mbox{or} ~~
  \itercu_t = \argmin_{x \in \Krylov} \{\fcu(x)\}
\end{equation}
for problems~\eqref{eqn:problem-tr} and~\eqref{eqn:problem-cubic},
respectively.  The Lanczos method can compute these solutions in time
dominated by the matrix-vector product cost
(see~\cite[Sec.~2]{GouldLuRoTo99,CartisGoTo11} and
Appendix~\ref{sec:lanczos}). Krylov subspace methods are familiar for
large-scale numerical problems, including conjugate gradient
methods, eigenvector problems, and the solution of linear 
systems~\cite{HestenesSt52,
  Nemirovski94,TrefethenBa97,GolubVa89}.

Since the $t$th iteration of gradient descent (initialized at the origin)
lies in $\mc{K}_t(A,b)$, Krylov subspace methods converge faster than
gradient descent by construction. We prove that they are in fact
quadratically faster, showing that $\mc{K}_t(A,b)$ contains an
$\eps$-optimal solution in at most $\min\{\sqrt{\kappa}\log\frac{1}{\eps},
\frac{1}{\sqrt{\eps}}\log \frac{1}{\alpha}\}$ iterations, with $\kappa$ and
$\alpha$ as defined above; this bound applies to both trust-region and
cubic-regularized subproblems. Our analysis follows the well-established
practice of appealing to uniform polynomial
approximations~\cite{TrefethenBa97, Nemirovski94} to construct ``good''
elements in $\mc{K}_t(A,b)$ achieving the
desired convergence. Complementing this approach, we  construct 
additional  reference elements in 
$\mc{K}_t(A,b)$ based on Nesterov's accelerated gradient 
method~\cite{Nesterov83,Nesterov04,Tseng08}.
The Krylov iterates~\eqref{eqn:krylov-iterates}
are then by construction better.

For both gradient descent and Krylov subspace methods, our rates of
convergence mirror and unify well-known guarantees for two special cases:
convex problems ($A\succeq 0$) and the eigenvector problem ($b=0$ and
$A\nsucc0$)~\cite{GolubVa89,TrefethenBa97}; see
Section~\ref{sec:discussion}.

\oldparagraph{Randomization for the ``hard case.''} 
The above iteration count bounds become vacuous for problem instances 
where $\kappa = \infty$ and $\rho=0$, the ``hard 
case''~\cite{ConnGoTo00}. We provide two randomization techniques: the 
first slightly perturbs $b$, and the second  expands the Krylov 
basis~\eqref{eqn:krylov-subspace} in a random direction. These techniques 
allow us to replace the term $\alpha= |\vmin^T b|/\norm{b}$ with 
$\frac{1}{\sqrt{d}}$, thus yielding high-probability convergence rates of the 
form $\frac{1}{\eps}\log\frac{d}{\eps}$ for gradient descent and 
problem~\eqref{eqn:problem-cubic}, and $\frac{1}{\sqrt{\eps}}\log d$ for 
Krylov subspace methods for~\eqref{eqn:problem-tr} and 
\eqref{eqn:problem-cubic}.

\paragraph{A first-order implementation of cubic-regularized 
Newton steps}
Returning to the model-based nonlinear optimization methods motivating 
our work, we integrate our Krylov subspace solver into a simple version of 
the cubic-regularized Newton 
method~\cite{Griewank81,NesterovPo06,WeiserDeEr07}. Leveraging the 
analysis of~\citet{NesterovPo06} and our convergence guarantees, we show 
that for a function $g$ with Lipschitz gradient and Hessian, a
method approximating the iteration~\eqref{eqn:generic-problem-iteration}
finds an $\epsilon$ second-order stationary point (satisfying 
$\norm{\grad g(x)} \le \epsilon$ and $\lambda_{\min}(\hess g(x)) \gtrsim 
-\sqrt{\epsilon}$) with roughly $\epsilon^{-3/2}$ 
gradient evaluations and $\epsilon^{-7/4}\log\frac{d}{\epsilon}$ 
Hessian-vector product evaluations. In comparison, simply applying 
gradient descent on $g$ requires $\epsilon^{-2}$ gradient evaluations to 
guarantee $\norm{\grad g(x)} \le \epsilon$ and does not provide a 
near-positivity guarantee on the Hessian.

\subsection{Prior work}
Despite their nonconvexity, it is possible to solve the
subproblems~\eqref{eqn:problem-tr} and~\eqref{eqn:problem-cubic} to machine
precision by iterative solution to linear systems of the form $(A+\lambda
I)x=-b$ with Newton-type procedures for the scalar
$\lambda$~\cite{ConnGoTo00,CartisGoTo11}. To handle large scale instances,
earlier work proposes both heuristic variants of the conjugate gradient
method~\cite{Griewank81,Steihaug83} and Krylov subspace
solutions~\cite{GouldLuRoTo99,CartisGoTo11}. While these works 
demonstrate strong practicality and are in common use, they do not bound 
the
iterations required to obtain approximate solutions.\footnote{%
  For almost all $A,b$, the Krylov subspace of order $d$ is $\R^d$, and
  consequently $d$ steps solve~\eqref{eqn:problem-tr}
  and~\eqref{eqn:problem-cubic} in exact arithmetic. However, guarantees of
  this type break down under finite precision~\cite{TrefethenBa97} and
  provide limited insight for high-dimensional problems, where the number of
  iterations is typically $\ll d$.  } Several
works~\cite{TaoAn98,BianconciniLiMoSc15,BeckVa18} also apply variants of
gradient descent to the subproblems~\eqref{eqn:problem-tr}
and~\eqref{eqn:problem-cubic} yet without dimension-free convergence
guarantees.

A recent thread of research has begun to give (nearly) dimension-free
theoretical bounds for
first-order-like methods. \citet{HazanKo16} give the first such guarantee,
finding an $\eps$-approximate solution with $\Otil{{1}/{\sqrt{\eps}}}$
matrix-vector products by reducing the trust-region subproblem to a sequence
of eigenvector problems and solving them approximately with an efficient
first-order method.  \citet{Ho-NguyenKi17} provide a different perspective,
using a single eigenvector calculation to reformulate the nonconvex
quadratic trust-region problem into a convex quadratically constrained 
quadratic program.
Unfortunately, these methods are less conducive to efficient implementation
than those above: each has several parameters that require
tuning, and we are unaware of numerical experiments testing them.

\citet{ZhangShLi17}, in work contemporaneous to the initial submission of
the work~\cite{CarmonDu19}, take an important step towards sharp analysis of
practical methods, showing a rate of convergence of the form
$\sqrt{\kappa}\log\frac{1}{\eps}$ for Krylov subspace solutions to the
trust-region problem.  Based on these bounds, the authors propose novel
stopping criteria for subproblem solutions in the trust-region optimization
method, showing good empirical results.  We complete the picture, showing
for Krylov subspace methods an $\eps^{-1/2}\log{d}$ convergence
guarantee that holds in the hard case where $\kappa=\infty$ and
extending the analysis to cubic regularization, for which we also give a
comprehensive analysis of gradient descent. 

Much of the literature on the problems~\eqref{eqn:problem-tr} 
and~\eqref{eqn:problem-cubic} considers them in the context of 
model-based optimization algorithms. \citet{ConnGoTo00} provide
a detailed account of trust-region methods. Cubic 
regularization of Newton's method was first proposed 
by~\citet{Griewank81} and subsequently independently rediscovered 
by~\citet{NesterovPo06} and~\citet{WeiserDeEr07}.
\citet{NesterovPo06} prove that for $g$ with Lipschitz Hessian and exact
subproblem solutions~\eqref{eqn:generic-problem-iteration},
cubic-regularized Newton's method
finds $\epsilon$ second-order stationary points in
order of $\epsilon^{-3/2}$ iterations; this is the first non-asymptotic
convergence rate to second-order stationarity as well as the first
improvement on gradient descent's $\epsilon^{-2}$ rate of convergence to 
first-order
stationarity.

\citet{CartisGoTo11b} give sufficient conditions on the accuracy of 
approximate subproblem solutions under which the $\epsilon^{-3/2}$ 
bound on subproblem number persists, though they leave open how to 
meet these conditions with a scalable subproblem solver. We provide 
alternative sufficient conditions, which we satisfy using the Krylov 
subspace method with roughly $\epsilon^{-1/4}$ Hessian-vector products 
per subproblem. Our approach is less practical than that of Cartis et al.\ 
(assuming knowledge of problem parameters rather than adapting to them 
as in~\cite{CartisGoTo11,CartisGoTo11b,ConnGoTo00}), but it 
nevertheless allows us to demonstrate that order roughly 
$\epsilon^{-7/4}$ gradient and Hessian-vector product evaluations are 
sufficient to guarantee $\epsilon$ second-order stationarity.

\subsection{Concurrent and subsequent work}
\label{sec:concurrent-subsequent}
The papers~\cite{CarmonDu19,CarmonDu18} forming the basis of this 
paper are part of an active body of research seeking better 
understanding of and efficient methods for nonconvex optimization.
We highlight three lines of work that closely interact with the 
contributions of our paper.

\paragraph{Improved rates for finding stationary points}
Approximate stationarity (a point $x$ satisfying $\norm{\grad g(x)} \le
\epsilon$) serves as a proxy for local optimality, and complexity estimates
to achieve it serve as a yardstick for comparing different methods. Gradient
descent finds an $\epsilon$-stationary point of functions with Lipschitz
gradient in $\epsilon^{-2}$ gradient evaluations~\cite{Nesterov04}, and this
is unimprovable without further
assumptions~\cite{CarmonDuHiSi19}.  Yet additional structure allows
improvement: if the Hessian $\grad^2 g(x)$ is Lipschitz continuous, several
recent first-order methods achieve convergence to $\epsilon$-stationarity in
roughly $\epsilon^{-7/4} \log d$ steps.  \citet{AgarwalAlBuHaMa17} propose a
variant of the cubic-regularized Newton method with an elaborate subproblem
solver based on reduction to eigenvalue computation. In independent work
with collaborators Hinder and Sidford~\cite{CarmonDuHiSi18}, we give a
different algorithm based on Nesterov's accelerated gradient descent and the
Lanczos method that attains the same first-order complexity. In subsequent
work~\cite{CarmonDuHiSi17} we propose a simpler technique using Nesterov
acceleration directly; this method is capable of exploiting even third-order
Lipschitz continuity, under which its rate of convergence improves to
$\wt{O}(\epsilon^{-5/3})$.
 Royer et
al.~\cite{RoyerWr18,RoyerONWr19} show that a careful
implementation of established techniques (line search and Newton CG)
also attains the improved
complexity~$\wt{O}(\epsilon^{-7/4})$. In this paper, we further strengthen
this point of view by showing that cubic regularization with a classical
Krylov subspace method attains this improved complexity as well.

\paragraph{Large-scale second-order methods}
In many large-scale problems---particularly those 
arising in machine learning---noisy evaluation of the objective and its 
derivatives is far cheaper than exact evaluation, motivating the use of 
stochastic gradient 
methods~\cite{BottouCuNo18}. Several 
works attempt to extend second-order model-based optimization techniques
to the stochastic setting, with some promising empirical 
findings~\cite{KohlerLu17,XuRoMa19,MartensGr15}. Adopting a theoretical 
perspective~\citet{TripuraneniStJiReJo18}
analyze a sub-sampled cubic-regularized Newton method, solving
sub-problems using our gradient descent scheme~\cite{CarmonDu19};
the noise inherent in stochastic sampling means that replacing gradient 
descent with the Krylov subspace method does not improve the leading 
terms in their complexity bound.

\paragraph{Structured nonconvex problems and their analysis}
Global minimization of nonconvex functions is generally
intractable~\cite{NemirovskiYu83,MurtyKa87}.  Yet a growing body of work  
identifies families of practically important structured problems that admit 
efficient solutions. There are (to us) two broad approaches. The first is a 
``classical'' decoupling approach, which shows that certain local solutions 
to the problem (e.g., second-order stationary points) are in fact global, and 
then argues that standard algorithms find these local solutions; examples 
include
matrix completion~\cite{GeLeMa16}, phase retrieval~\cite{SunQuWr18},
more general low-rank problems~\cite{GeJiZh17}, and
linear dynamical system identification~\cite{HardtMaRe18}.
The second we term the ``dynamics-based'' approach,
where the trajectory of an optimization method is central and one
proves it converges to a global minimum. Here,
the simplicity of gradient descent makes it essentially the only
feasibly analyzed algorithm, and examples include
guarantees for two-layer neural networks~\cite{LiYu17} and matrix
completion, phase retrieval, and blind deconvolution~\cite{MaWaChCh19}.
We view our analysis of gradient descent as a potential prototype for
the latter trajectory-based approaches, providing
a particularly simple example of the mechanism that keeps gradient descent
away from the bad local minimum and allows it to quickly bypass saddle
points. 
In contrast, our analysis of Krylov methods falls firmly
in the former approach. 

\subsection{Paper organization}
In Section~\ref{sec:prelim} we define our notation and review basic 
structural properties of the problems we study. Section~\ref{sec:gd} gives 
our results for gradient descent and Section~\ref{sec:krylov} gives our 
results for Krylov subspace methods. We revisit both methods in 
Section~\ref{sec:the-hard-case} when we tackle the hard case via 
randomization.
To illustrate our results, 
we accompany 
Sections~\ref{sec:gd}--\ref{sec:the-hard-case}  with numerical 
experiments. Then, in 
Section~\ref{sec:majorization}, we apply our randomized Krylov subspace 
solver within a cubic regularized model-based optimization method for 
general nonlinear functions, and establish a rate of convergence to 
approximate second order stationary points. Section~\ref{sec:discussion} 
concludes our paper by situating our approaches in the context of
convergence guarantees for convex optimization and eigenvector problems.

\section{Preliminaries and solution structure}\label{sec:prelim}
\paragraph{Notation}
Recall the function $\f(x) = \half x^T A x +
b^T x$, where $b\in\R^{d}$ and $A\in \R^{d \times d}$ is a symmetric
(possibly indefinite) matrix.
The eigenvalues of the matrix $A$ are
$\lambda\supind{1} (A) \le 
\cdots\leq\lambda\supind{d}(A)$,
where the $\lambda\supind{i}(A)$ may be negative, and have
associated eigenvectors
$v_1, \ldots, v_d$, so that
$A = \sum_{i = 1}^d \lambda\supind{i}(A) v_i v_i^T$.
Importantly, throughout the paper we
work in the eigenbasis of $A$, and for any vector $w \in \R^d$ we let
\begin{equation}\label{eq:superscript-convention}
  w\supind{i} = v_i^T w~\mbox{denote the $i$th
    coordinate of $w$ in the eigenbasis of $A$.} 
\end{equation}
We also let $\lmin$ and $\lmax$ be the minimum and maximum eigenvalues
of $A$, and $\vmin$ and $\vmax$ be the corresponding (unit) eigenvectors.
We let $\opnorm{\cdot}$ be the $\ell_2$-operator norm, so $\opnorm{A} =
\max_{\norm{u} \le 1} \norm{A u} = \max_i |\lambda\supind{i}(A)|$, and define
\begin{equation*}
  \beta \defeq \opnorm{A}
  = \max\{-\lmin, \lmax\},
\end{equation*}
Our results frequently depend on the quantity $\beta$, but they hold for any
estimate satisfying $\beta \ge \opnorm{A}$.  We say a function $g$ is
$L$-\emph{smooth} on a set $X$ if $\norm{\grad g(x) - \grad g(y)} \le
L \norm{x - y}$ for all $x, y \in X$.  We denote the positive
part of $s\in\R$ by $\hinge{s} = \max\{s,0\}$.

\subsection{Characterization of solutions}
We let $\str$ be a solution of the trust region
problem~\eqref{eqn:problem-tr}, while $\scu$ denotes a solution of the
cubic-regularized quadratic problem in~\eqref{eqn:problem-cubic}.  The
structure of the problems allows relatively transparent characterizations of
their solutions~\cite[e.g.][]{Martinez94,ConnGoTo00,NesterovPo06}:
\begin{proposition}[%
    \cite{ConnGoTo00}, Ch.~7 and
    \cite{CartisGoTo11}, Theorem 3.1]
  \label{prop:characterization}
  A vector $\str$ solves the
  trust-region problem~\eqref{eqn:problem-tr} if
  and only if there exists $\ltr$ such that
  \begin{equation}
    \label{eqn:trust-region-char}
    (A +\ltr) \str + b = 0,
    ~~ \ltr \ge \hinge{-\lmin}, ~~
    \mbox{and} ~~
    \ltr(R - \norms{\str}) = 0,
  \end{equation}
  and $\str$ is unique if $\ltr > \hinge{-\lmin}$.
  A vector $\scu$ solves the
  cubic-regularized problem~\eqref{eqn:problem-cubic}
  if and only if
  \begin{equation}
    (A + \rho \norm{\scu} I) \scu + b = 0
    ~~ \mbox{and} ~~ \rho \norm{\scu} \ge \hinge{-\lmin},
    \label{eqn:cubic-optimality}
  \end{equation}
  and $\scu$ is unique if $\rho \norm{\scu} > \hinge{-\lmin}$.
\end{proposition}
\noindent
In other words, $x$ solves the corresponding problem
if and only if it is stationary and satisfies
$\lambda(x) \ge \hinge{-\lmin}$, where 
$\lambda(x)$ is the Lagrange multiplier for the constraint $\norm{x}\le R$ 
for~\eqref{eqn:problem-tr} and
$\lambda(x)=\rho\norm{x}$ for~\eqref{eqn:problem-cubic};
$x$ is unique if $\lambda(x) > \hinge{-\lmin}$.

 For matrices $A$ with
distinct eigenvalues, each problem may have a single suboptimal local
minimizer, a single local maximizer, and up to $2(d-1)$ saddle points
(cf.~\cite[Section 3]{Griewank81} or \cite[Thm.~3.1]{Martinez94}); see
Figure~\ref{fig:simple-gradients} for an example with $d=2$. The next 
result characterizes the solutions to both the trust-region and 
cubic-regularized problems in terms of
stationarity and the direction $b$ in the space spanned by the 
eigenvector $\vmin$ corresponding to $\lmin$. It forms the basis for our 
analysis of
gradient descent.
\begin{proposition}
  \label{proposition:stationary-eigen}
  Let $b\supind{1} \neq 0$. Then $\str$ and $\scu$ are the unique stationary
  points (respectively) of the objectives~\eqref{eqn:problem-tr}
  and~\eqref{eqn:problem-cubic} satisfying
  \begin{equation*}
    b\supind{1} x\supind{1} \le 0
    ~~~ \mbox{and so necessarily} ~~~
    b\supind{1} x\supind{1} < 0.
  \end{equation*}
\end{proposition}
\begin{proof}
  Let $x$ be a stationary point of either problem and note that it  
  satisfies $Ax+b+\lambda x=0$ for some $\lambda\ge 0$ by 
  Proposition~\ref{prop:characterization}; 
  for~\eqref{eqn:problem-cubic} we have $\lambda=\rho\norm{x}$ and 
  for~\eqref{eqn:problem-tr} $\lambda=\ltr$ is the Lagrange multiplier for 
  the constraint $\norm{x}\le R$. Focusing on the first (eigen)coordinate, we 
  have
    \begin{equation*}
      0 = \vmin^T ((A + \lambda I) x + b)
      = (\lmin + \lambda) x\supind{1} + b\supind{1}.
    \end{equation*}
   Therefore, $b^{(1)}\neq 0$ implies both $x^{(1)} \neq 0$ and
  $\lambda + \lmin \neq 0$. This strengthens the inequality
  $b^{(1)}x^{(1)} \le 0$ to $b^{(1)}x^{(1)} <
  0$. Hence
  $\lambda + \lmin = -b^{(1)}x^{(1)}/[x^{(1)}]^2 > 0$ and consequently 
  $\lambda > (-\lmin)_+$. By Proposition~\ref{prop:characterization} and 
  the above characterization of $\lambda$, the point $x$ is
  the unique global minimum.
\end{proof}

\begin{figure}
  \begin{center}
    \siam{\includegraphics[width=.5\columnwidth]{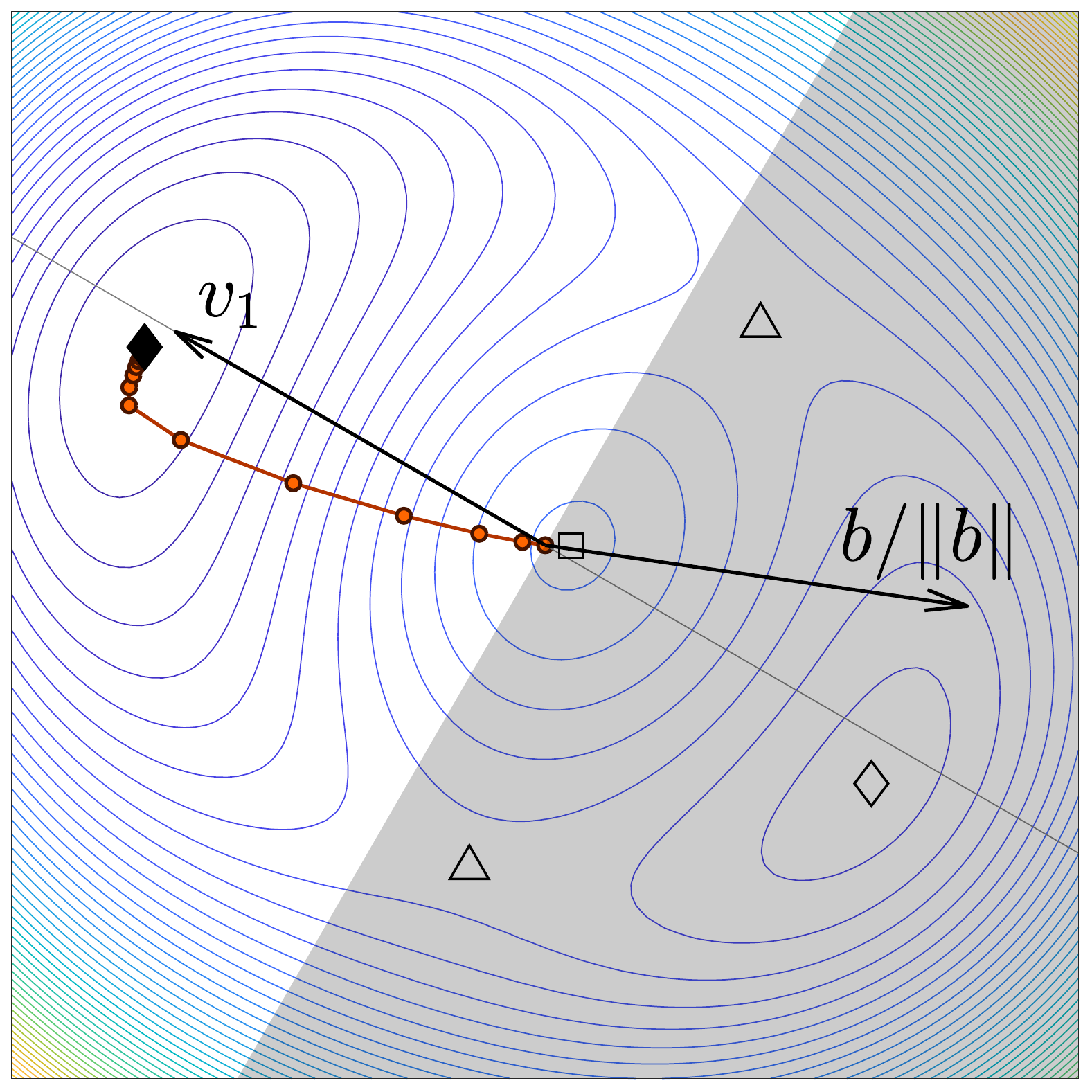}}%
    \arxiv{\includegraphics[width=.4\columnwidth]{figures/figureContourWithHPandGD.pdf}}
    \caption{\label{fig:simple-gradients} Contour plot of a two-dimensional
      instance of \eqref{eqn:problem-cubic}, featuring a local maximum
      ($\square$), saddle points ($\triangle$), and local minima
      ($\lozenge$). The line of circles indicates the path of gradient
      descent initialized at the origin, and the gray area is the half-plane
      $(v_1^T b)(v^T x) = b\supind{1}x\supind{1} > 0$. Note that the global
      minimum is the only critical point outside this half-plane
      (Proposition~\ref{prop:characterization}). The gradient descent
      iterates have increasing norm (Lemma~\ref{lemma:monotone}), lie
      outside the half-plane (Lemma~\ref{lemma:signs}), and converge to
      $\scu$ (Proposition~\ref{prop:converge}).
    }
  \end{center}
\end{figure}

\subsection{Bounds on the solutions}
\label{sec:cubic-solution-bounds}
The magnitude of the solution of~\eqref{eqn:problem-cubic} and its 
optimal value are important in our coming analysis (trivially
$\norm{\str} \le R$), and we
therefore provide bounds for these quantities. First, we define
\begin{equation*}
  \As \defeq A + \rho \norm{\scu} I.
\end{equation*}
By Proposition~\ref{prop:characterization}, $\scu$ solves
problem~\eqref{eqn:problem-cubic} if and
only if it is stationary and $\As \succeq 0$. 
Let $f(x) = \fcu(x)=\half x^T A x  + b^T x + \frac{\rho}{3}\norm{x}^3$
for short.
Then algebraic manipulation shows that
\begin{equation}
  \label{eq:fx-s-expression}
  f\left(x\right)=f\left(\scu\right)+\frac{1}{2}(x-\scu)^T\As(x-\scu) 
  +\frac{\rho}{6}\left(\norm{\scu} -\norm{x} \right)^{2}
  \left(\norm{\scu} +2\norm{x} \right),
\end{equation}
which makes it clear that $\scu$ is indeed the global minimum,
as both of the $x$-dependent terms are non-negative and minimized at
$x=\scu$, and the minimum is unique whenever $\rho \norm{\scu} > -\lmin$,
because $\As \succ 0$ in this case.

To bound the norm  of $\scu$,   observe that $\norm{b} = \norms{\As 
\scu} \ge (\lmin +\rs)
\norm{\scu}$. Solving for $\norm{\scu}$ gives the upper bound
\begin{equation}
  \label{eq:R-upper-def}
  \norm{\scu} \leq
  \frac{-\lmin}{2\rho} + \sqrt{\left(\frac{\lmin}{2\rho}\right)^{2}
    +\frac{\norm{b} }{\rho}} \leq \frac{\beta}{2\rho}
  +\sqrt{\left(\frac{\beta}{2\rho}\right)^{2}+\frac{\norm{b} }{\rho}} \eqdef
  \Rupper
\end{equation}
where we recall that $\beta = \opnorm{A} \ge |\lmin|$.  An analogous lower
bound is available:
\begin{equation}
  \label{eq:Rc-def}
  \norm{\scu} \geq \Rcauchy
  \defeq
  \frac{-b^{\T}Ab}{2\rho\norm{b}^{2}}
  + \sqrt{\left(\frac{b^{\T}Ab}{2\rho\norm{b} ^{2}}\right)^{2}+\frac{\norm{b} }{\rho}}
  \geq
\Rupper - \frac{\beta}{\rho}.
\end{equation}
The quantity $\Rcauchy$ is the \emph{Cauchy radius}~\cite{ConnGoTo00}---the 
magnitude of the (global) minimizer of $f$ in the span of $b$:
$\Rcauchy = \argmin_{\zeta\in\R} f(-\zeta b/\|b\|)$.  To see the claimed lower
bound~\eqref{eq:Rc-def}, set $\xcauchy = -\Rcauchy b/\norm{b}$ (the \emph{Cauchy 
point}) and note by a calculation that
$f(\xcauchy) = -(1/2)\|b\|\Rcauchy -(\rho/6)\Rcauchy^3$. Therefore,
  $0 \le f(\xcauchy) - f(\scu) \le \frac{1}{2}\|b\|(\norm{\scu}-\Rcauchy) + \frac{1}{6}\rho(\norm{\scu}^3 - 
  \Rcauchy^3)$,
which implies $\norm{\scu} \ge \Rcauchy$. 

\newcommand{\proj}{\pi}

\section{Gradient descent for nonconvex quadratics}
\label{sec:gd}
For the problem of minimizing $f(x)$ subject to constraints that
$x \in X$, the projected gradient method begins at 
$x_0 \in \R^d$ and for a fixed stepsize $\eta > 0$ iterates
\begin{equation}
  \label{eqn:pgd}
  x_{t + 1}
  = \argmin_{x \in X} \Big\{f(x_t) + \grad f(x_t)^T (x - x_t)
  + \frac{1}{2 \eta} \norm{x - x_t}^2 \Big\}.
\end{equation}
For the trust-region problem~\eqref{eqn:problem-tr}, where $X=\{x\mid 
\norm{x}\le R\}$, this is the iteration~\eqref{eq:grad-iter-tr},
while for the cubic-regularized problem~\eqref{eqn:problem-cubic}, where 
$X=\R^d$,
this is the iteration~\eqref{eq:grad-iter}.
We will 
show that the iteration~\eqref{eqn:pgd} converges to global
minimizers for both problems~\eqref{eqn:problem-tr}
and~\eqref{eqn:problem-cubic}, providing an asymptotic
guarantee in Sec.~\ref{sec:asymptotic} and an 
explicit rate guarantee in Sec.~\ref{sec:non-asymptotic} for the
iteration~\eqref{eq:grad-iter} for problem~\eqref{eqn:problem-cubic}.

Recalling the
definitions~\eqref{eq:R-upper-def} and~\eqref{eq:Rc-def} of $\Rupper$ and
$\Rcauchy$ as well as $\opnorm{A} = \beta$, throughout our analysis we make
the following assumptions.
\begin{assumption}
  \label{assu:init}
  The initialization $x_0$ of~\eqref{eqn:pgd}
  satisfies $x_{0} = -r\frac{b}{\norm{b}}$
  where $r \in [0, R]$ for problem~\eqref{eqn:problem-tr} or
  $r \in [0, \Rcauchy]$ for problem~\eqref{eqn:problem-cubic}.
\end{assumption}
\begin{assumption}
  \label{assu:step-size}
  The step size $\eta$ satisfies $0 < \eta \le \frac{1}{\beta}$ for
  problem~\eqref{eqn:problem-tr} or $0<\eta\leq \frac{1}{4(\beta+\rho
    \Rupper)}$ for problem~\eqref{eqn:problem-cubic}.
\end{assumption}
\noindent
To select a step size $\eta$ satisfying Assumption~\ref{assu:step-size},
only a rough upper bound on $\opnorm{A}$ is necessary. One way to obtain
such a bound is to apply a few power iterations on $A$.

\subsection{Asymptotic convergence guarantees and iterate structure}
\label{sec:asymptotic}

We begin our analysis via a few properties of the gradient descent
trajectory.  First, we establish that $\|x_t\|$ is monotonic and bounded for
the iteration~\eqref{eq:grad-iter} of gradient descent from 
problem~\eqref{eqn:problem-cubic}.

\begin{lemma}
  \label{lemma:monotone}
  Let Assumptions~\ref{assu:init} and \ref{assu:step-size} hold.
  Then the iterates~\eqref{eq:grad-iter}
  satisfy that $x_t^{\T}\grad \fcu(x_t) \le 0$, the norms $\norm{x_t}$ are
  non-decreasing with $\norm{x_t} \le \norm{\scu}$, and
  $\fcu$ is $(\beta + 2\rho \norm{\scu})$-smooth on the ball $\{x \mid \norm{x}
  \le \norm{\scu}\}$. %
\end{lemma}
\noindent
This lemma (proved in Appendix~\ref{app:prel}) is the key to our analysis
of the cubic-regularized problem.
The iterate structure is convenient for both problems, as the
next lemma shows that 
$x_t$ and $b$ have opposite signs at all coordinates
in the eigenbasis of $A$.
\begin{lemma}
  \label{lemma:signs}
  Let Assumptions~\ref{assu:init} and~\ref{assu:step-size} hold.
  Let the iterates $x_t$ be generated by the gradient descent
  iteration~\eqref{eqn:pgd} for either problem~\eqref{eqn:problem-tr}
  or~\eqref{eqn:problem-cubic}, and let $x\subopt$ be a solution
  to the given problem.
  Then for all $t\geq0$ and $i\in [d]$,
  \begin{equation*}
    b^{(i)} x\subopt^{(i)}\leq0, ~~
    x_{t}^{(i)}b^{(i)}\leq0, ~~
    \text{and} ~ x_{t}^{(i)} x\subopt^{(i)} \geq 0.
  \end{equation*}
  Consequently, $x\subopt^T b \le 0$, and
  $x_t^{\T}b \le 0$ and $x_{t}^{\T}x\subopt\geq0$ for all $t$.
\end{lemma}
\begin{proof}
  We first show that $b\supind{i} x\subopt\supind{i} \le 0$ for both
  problems. Letting $\ltr \ge \hinge{-\lmin}$ be the dual
  parameter~\eqref{eqn:trust-region-char}, we make the context-dependent
  definitions $\As = A + \ltr I$ or $\As = A + \rho \norm{\scu}
  I$. By Proposition~\ref{prop:characterization}, we have $\As x\subopt =
  -b$ and $\As \succeq 0$. Recalling the
  eigenbasis notation~\eqref{eq:superscript-convention}, we evidently 
  have $\lambda\supind{i}(\As)
  x\subopt\supind{i} = -b\supind{i}$, and 
  therefore $b\supind{i} x\subopt\supind{i} = -\lambda\supind{i}(\As) 
  [x\subopt\supind{i}]^2 \le 0$.

  Now, we consider the iterates of gradient descent. The initialization
  in Assumption~\ref{assu:init} guarantees $x_0\supind{i} b\supind{i}
  \le 0$ for either method, forming the base case of our induction. For
  the trust-region
  problem, writing the iteration~\eqref{eq:grad-iter-tr} in the eigenbasis 
  gives
  \begin{equation*}
  x_{t+1}\supind{i} b\supind{i} = \alpha_t \left[ 
  (1-\eta \lambda\supind{i})x_{t+1}\supind{i} b\supind{i} - \eta 
  [b\supind{i}]^2 
  \right], ~~~\alpha_t = \min\left\{1,\tfrac{R}{\norm{(I-\eta 
  A)x_t 
  -\eta b}}\right\}.%
  \end{equation*}
  As $1-\eta \lambda\supind{i} \ge 0$ by 
  Assumption~\ref{assu:step-size} and 
  $\alpha_t > 0$, we have that $x_t\supind{i} b\supind{i} \le 0$ implies 
  $x_{t+1}\supind{i} b\supind{i} \le 0$, completing the induction. Writing
  the cubic-regularized iteration~\eqref{eq:grad-iter}
  similarly gives
  \begin{equation*}
    x_{t+1}\supind{i} b\supind{i} =
    \left(1-\eta\lambda^{(i)}(A)-\eta\rho\| x_t\| \right)
    x_{t}^{(i)}b\supind{i}-[b\supind{i}]^2.
  \end{equation*}
  Assumption~\ref{assu:step-size} and Lemma \ref{lemma:monotone}
  imply $1-\eta\lambda^{(i)}(A)-\eta\rho\| x_{t-1}\|
  \ge 1 - \eta(\beta + \rho \Rupper) > 0$ for all $t, i$.
  Therefore, $x_{t+1}^{(i)}b^{(i)} \leq 0$ by induction.

  The remaining claims of the lemma are immediate from the preceding.
\end{proof}

\newcommand{\proxstep}{\mathsf{T}_\eta}

Lemmas~\ref{lemma:monotone}, \ref{lemma:signs}, and
Proposition~\ref{proposition:stationary-eigen} lead to the
following guarantee.
\begin{proposition}
  \label{prop:converge}
  Let Assumptions~\ref{assu:init} and~\ref{assu:step-size} hold, and assume
  that $b\supind{1} \neq 0$. Let $x_t$ follow the gradient
  iteration~\eqref{eqn:pgd} for either problem~\eqref{eqn:problem-tr}
  or~\eqref{eqn:problem-cubic}, and $x\subopt$ solve the corresponding
  problem.  Then $x_t \to x\subopt$ and the objective is monotone 
  decreasing.
\end{proposition}
\begin{proof}
  We recall a few standard results~\cite[\S 2.2.3]{Nesterov04}. For a
  differentiable $f$, closed convex $X \subset \R^d$, and $x \in \R^d$,
  define $\proxstep(x) = \argmin_{y \in X} \{\grad f(x)^T (y - x) +
  \frac{1}{2\eta} \norm{y - x}^2\}$ and the gradient mapping
  $\gradmap_\eta(x) = \frac{1}{\eta} (x - \proxstep(x))$, where
  $\gradmap_\eta(x) = \grad f(x)$ if $X = \R^d$, so that gradient descent
  iterates $x_{t + 1} = \proxstep(x_t)$. The first-order optimality
  conditions for convex optimization give that $(\grad f(x) -
  \gradmap_\eta(x))^T(y - \proxstep(x)) \ge 0$ for all $y \in
  X$, and substituting $y = x$ in this inequality, for any $L$-smooth $f$ we
  obtain
  \begin{equation*}
    f(\proxstep(x)) 
    \le f(x) - \eta \left(1 - \frac{L \eta}{2}\right)
    \norm{\gradmap_\eta(x)}^2.
  \end{equation*}
  In the case of problem~\eqref{eqn:problem-tr}, we have $f(x) = \f(x)$ and
  $L = \beta = \opnorm{A}$, while for~\eqref{eqn:problem-cubic}, $f(x) =
  \f(x) + \frac{\rho}{3} \norm{x}^3$, which is $L = \beta + 2 \rho
  \Rupper$-smooth over the ball containing the iterates $x_t$ by
  Lemma~\ref{lemma:monotone}.  As $\eta \le \frac{1}{L}$ for
  either problem,  we have the decrease $f(x_{t
    + 1}) \le f(x_t) - \frac{\eta}{2} \norm{\gradmap_\eta(x_t)}^2$,
  and
  \begin{equation*}
    \frac{\eta}{2} \sum_{\tau=0}^{t-1} \|\gradmap_\eta(x_\tau)\|^2
    \le f(x_0)-f(x_t) \le f(x_0)-f(x\subopt).
  \end{equation*}

  Let $\hat{x}$ be a limit point of the sequence $x_t$, which must satisfy
  $b\supind{1} \hat{x}\supind{1} \le 0$ by Lemma~\ref{lemma:signs}. The
  continuity of $x \mapsto \gradmap_\eta(x)$ means that
  $\gradmap_\eta(\hat{x}) = 0$, and consequently $\hat{x}$ is stationary. 
  As $b\supind{1}\ne0$ and $\hat{x}$ is a stationary point with 
  $b\supind{1} \hat{x}\supind{1} \le 0$, we have that $\hat{x}=x\subopt$ 
  by Proposition~\ref{proposition:stationary-eigen}.
\end{proof}

\subsection{Convergence rate guarantees for the cubic-regularized problem}
\label{sec:non-asymptotic}
Proposition~\ref{prop:converge} guarantees that gradient descent 
converges for
 both problems~\eqref{eqn:problem-tr} and~\eqref{eqn:problem-cubic}
whenever $b\supind{1} \neq 0$.  We now present stronger non-asymptotic
guarantees for the cubic problem, deferring the
treatment when $b\supind{1} = 0$ (the so-called ``hard
case''~\cite{ConnGoTo00, CartisGoTo11}) to Section~\ref{sec:the-hard-case}.
(Recall our convention~\eqref{eq:superscript-convention}, that parenthesized
superscripts denote components in the eigenbasis of $A$, and the 
additional notation $\lmin = \lambda\supind{1}(A)$ and $\beta = 
\opnorm{A}$.)  We 
have the following convergence guarantee.

\begin{theorem}
  \label{theorem:gradient-descent-cubic}
  Let Assumptions \ref{assu:init} and \ref{assu:step-size} hold, $b\supind{1} 
  \neq 0$,
  and $\eps > 0$.
  Define
  \begin{equation*}
    \tgrow = 
    6\log\left(1+\frac{\hinge{-\lmin}^2}{4\rho|b\supind{1}|}\right)
    ~ \mbox{and} ~
    \tconv(\eps)=6\log\left(\frac{ (\beta + 2\rs)\norm{\scu}^2 
   }{\eps}\right).
  \end{equation*}
  Then the gradient descent iterates~\eqref{eq:grad-iter} satisfy $\fcu(x_t) 
  \le \fcu(\scu) + \eps$ for all
  \begin{equation*}
    t \geq \frac{\tgrow + \tconv\left(\eps\right )}{\eta} 
    \min\left\{\frac{1}{\rs + \lmin},
    \frac{10 \norm{\scu}^2}{\eps}
    \right\}.
  \end{equation*}
\end{theorem}
\noindent
\newcommand{\rthresh}{r_{\mathrm{thresh}}}%
Deferring the full proof of the theorem to
Appendix~\ref{sec:proof-gd-cubic}, we provide a brief sketch here. We 
first show that there is a basin of attraction where iterates with norm
above roughly $-\lmin/\rho$ contract towards the global solution:
\begin{equation}
  \label{eq:contraction}
  \norm{x_{t+1}-\scu}^2 \le 
  \left[1-\frac{\eta}{6}(\rs + \lmin)\right] \norm{x_{t}-\scu}^2
\end{equation}
for all $t$ satisfying
\begin{equation*}
  \norm{x_t}\ge \rthresh\defeq \frac{-\lmin-\frac{1}{3}(\rs+\lmin)}{\rho}.
\end{equation*}
As $\norm{x_t}$ is monotonic by Lemma~\ref{lemma:monotone}, the 
contraction~\eqref{eq:contraction} guarantees that once $\norm{x_{T_1}} 
\ge \rthresh$, then $\fcu(x_t) - \fcu(\scu) \le \eps$ for all $t \ge T_1 + 
\frac{\tconv(\eps)}{\eta} (\rs + \lmin)^{-1}$. 
It remains to establish that the iterates escape the ball of radius $\rthresh$ 
quickly, which is nontrivial only in the nonconvex setting where $\lmin < 
0$. To this end, we prove the iterate norm grows exponentially, showing 
that if $\norm{x_t} \le \rthresh$ then
\begin{equation*}
  \norm{x_{t+1}}\ge |x\supind{1}_{t+1}| \ge 
  \left[1+\frac{\eta}{6}(\rs + \lmin)\right]|x\supind{1}_{t}| + \eta 
  |b\supind{1}|.
\end{equation*}
Consequently, $\norm{x_{T_1}}\ge \rthresh$ holds for $T_1\ge 
\frac{\tgrow}{\eta}(\rs+\lmin)^{-1}$, establishing the linear convergence 
rate in Theorem~\ref{theorem:gradient-descent-cubic}: the total number of 
iterations to $\eps$-optimality is $O(\kappa_\eta\log\frac{1}{\eps})$, 
where $\kappa_\eta \defeq \frac{1}{\eta(\lmin+\rs)}$ has the same order 
as the problem condition number $\kappa = \frac{\lmax + \rs}{\lmin + 
\rs}$ when $\eta$ is the maximum step size 
Assumption~\ref{assu:step-size} allows and $\lmax\ge 0$.

Theorem~\ref{theorem:gradient-descent-cubic} also provides an 
 $O(\frac{1}{\eps}\log\frac{1}{\eps})$ sublinear convergence rate, which is 
stronger than the linear convergence result when $\lmin + \rs = O(\eps)$. 
To prove it, we argue that geometric contraction  to the optimum 
still occurs in the subspace of eigenvectors corresponding to eigenvalues 
greater than $\lmin + O(1)\eps/\norm{\scu}^2$.  In the 
complementary subspace (of eigenvalues close to $\lmin$), we argue that 
the objective $f$ is very smooth outside a ball of radius $\rthresh$, and 
consequently that errors in that subspace do not significantly affect the 
objective value.

\subsection{Numerical illustration}

\begin{figure}
  \begin{center}
    \siam{\includegraphics[
	width=0.9\columnwidth]{%
	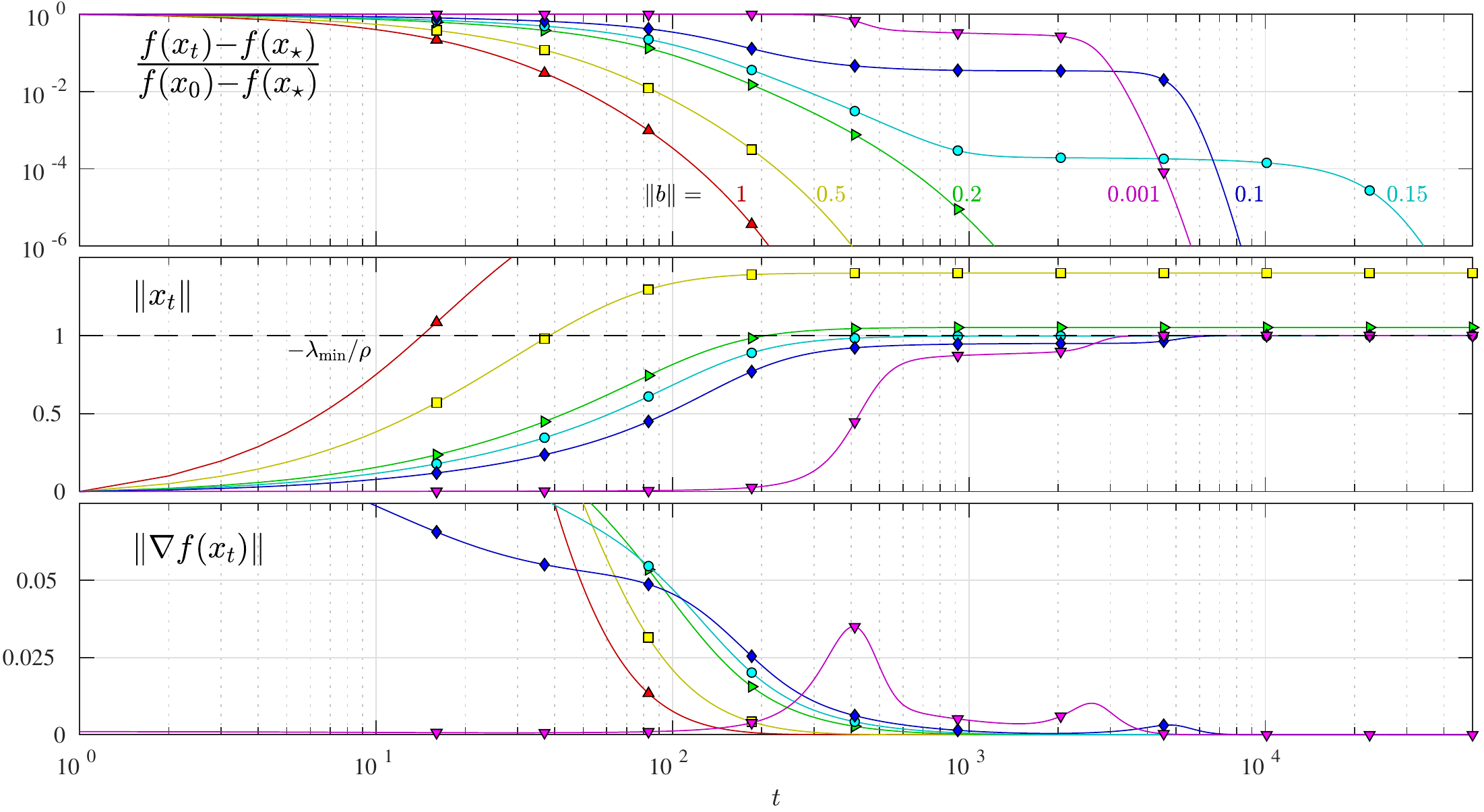}}
    \arxiv{\includegraphics[
	width=0.8\columnwidth]{%
	figures/figureTrajectoryDifferentB_gamma0,2_gap0,02.pdf}}
    \caption{\label{fig:trajB} Trajectories of gradient descent with
      $\lambda^{(1)}(A) = -0.2$ and
      $\lambda^{(2)}(A),...,\lambda^{(d)}(A)$ equally spaced between 
      $-0.18$
      and $\beta=1$, and different vectors $b$ proportional to $[0.01, 1, 
	1, 
	1, \ldots]$ in the eigenbasis of $A$.
      The rest of the parameters
      are $d=10^3$, $\eta = 0.1$, $\rho = 0.2$ and $x_0=0$.} 
  \end{center}
\end{figure}

We examine 
the behavior of gradient descent on a few problem instances, looking at 
convergence behavior as we vary the vector $b$ by scaling its norm 
$\norm{b}$. The selected norm values 
$\norm{b}\in\{1, 0.5, 0.2, 0.15, 0.1, 0.001\}$ correspond to condition 
numbers $(\beta+\rs)/(\lmin + \rs)\in \{7.6, 16, 120, 
5.5\cdot10^{3}, 2.9\cdot10^{4}, 3.8\cdot10^{6}\}$; the problem 
conditioning becomes worse as $\norm{b}$ decreases.   
Figure~\ref{fig:trajB} 
summarizes our results and describes the settings of the other 
parameters in the experiment.

The plots show two behaviors of gradient descent.  The problem is 
well-conditioned when $\|b\|\ge 0.2$, and in these cases gradient descent 
behaves as though the problem were strongly convex, with $x_t$ 
converging 
linearly to $\scu$. For $\|b\| \le 0.15$ the problem 
becomes ill-conditioned and gradient descent stalls around saddle points. 
Indeed, 
the third plot of Figure~\ref{fig:trajB} shows that for the ill-conditioned 
problems, we have $\norm{\nabla f(x_t)}$ increasing over some iterations, 
which does not occur in convex quadratic 
problems.  The length of the stall does not depend only on the condition 
number; for $\norm{b} = 10^{-3}$ the stall is shorter than for 
$\norm{b} \in \{0.1, 0.15\}$. Instead, it appears to depend on the norm of 
the 
saddle point causing it, which we observe from the value of 
$\norm{x_t}$ at the time of the stall; we see that the closer the norm is to 
$-\lmin/\rho$, the longer the stall takes. This is explained by observing 
that $\hess f(x) \succeq (\lmin+\rho\norm{x})I$, so every 
saddle point with norm  $\norm{x} \approx -\lmin/\rho$ must have only 
small negative curvature and is therefore harder to escape (see also 
Lemma~\ref{lem:growth} in the appendix). Fortunately, as we see in 
Fig.~\ref{fig:trajB}, saddle points with large norm have near-optimal 
objective 
value---this is the intuition behind our proof of the sub-linear 
convergence 
rates. 

\section{Krylov subspace methods}\label{sec:krylov}
We now turn to solutions to~\eqref{eqn:problem-tr} and
\eqref{eqn:problem-cubic} constrained to the Krylov
subspaces~\eqref{eqn:krylov-subspace} of order $t$.  Given orthogonal
$Q_t\in\R^{d\times t}$, $Q_t^T Q_t = I$, with columns in
$\Krylov$, the subspace-constrained problems~\eqref{eqn:krylov-iterates}
reduce to $t$-dimensional updates
$y_t = \argmin_{y:\norm{y}\le R}\{\f[Q_t^T A Q_t, Q_t^T b](y)
\}$ and
$y_t = \argmin_y \{\f[Q_t^T A Q_t,Q_t^T b, \rho](y)\}$ with
$x_t = Q_t y_t$.
The Lanczos process allows us to compute an orthogonal basis $Q_t$ such 
that $Q_t^T A Q_t$ is tridiagonal in time dominated by the cost of $t$ 
matrix-vector products. The tridiagonal structure allows fast linear system 
solution, making the reduced instance solvable in time roughly linear in 
$t$; see Appendix~\ref{sec:lanczos} for details.  Consequently, the 
computational cost of a Krylov subspace solution of order $t$ is roughly 
the same as that of $t$ gradient descent steps.

In this section, we develop bounds of the optimality gap of Krylov subspace 
solution. In contrast to our treatment of gradient descent, here we find it
more convenient to obtain guarantees for the trust-region problem, from 
which we obtain analogous guarantees for the cubic-regularized problem 
as an immediate corollary.

\subsection{Convergence guarantees for the trust region problem}
Let
\begin{equation*}
\itertr_t \in \argmin_{x\in\Krylov,~\norm{x}\le R} 
\f(x) = \half x^T A x + b^T x
\end{equation*}
denote the order $t$ Krylov subspace solution to the trust region 
problem~\eqref{eqn:problem-tr}. With the notation of 
Section~\ref{sec:prelim} and Proposition~\ref{prop:characterization} in 
particular, our main result on
convergence in trust region problems follows.
\begin{theorem}
  \label{theorem:tr-krylov}
  For every $t > 0$,
  \begin{equation*}
    \f(\itertr_t) - \f(\str) \le 36\left[\f(0) - \f(\str) \right]
    \exp\left( -4t\sqrt{\frac{\lmin + \ltr}{\lmax +\ltr}}
    \right),
  \end{equation*}
  and
  \begin{equation*}
    \f(\itertr_t) - \f(\str) \le \frac{(\lmax - 
      \lmin)\norm{\str}^2}{(t-\half)^2} 
    \left[ 
      4 + \frac{\I_{\{\lmin < 0\}}}{8}
      \log^2\left(\frac{4\norm{b}^2}{(b\supind{1})^2}\right)
      \right].
  \end{equation*}
\end{theorem}

Theorem~\ref{theorem:tr-krylov} characterizes
linear and sublinear
convergence regimes.
Linear convergence occurs when $t \gtrsim \sqrt{\kappa}$, where $\kappa 
=
\frac{\lmax + \ltr}{\lmin + \ltr} \ge 1$ is the condition number for the
problem, and the error falls beneath $\eps$
in roughly $\sqrt{\kappa}\log{\frac{1}{\eps}}$ Lanczos
iterations. Sublinear convergence occurs when $t \lesssim \sqrt{\kappa}$, 
and
there the error decays polynomially and falls beneath $\eps$ in roughly
${1}/{\sqrt{\eps}}$ iterations. For worst-case problem instances
this characterization is tight to numerical constant
factors~\cite[Sec.~4]{CarmonDu18}.

The guarantees of Theorem~\ref{theorem:tr-krylov} closely resemble the
guarantees for the conjugate gradient
method~\cite{TrefethenBa97}, including them as the special case $R = \infty$
and $\lmin \ge 0$. For convex problems, the radius constraint $\norm{x}\le
R$ always improves the conditioning of the problem, as $\frac{\lmax}{\lmin}
\ge \frac{\lmax+\ltr}{\lmin+\ltr}$; the smaller $R$ is, the better
conditioned the problem becomes; see additional discussion in 
Sec.~\ref{sec:discussion}.
For nonconvex problems, the sublinear rate
features an additional logarithmic term that captures the role of the
eigenvector $\vmin$. The first rate of Theorem~\ref{theorem:tr-krylov}
is similar to
those of \citet[Thm.~4.11]{ZhangShLi17}, though with somewhat more explicit
dependence on $t$.

In the ``hard case,'' which corresponds to $b\supind{1} = 0$ and $\lmin +
\ltr = 0$ (cf.~\cite[Ch.~7]{ConnGoTo00}), both the bounds in
Theorem~\ref{theorem:tr-krylov} become vacuous, and indeed $\itertr_t$ may
not converge to the global minimizer in this case. However, as the sublinear
bound of Theorem~\ref{theorem:tr-krylov} depends only logarithmically on
$b\supind{1}$, it remains valid even extremely close to the hard case. In
Section~\ref{sec:the-hard-case} we describe simple randomization techniques
with convergence guarantees that are valid in the hard case as well.

For convenience of the reader, we provide a sketch of the
proof of Theorem~\ref{theorem:tr-krylov} here,
deferring the full proof to Appendix~\ref{sec:proof-krylov}.
Our analysis rests on two elementary observations. First
Krylov subspaces are invariant to shifts by scaled identity matrices, \ie 
$\mc{K}_t(A,b) = \mc{K}_t(A_\lambda, b)$ for any $A,b,t$ where 
$\lambda\in\R$, and 
\begin{equation*}
  A_\lambda \defeq A + \lambda I. 
\end{equation*}
Second, for every point $x$ and $\lambda\in\R$
\begin{align}
  \label{eq:tr-gamma-pivot-outline}
  \f(x) - \f(\str)  & = \f[A_{\lambda},b](x) - \f[A_{\lambda},b](\str) + 
  \frac{\lambda}{2}(\norm{\str}^2 - \norm{x}^2) 
\end{align}
Our strategy then is to choose $\lambda$ such that $A_\lambda \succeq 0$,
and then use known results to find $y_t \in \Krylov[t][A_\lambda,b] =
\Krylov[t][A,b]$ that rapidly reduces the ``convex error'' term
$\f[A_{\lambda},b](y_t) - \f[A_{\lambda},b](\str)$. 
We then adjust $y_t$ to obtain a feasible 
point $x_t$ such that the ``norm error'' term 
$\frac{\lambda}{2}(\norm{\str}^2 - \norm{x_t}^2)$ is small.
To establish linear convergence, we take $\lambda=\ltr$ and adjust the 
norm of $y_t$ by taking $x_t=(1-\alpha)y_t$ for some small $\alpha$ that 
guarantees $x_t$ is feasible and that the ``norm error'' term is small. To 
establish sublinear convergence we set $\lambda=-\lmin$ and take $x_t = 
y_t + \alpha \cdot z_t$, where $z_t$ is an approximation for $\vmin$ 
within $\Krylov$, and $\alpha$ is chosen to make 
$\norm{x_t}=\norm{\str}$. This means the ``norm error'' vanishes, 
while the ``convex error'' cannot increase too much, as $A_{-\lmin}z_t 
\approx A_{-\lmin}\vmin=0$.

\subsection{Convergence guarantees for the cubic-regularized problem}
\label{sec:krylov-cubic}
Comparing the optimality characterization~\eqref{eqn:cubic-optimality} for
the cubic problem~\eqref{eqn:problem-cubic} to that for the trust-region
problem~\eqref{eqn:trust-region-char}, we see that any instance $(A,b,\rho)$
of cubic regularization has an equivalent trust-region instance $(A,b,R)$,
with $R=\norm{\scu}$ and identical global minimizers. This trust-region
instance has optimal Lagrange multiplier $\ltr = \rs$, and at any
trust-region feasible $x$ (satisfying $\norm{x}\le
R=\norm{\scu}=\norm{\str}$), the cubic-regularization optimality gap is
smaller than its trust-region equivalent,
\begin{equation*}
  \fcu(x)-\fcu(\scu) = \f(x) - \f(\str) 
  + \frac{\rho}{3}\big( \norm{x}^3 - \norms{\str}^3 \big)
  \le \f(x) - \f(\str).
\end{equation*}
Letting $\itercu_t$ denote the minimizer of $\fcu$ in $\Krylov$ and letting
$\itertr_t$ denote the Krylov subspace solution of the equivalent
trust-region problem, we conclude that
\begin{equation*}
  \fcu(\itercu_t) - \fcu(\scu) \le \fcu(\itertr_t) - \fcu(\scu) \le
  \f(\itertr_t) - \f(\str);
\end{equation*}
cubic regularization Krylov subspace solutions always have a 
\emph{smaller optimality gap} than their 
trust-region equivalents. Theorem~\ref{theorem:tr-krylov} thus
gives the following result.
\begin{corollary}\label{cor:cu}
  Let $\fcu\opt = \fcu(\scu)$. For every $t>0$, 
  \begin{equation*}
    \fcu(\itercu_t) - \fcu\opt \le 36\left[\fcu(0) - \fcu\opt \right]
    \exp\left\{
    -4t\sqrt{\frac{\lmin + \rs}{\lmax +\rs}}
    \right\},
  \end{equation*}
  and
  \begin{equation*}
    \fcu(\itercu_t) - \fcu\opt \le \frac{(\lmax - 
      \lmin)\norm{\scu}^2}{(t-\half)^2} 
    \left[ 
      4 + \frac{\I_{\{\lmin < 0\}}}{8}
      \log^2\left(\frac{4\norm{b}^2}{(b\supind{1})^2}\right)
      \right].
  \end{equation*}
\end{corollary}
\begin{proof}
  We look forward to use the
  bound~\eqref{eq:tr-lin-time-bound-stronger} in the proof of
  Theorem~\ref{theorem:tr-krylov} (Appendix~\ref{sec:proof-krylov})
  with the inequality $18 (\str)^T \Atr \str +
  4\ltr\norm{\str}^2 \le 36[ \half{\scu}^T A \scu + \frac{1}{6}\rho
    \norm{\scu}^3] = 36[\fcu(0)-\fcu(\scu)]$.
\end{proof}

\subsection{Numerical illustration}\label{sec:exp-rates}
To illustrate our convergence rate guarantees, for each of three controlled
condition numbers $\kappa = \frac{\lmax + \rs}{\lmin + \rs} \in \{10^2,
10^4, 10^6\}$, we generate 5,000 random cubic-regularization problems $f = \fcu$
in
dimension $d=10^{6}$
(see
Appendix~\ref{sec:exp-details} for more details). We solve these problems
with both gradient descent (with step size $\eta = \frac{1}{4}$) and the
Krylov subspace method.  Figure~\ref{fig:exp} summarizes the result, showing
the cumulative distribution (represented by shading) of suboptimality
$\frac{f(x_t) - f(\scu)}{f(0) - f(\scu)}$
versus
iteration number across the generated instances.

As the figure shows, about 20 Lanczos iterations suffice to solve even the 
worst-conditioned instances to about $10\%$ relative accuracy, and 100 
iterations give accuracy better than $1\%$. Moreover, for $t \gtrapprox 
\sqrt{\kappa}$, the approximation error decays exponentially with precisely 
the rate $4/\sqrt{\kappa}$ predicted by our analysis, for almost all the 
generated problems. For $t \ll \sqrt{\kappa}$, the error decays 
approximately as $t^{-2}$. 
Gradient descent converges more slowly, exhibiting linear 
convergence for low $\kappa$ and sublinear convergence with rate $1/t$ 
when $\kappa$ is large. This is consistent with our bounds 
from Section~\ref{sec:non-asymptotic}.

\begin{figure}
	\centering
	\includegraphics[width=0.99\columnwidth]{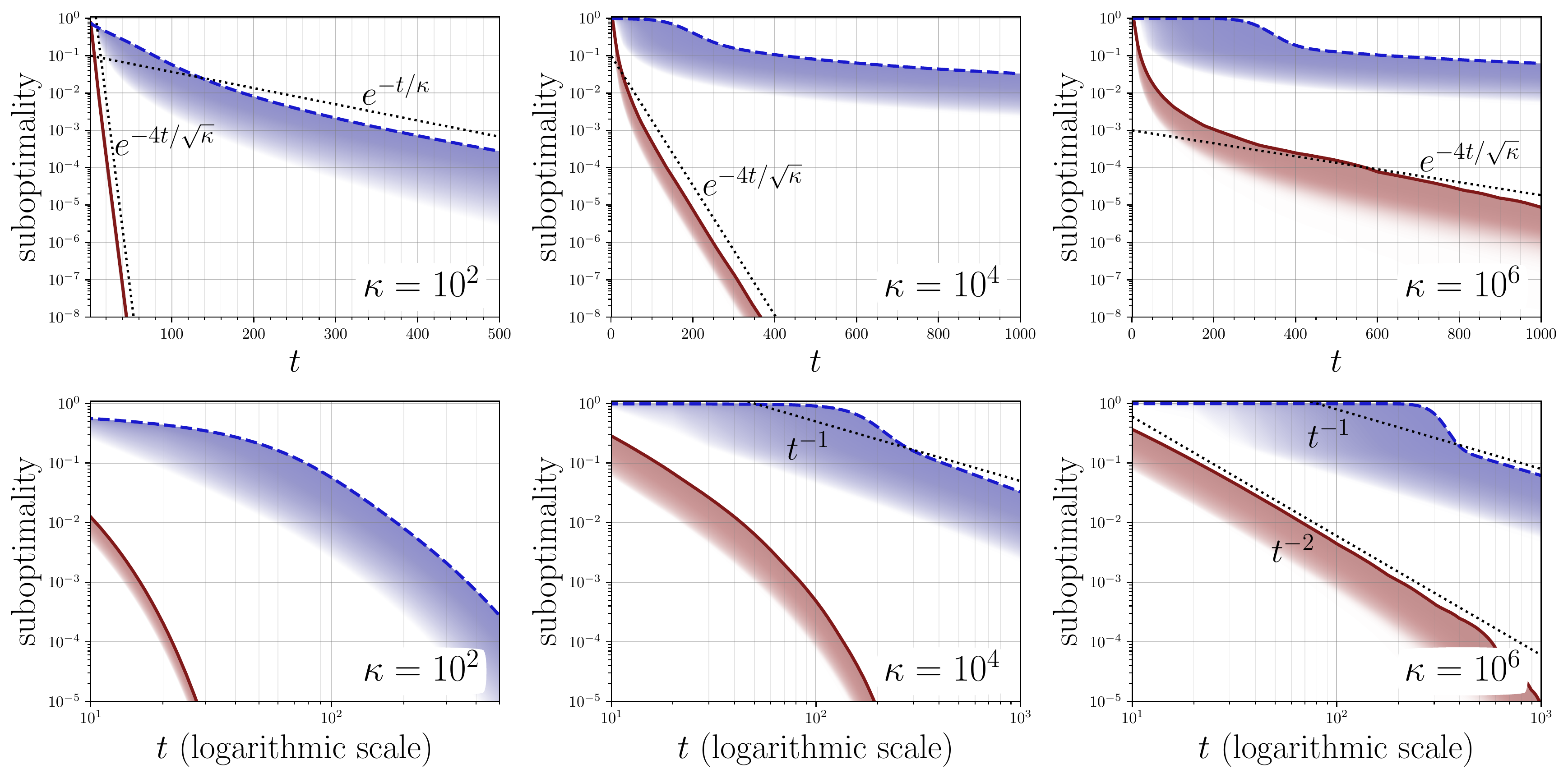}
	\caption{\label{fig:exp}%
		Optimality gap of 
		Krylov subspace solutions (red, solid line) and gradient descent (blue, 
		dashed line) on random 
		cubic-regularization problems, versus iteration count  
		$t$. %
		The shaded regions indicate the cumulative distribution of the 
		optimality gap as a function of $t$, and the bold lines show the 
		maximum value across all generated problems. 
		Columns correspond to different 
		condition numbers 
		$\kappa=({\lmax+\rho\norm{\scu}})/({\lmin+\rho\norm{\scu}})$, 
		and rows differ by scaling of the $t$ axis (linear at the top and 
		logarithmic at the bottom). 
	}
\end{figure}

\section{Randomizing away the hard case}
\label{sec:the-hard-case}

Both gradient descent and Krylov subspace methods may fail to converge to
the global solution of problems~\eqref{eqn:problem-tr}
and~\eqref{eqn:problem-cubic} in the ``hard
case''~\cite{ConnGoTo00,NocedalWr06}, that is, when $b\supind{1} = 
\vmin^T b = 0$.
This is
unavoidable, since in this case methods generate iterates in a subspace 
orthogonal to $\vmin$,
while $x\subopt\supind{1}$ may be non-zero.  Yet as with eigenvector
methods~\cite{KuczynskiWo92,GolubVa89}, simple randomization approaches
address the hard case with high probability, at the cost
of introducing a logarithmic dependence on $d$ to the error bounds. We
describe two approaches: one that perturbs the data $b$,
and one that expands the span of the iterates.

\subsection{Data perturbation}\label{sec:data-perturbation}
Our first approach is to perturb $b$ to a random vector $\tilde{b}$ very
near $b$, which guarantees that $\tilde{b}\supind{1} \neq 0$, while being
near enough $b$ that the corresponding perturbed solutions nearly solve the
initial problem.  We showcase this approach for gradient descent on the 
cubic regularized problem; analogous results for Krylov subspace methods 
for both trust region and cubic regularization are 
straightforward~\cite[Cor.~3]{CarmonDu18}.

\begin{corollary}
  \label{corr:gradient-pert}
  Let Assumptions~\ref{assu:init} and~\ref{assu:step-size} hold, let 
  $\eps,\delta > 0$, and let $\univar \sim \uniform(\sphere^{d-1})$.  Let
  $\tilde{x}_t$ be generated by the gradient descent
  iteration~\eqref{eq:grad-iter} for problem~\eqref{eqn:problem-cubic} with
  $\tilde{b} = b + \sigma \univar$ replacing $b$, where
  \begin{equation*}
    \sigma = \frac{\rho\eps}{\beta+2\rs}\cdot 
    \frac{\sigbar}{12}~\mbox{with}~\sigbar\le1.
  \end{equation*}
  Then with probability at least $1 - \delta$, we have 
  $\fcu(\tilde{x}_t) \le \fcu(\scu) + (1+\sigbar)\eps$ 
  for
  \begin{equation*}
    t \ge
    \frac{6 \tiltgrow(\delta,\sigbar) + 14 \tiltconv\left(\eps\right 
      )}{(1+\sigbar)^{-1}\eta} 
    \min\left\{ 
    \frac{1}{\lmin + \rs},  \frac{10 \norm{\scu}^2}{\eps}
    \right\}.
  \end{equation*}
  where
    $\tiltgrow(\delta,\sigbar)
    \defeq 
    \log\Big(1+\frac{3\I_{\{\lmin < 0\}}\sqrt{d}}{\sigbar\delta}\Big)$ 
    and 
    $\tiltconv(\eps) \defeq
    \log\left(\frac{ (\beta + 2\rs)\norm{\scu}^2 
    }{\eps}\right)$.
\end{corollary}
\noindent
See Appendix~\ref{sec:proof-gradient-pert} for a proof.

\subsection{Subspace perturbation for Krylov methods}
\label{sec:subspace-perturbation}
For Krylov subspace methods we need not
perturb the data and may
instead draw a spherically symmetric random vector
$\univar$ and use the \emph{joint Krylov subspace}
\begin{equation*}
  \Krylov[t][A, \{b, \univar\}] \defeq \mathrm{span}\{b, Ab, \ldots, A^{t-1}b, 
  \univar, A\univar, \ldots, A^{t-1}\univar\}.
\end{equation*}
The \emph{block Lanczos} method~\cite{CullumDo74,Golub77} efficiently solves
both the trust-region and
cubic-regularized problems over $\Krylov[t][A,
  \{b, \univar\}]$, iterating
\begin{equation*}
  \jittertr_t \in \argmin_{x\in\Krylov[t][A,\{b,\univar\}], \norm{x}\le 
    R} \f[A,b](x)
  ~~~ \mbox{and} ~~
  \jittercu_t \in \argmin_{x \in \Krylov[t][A,\{b,\univar\}]}
  \fcu(x)
\end{equation*}
for $\univar \sim \uniform(\sphere^{d-1})$;
we review the technique in
Appendix~\ref{sub:block-lanczos}. Theorem~\ref{theorem:tr-krylov} and
Corollary~\ref{cor:cu} then nearly immediately imply the following convergence
guarantee, whose proof we provide in Appendix~\ref{sec:proof-krylov-random}.
\begin{corollary}
  \label{cor:tr-rand-joint}
  Let $0 < \delta < 1$ and $\jittertr_t$ and $\jittercu_t$ be
  as above, where $\univar \sim \uniform(\sphere^{d-1})$.
  With probability at least $1 - \delta$
  over the choice of $\univar$,  for all $t \in \N$
  \begin{equation*}
    \f(\jittertr_t) - \f(\str) 
    \le \frac{(\lmax - \lmin)R^2}{t^2}
    \left[ 
      4 + \frac{\I_{\{\lmin < 0\}}}{2}
      \log^2\left(\frac{4 d}{\delta^2}\right)
      \right]
  \end{equation*}
  and
  \begin{equation*}
    \fcu(\jittercu_t) - \fcu(\scu)
    \le \frac{(\lmax - \lmin) \norm{\scu}^2}{t^2}
    \left[ 
      4 + \frac{\I_{\{\lmin < 0\}}}{2}
      \log^2\left(\frac{4 d}{\delta^2}\right)
      \right].
  \end{equation*}
\end{corollary}

Corollary~\ref{cor:tr-rand-joint} implies we can solve the trust-region
problem to $\epsilon$ accuracy in roughly $\epsilon^{-1/2}\log d$
matrix-vector products, even in the hard case. The main drawback of this
randomization approach is that half the matrix-vector products are expended
on the random vector; when the problem is well-conditioned or when
$|b\supind{1}|/\norms{b}$ is not extremely small, using the standard
subspace solution is nearly twice as fast. In comparison to the data 
perturbation strategy (Corollary~\ref{corr:gradient-pert}), however, the 
subspace perturbation strategy
converges to the optimal solutions with probability 1 rather than hitting an 
error floor due to the choice of perturbation magnitude $\sigma$.

\subsection{Numerical illustration}\label{sec:exp-rand}
To test the effect of randomization, we generate ``hard case'' problem 
instances (with $\kappa=\infty$; see details in 
Appendix~\ref{sec:exp-details}) and compare the subspace 
randomization scheme (Section~\ref{sec:subspace-perturbation}) with 
data perturbation (Section~\ref{sec:data-perturbation}) applied to a 
Krylov subspace solver with different magnitudes of the perturbation 
parameter $\sigma$. Figure~\ref{fig:pert} shows the results: for 
any fixed target accuracy, some choices of $\sigma$ yield faster 
convergence than the joint subspace scheme. However, for any fixed 
$\sigma$, optimization eventually hits a noise floor,
while the joint subspace scheme continues to improve. Choosing $\sigma$ 
requires striking a balance: if too large, the noise floor is high and may
be worse than no perturbation at all; if too small, escaping the 
unperturbed error level will take too long, and the method might falsely 
declare convergence. A practical heuristic for safely choosing $\sigma$ is 
an interesting topic for future research.

\begin{figure}
	\centering
	\includegraphics[width=0.6\columnwidth]{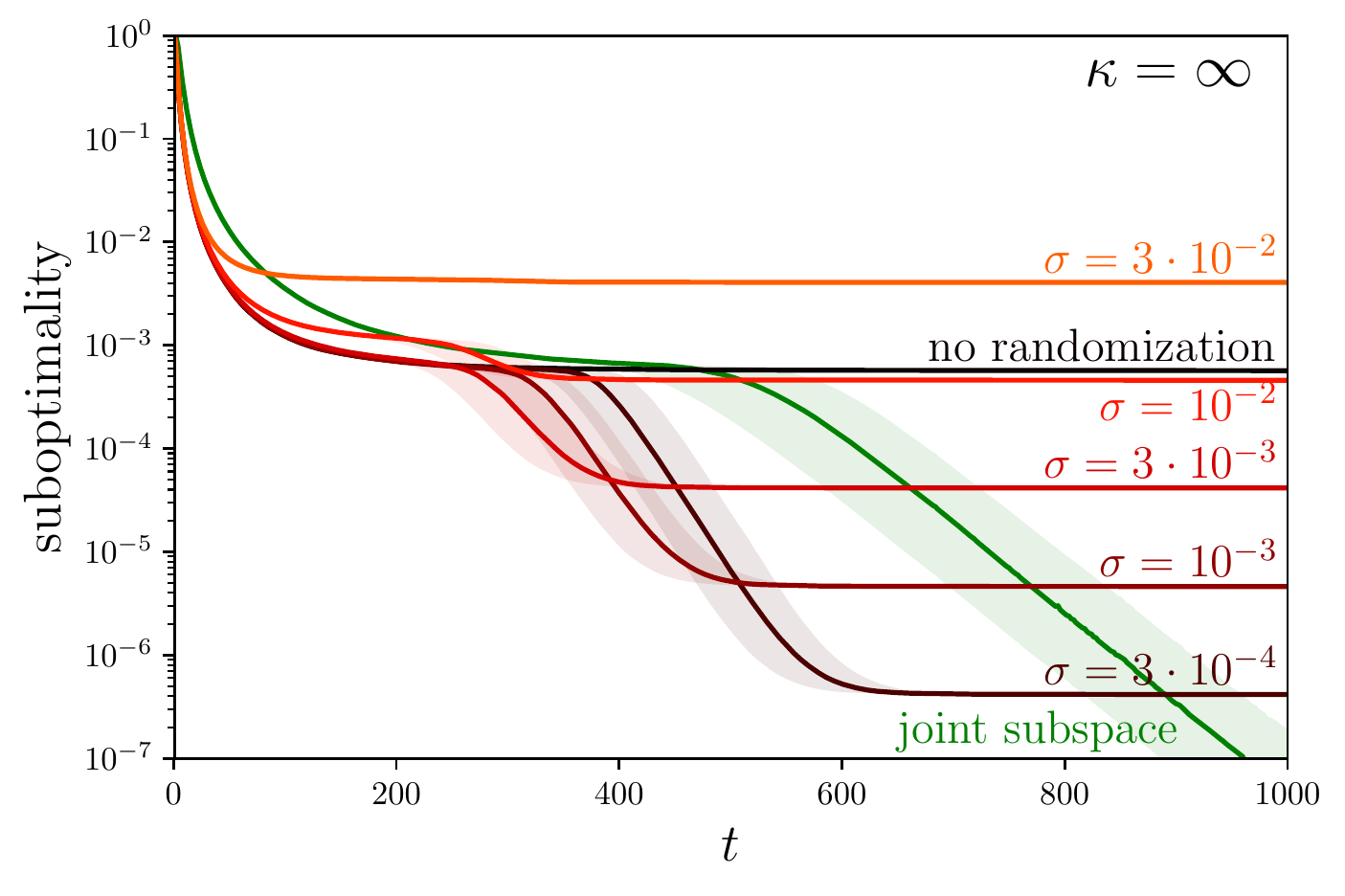}
	\caption{\label{fig:pert}%
		Optimality gap of 
		Krylov subspace solutions on random 
		cubic-regularization problems, versus matrix-vector product number 
		$t$. %
		 Each line represents median suboptimality, and shaded 
				regions represent inter-quartile range.
				Different lines 
				correspond to different randomization settings.
	}
\end{figure}

\section{A Hessian-free majorization method}
\label{sec:majorization}
As the final component of our development, we use our results to analyze the
optimization scheme~\eqref{eqn:generic-problem-iteration}, where we use the
Krylov solver~\eqref{eqn:krylov-iterates} to approximate cubic-regularized
Newton steps. Our purpose is to demonstrate that a method close to
practically effective nonlinear optimization methods---such as
trust-region~\cite{ConnGoTo00} or Adaptive Regularization of Cubics
(ARC)~\cite{CartisGoTo11}---achieves convergence guarantees dominating those
possible for gradient descent.

We wish to minimize a twice differentiable function $g : \R^d \to \R$. More
precisely, we assume that the Hessian $\grad^2 g$ of $g$ is $\rho$-Lipschitz
continuous, and we follow \citet{NesterovPo06} (see
also~\cite{CartisGoTo11,GeHuJiYu15}) to seek an $\epsilon$-second-order
stationary point $y_\epsilon$:
\begin{equation}
  \label{eqn:second-order-crit}
  \norm{\grad g(y_\epsilon)}
  \le \epsilon ~~\textrm{and} ~~
  \hess g(y_\epsilon) \succeq -\sqrt{\rho\epsilon}I.
\end{equation} 
Such points
approximately satisfy second-order necessary conditions
for local minima, providing a stronger guarantee than
$\epsilon$-stationary points satisfying
only $\norm{\grad g(y)} \le \epsilon$.

We revisit method~\eqref{eqn:generic-problem-iteration}, which iteratively
minimizes regularized quadratic models of the function $g$.  To guarantee
convergence, we impose a few assumptions on $g$.
\begin{assumption}
  \label{assumption:g}
  The function $g$ satisfies $\inf g = \glb > -\infty$, is 
  $\beta$-smooth and 
  has $2\rho$-Lipschitz Hessian, i.e., $\norm{\hess g(y) - \hess g(y')} \le 
  2\rho \norm{y-y'}$ for all $y,y'\in\R^d$.
\end{assumption}
The assumptions on boundedness and smoothness are standard, while the third
implies~\cite[Lemma~1]{NesterovPo06} that a cubic-regularized quadratic
model bounds $g$: for all $x, y$,
\begin{equation}
  \label{eqn:majorization}
  g(x) \le g(y)
  + \grad g(y)^{\T}(x-y) + \half (x-y)^{\T} \hess g(y) (x-y)
  + \frac{\rho}{3} \norm{x-y}^3.
\end{equation}
For simplicity we assume that the constants $\beta$ and $\rho$ are
known. This is benign, as we may estimate
these constants without significantly affecting the complexity
bounds, though in practice, careful adaptive estimation
of $\rho$ is crucial for good performance, a primary strength of the ARC
method~\cite{CartisGoTo11}.

\begin{algorithm}[t]
	\caption{A second-order majorization method}\label{alg:SP}
	\begin{algorithmic}[1]
		\Function{\SP}{$y_0$, $g$, $\beta$, $\rho$, $\epsilon$, $\delta$}
		\label{func:SP}
		\State Set $\slb = 
		\sqrt{\frac{\epsilon}{9\rho}}$
		\For {$k = 1, 2, \ldots$} \Comment{guaranteed to terminate in at 
		most  
			$O(\epsilon^{-3/2})$ iterations}
		\State $\Delta_k \gets \callSSP{\hess g(y_{k-1}),\
			\grad g(y_{k-1}),\  \rho,\  \beta,\ 
			\slb,\ \frac{\delta}{2k^2}}$
		\State $y_k \gets y_{k - 1} + \Delta_k$
		\If{ $g(y_k)
			> g(y_{k-1}) - \frac{1}{12}\rho r^3$} 
			\label{line:progress}
		\State $\Delta\subfin
		\gets \callSFSP{\hess g(y_{k-1}),\  \grad g(y_{k-1}),\
			\rho,\  \beta,\ r,\ \frac{2\epsilon}{3}}$
		\State \Return{$y_{k-1} + \Delta\subfin$}
		\EndIf
		\EndFor
		\EndFunction
	\end{algorithmic}
	\begin{algorithmic}[1]
		\Function{\SSP}{$A$, $b$, $\rho$, $\beta$, $\slb$,
			$\delta$}\label{func:SSP}
		\State Draw $\univar$ uniformly from the sphere in $\R^d$
		\State Set $\Tinner =
		\ceil{\sqrt{\frac{24 \beta}{\rho \slb}
				\left(4 + \half \log^2 \frac{4d}{\delta^2}\right)}\,}$
			\Comment{
				$\Tinner = \Otil{\epsilon^{-1/4}}$}
		\State \Return $\argmin_{x \in \Krylov[\Tfinal][A, \{b, \univar\}]}
		\{\fcu(x)=\half x^T A x + b^T x + \frac{\rho}{3}\norm{x}^3\}$
		\EndFunction
	\end{algorithmic}
	\begin{algorithmic}[1]
		\Function{\SFSP}{$A$, $b$, $\rho$, $\beta$, 
			$r$,
			$\epsgrad$}\label{func:SFSP}
		\State Set $\Tfinal = \ceil{
			\frac{1}{4}\sqrt{\frac{\beta+2\rho r}{\rho r}} \log 
			\frac{36(\beta+2\rho r)^2 
			r^2}{\epsgrad^2}}$
		\Comment{
			$\Tfinal = \Otil{\epsilon^{-1/4}}$}
		\State \Return $\argmin_{x \in \Krylov[\Tfinal]}
		\{\half x^T A x + b^T x + \rho r \norm{x}^2\}$
		\EndFunction
	\end{algorithmic}
\end{algorithm}

Algorithm~\ref{alg:SP} outlines a
majorization-minimization~\cite{NocedalWr06} strategy for
optimizing $g$.  At each iteration, the method approximately minimizes a
cubic-regularized quadratic model $\fcu[\grad^2 g(y_k),\grad g(y_k), \rho]$
of $g$ via a call to \callSSP{}, then shifts the
iterate by the approximate minimizer $\Delta_k$. If the new iterate $y_k$
makes sufficient progress (decreasing the $g$ by at least
$\frac{1}{12}\rho r^3$), the algorithm proceeds. Otherwise, it minimizes a
single regularized quadratic model via a call of \callSFSP{} and halts.
  
The progress criterion immediately bounds the number $\Touter$ of calls 
to \callSSP{}, as each iteration satisfies $g(y_{k-1}) - 
g(y_k) 
\ge \frac{1}{12}\rho r^3 = \Omega(1) \frac{\epsilon^{3/2}}{\rho^{1/2}}$, so
\begin{equation}
  \label{eqn:nesterov-polyak-progress}
  g(y_0) - g(y_{\Touter}) 
  = \sum_{k = 1}^{\Touter} \left[g(y_{k-1}) - g(y_{k})\right]
  = \Omega(1) \Touter \rho^{-1/2} \epsilon^{3/2} .
\end{equation}
As $g(y_{\Touter}) \ge \glb$, we rearrange this to obtain $\Touter \le
O(1) \frac{\sqrt{\rho}(g(y_0) - \glb)}{\epsilon^{3/2}}$; this is the
familiar $\epsilon^{-3/2}$ iteration bound of \citet{NesterovPo06}. 
To within logarithmic factors,
both $\Tinner$ and $\Tfinal$ in Alg.~\ref{alg:SP} 
scale as $\beta^{1/2}(\rho \epsilon)^{-1/4}$, so the 
total first-order evaluation cost of the algorithm is 
$\frac{\beta^{1/2}\rho^{1/4}(g(y_0)-\glb)}{\epsilon^{7/4}}$
(to within logarithmic factors).

It remains to guarantee that at termination, \callSP{} outputs 
an approximate second-order stationary point.
The majorization property~\eqref{eqn:majorization} guarantees that the 
step of $y_k \gets y_{k-1} + \Delta_k$ decreases $g$ by at least the 
amount that $\Delta_k$ decreases the model at $y_{k-1}$. The following 
lemma guarantees that $\callSSP{}$ decreases the model by at least 
$\frac{1}{12}\rho r^3$ whenever the exact model minimizer has norm at 
least $r$.

\begin{lemma}
  \label{lem:SSP}
  Let $A \in \R^{d\times d}$ satisfy $\opnorm{A} \le \beta$, $b\in\R^d$,
  $\rho > 0, \slb > 0$, $\delta \in (0,1)$, and
  $\scu = \argmin_x \fcu[A,b,\rho](x)$. With probability at least
  $1-\delta$, if $\norm{\scu} \ge \slb$ 
  then $x=\text{\callSSP{$A$, $b$, $\rho$, $\slb$, $\delta$}}$ satisfies 
  $\fcu(x) \le -\frac{1}{12} \rho\slb^3$.
\end{lemma}
\noindent
Lemma~\ref{lem:SSP} follows by straightforward application of 
Corollary~\ref{cor:tr-rand-joint}; we provide the proof in 
Appendix~\ref{sec:proof-SSP}. It is the inexact analogue of the 
 progress guarantee 
of~\citet[Lemma 4]{NesterovPo06}, which forms the basis of their 
convergence proof.

Now let 
$K=\Touter=O(\epsilon^{-3/2})$ be the final iterate of 
Algorithm~\ref{alg:SP}, and let 
 $\siter_K$ be the global
minimizer (in $\Delta$) of the model~\eqref{eqn:majorization} at 
$y=y_{K-1}$.
Lemma~\ref{lem:SSP} guarantees
that with high probability, since \callSSP{} fails to meet the progress
condition in line~\ref{line:progress}, then $\|\siter_k\|
\le r$. Therefore, by Proposition~\ref{prop:characterization}, it holds that 
$\hess g(y_{K-1}) \succeq -\rho r I \succeq \sqrt{\rho \epsilon} I$. 
 It is possible, nonetheless,
that $\norm{\grad g(y_{K-1})} > \epsilon$; to address this, we correctively 
minimize a regularized quadratic model around $y_{K-1}$, taking 
advantage of the fact that $\hess g(y_{K-1}) \succeq -\rho r I$ to argue 
that the regularized model is strongly convex and hence that the Krylov 
subspace (i.e., conjugate gradient) method converges linearly. We 
formalize 
this guarantee in the following lemma; see Appendix~\ref{sec:proof-SFSP} 
for proof.
\begin{lemma}
	\label{lem:SFSP}
	Let $A \in \R^{d\times d}$ satisfy $-\rho r I \preceq A \preceq \beta I$ 
	for  
	$\rho, \slb, \beta > 0$, and let $b\in\R^d$. If $\norm{(A+2\rho r 
	I)^{-1}b} \le r$ 
	then $x=\text{\callSFSP{$A$, $b$, $\rho$, $\slb$, $\epsgrad$}}$ 
	satisfies $\norm{x} \le r$ and $\norm{Ax +b} \le \epsgrad + 2\rho r^2$.
\end{lemma}

Combining Lemmas~\ref{lem:SSP} and~\ref{lem:SFSP} and leveraging the  
continuity of $\hess g$ similarly to 
\citet[Lemma 5]{NesterovPo06}, we obtain the following guarantee
for Algorithm~\ref{alg:SP}, whose proof we provide in
Appendix~\ref{sec:proof-SP}.
\begin{proposition}
  \label{prop:SP}
  Let $g$ satisfy Assumption \ref{assumption:g},  let $y_0\in\R^d$ be 
  arbitrary,
  and let $\delta \in (0, 1]$ and $\epsilon \le
    \min\{\beta^2/\rho,\rho^{1/3}(g(y_0)-\glb)^{2/3}\}$. With probability at
    least $1-\delta$, Algorithm \ref{alg:SP} finds an
    $\epsilon$-second-order stationary point~\eqref{eqn:second-order-crit}
    in at most
    \begin{equation}
      O(1) \cdot \frac{\beta^{1/2} \rho^{1/4} (g(y_0)-\glb)}{\epsilon^{7/4}}
      \left[
        \log\frac{d}{\delta^2}
        + \log \frac{\beta^{1/2} \rho^{1/4} (g(y_0)-\glb)}{\epsilon^{7/4}}
        \right]
      \label{eq:trustregion-complexity}
    \end{equation}
    Hessian-vector product evaluations and
    at most
    \begin{equation*}
      O(1) \cdot \frac{\sqrt{\rho}(g(y_0)-\glb)}{\epsilon^{3/2}}
    \end{equation*}
    calls to \callSSP{} and gradient evaluations.
\end{proposition}

We conclude with two brief remarks: First, as a consequence of the results
here, a cubic-regularization approach with a natural efficient cubic
subproblem solver achieves the best known rates of convergence for
first-order methods, meeting the bounds of recent methods using acceleration
techniques~\cite{AgarwalAlBuHaMa17, CarmonDuHiSi18}.  Second, the
assumptions on $\epsilon$ in Proposition~\ref{prop:SP} guarantee that the
bound~\eqref{eq:trustregion-complexity} is non-trivial.  If $\epsilon >
\beta^2 / \rho$, then the Hessian guarantee~\eqref{eqn:second-order-crit} is
trivial, and with constant stepsize $\eta = \frac{1}{\beta}$, gradient
descent guarantees~\cite[Eq.~(1.2.13)]{Nesterov04} an iterate $y_k$ with
$\norm{\grad g(y_k)} \le \epsilon$ in at most
\begin{equation*}
  \frac{\beta (g(y_0) - \glb)}{\epsilon^2}
  = \frac{\beta^{1/2} (g(y_0) - \glb)}{\epsilon^{7/4}}
  \left(\frac{\beta^2}{\epsilon}\right)^{1/4}
  < \frac{\beta^{1/2} \rho^{1/4} (g(y_0) - \glb)}{\epsilon^{7/4}}
\end{equation*}
iterations,
so that gradient descent outperforms the majorization method.
Similarly, the final statement in Proposition~\ref{prop:SP}
shows that if $\epsilon > \rho^{1/3}(g(y_0)-\glb)^{2/3}$ then
\callSSP{} executes $O(1)$ times, and the overall first-order complexity
becomes $\Otil{1} \frac{\beta^{1/2}}{(\rho \epsilon)^{1/4}}$.

\section{Discussion}
\label{sec:discussion}
We explore the connections between our results on potentially nonconvex
quadratic problems and classical results on convex optimization and the
eigenvector problem in more detail. We also remark on the differences
between the analyses we employ for the gradient-descent and Krylov methods
and note a few additional results that appear in the original
papers~\cite{CarmonDu19,CarmonDu18}.

\subsection{Comparison to convex optimization}
For $L$-smooth and $\lambda$-strongly convex functions with a bound 
$R$ on the distance between the initial point and an optimum, gradient 
descent finds an $\eps$-suboptimal point in 
\begin{equation*}
  O(1) \cdot \min\left\{ \frac{L}{\lambda}\log \frac{L R^2}{\eps},
    \frac{L R^2}{\eps} \right\}
\end{equation*}
iterations~\cite{Nesterov04}. 
For the (possibly nonconvex) problem~\eqref{eqn:problem-cubic},
gradient descent finds an
$\eps$-suboptimal point (with probability at least $1 - \delta$) within
\begin{equation*}
  O(1) \cdot
  \min \left\{
    \frac{\Ls}{\lambda\subopt}, 
    \frac{\Ls \norm{\scu}^2}{\eps} \right\}\left[\log\frac{\Ls \norm{\scu}^2}{ \eps}
    + \log
    \left(1 + \I_{\{\lmin < 0\}}\frac{ d}{\delta}\right)
  \right]
\end{equation*}
iterations by Corollary~\ref{corr:gradient-pert},
where $\Ls = \beta + 2\rs$ and $\lambda\subopt = \rs + 
\lmin$. The parallels are immediate: by
Lemma~\ref{lemma:monotone}, $\Ls$ and $\norm{\scu}$ are precise analogues of
$L$ and $R$ in the convex setting. Moreover,
$\lambda\subopt$
plays the role of the strong convexity parameter $\lambda$ but is
well-defined even when $\fcu$ is not convex.  When $\lmin(A) \ge 0$, $\fcu$
is $\lmin$-strongly convex, and because $\rs + \lmin > \lmin$, our analysis
for the cubic problem~\eqref{eqn:problem-cubic} guarantees better
conditioning than the generic convex result.  The difference between $\rs +
\lmin$ and $\lmin$ becomes significant when $b$ is large, as
$\norms{\scu}$ is monotonic in $\alpha > 0$ whenever
$b = \alpha u$ for a vector $u$.
Even in the
nonconvex case that $\lmin < 0$, gradient descent still exhibits linear
convergence for high accuracy solutions when $\eps/\norm{\scu}^2
\le \rs + \lmin$.
When $\lmin < 0$, our guarantee becomes probabilistic and contains a
$\log(d/\delta)$ term. Such a term does not appear in results on convex
optimization, and saddle-points in the objective~\cite{SimchowitzAlRe18}
make it fundamental.

The analogy of our results to the convex case
extends to accelerated methods in optimization.
With the notation as above, for $L$-smooth and
$\lambda$-strongly convex functions and $R \ge \norm{x_0 - x\opt}$,
Nesterov's accelerated gradient method~\cite{Nesterov04} finds an
$\eps$-suboptimal point within
\begin{equation*}
  O(1) \cdot \min\bigg\{\sqrt{\frac{L}{\lambda}}
  \log \frac{L R^2}{\eps}, \sqrt{\frac{LR^2}{\eps}}\bigg\}
\end{equation*}
iterations. As $\fcu(0) - \fcu\opt \le O(1) \Ls \norm{\scu}^2$, 
Corollary~\ref{cor:tr-rand-joint} guarantees that the perturbed 
joint Krylov method of Section~\ref{sec:subspace-perturbation}
finds an $\eps$-suboptimal point within
\begin{equation*}
  O(1) \cdot \min\bigg\{
  \sqrt{\frac{\Ls}{\lambda\subopt}} \log \frac{\Ls \norm{\scu}^2}{\eps},
  \sqrt{\frac{\Ls \norm{\scu}^2}{\eps}}
  \left(1 + \I_{\{\lmin < 0\}}\log\frac{d}{\delta}\right)
  \bigg\}
\end{equation*}
iterations. Just as with gradient descent, the nonconvexity engenders
a necessary $\log\frac{d}{\delta}$ term, but we see completely parallel
results. Indeed, in the case that $A\succeq 0$ and $\rho=0$ (or 
$R=\infty$ for~\eqref{eqn:problem-tr}), the Krylov subspace solutions are 
the iterates of conjugate gradient, and our bounds include their 
convergence guarantees as special cases.

\subsection{Comparisons with the eigenvector problem}
Minimizing $x^T A x$ subject to $\norm{x}=1$ is a prototypical nonconvex
yet tractable optimization problem, whose solution is the eigenvector
$\vmin$ of $A$ corresponding to its smallest eigenvalue $\lmin$. The power
method iterates $x_{t+1} = (I- (1/\beta) A)x_t/\norms{(I-(1/\beta) A)x_t}$
to solve this problem, and when $x_0$ is uniform on the unit sphere, it
achieves accuracy $\eps$ in
$O(1)\frac{\beta}{\eps}\log\frac{d}{\delta}$ steps with probability at least
$1-\delta$~\cite{GolubVa89,KuczynskiWo92}. The power method is precisely
projected
gradient descent on $\{x \mid \norm{x} = 1\}$, and its convergence guarantee
mirrors our Corollary~\ref{corr:gradient-pert} for data
perturbation. Indeed, when $b=0$ and $\lmin(A)<0$, the solution
to~\eqref{eqn:problem-cubic} is proportional to $\vmin$ and
data-perturbed gradient descent finds it.  Krylov subspace methods also
solve the eigenvector problem; this is the typical Lanczos
method~\cite[cf.][]{KuczynskiWo92,TrefethenBa97}. Indeed, our analysis of
Krylov subspace solutions to the more general problem~\eqref{eqn:problem-tr}
directly relies on this approach, and we recover its guarantees for the
subspace perturbation approach (Corollary~\ref{cor:tr-rand-joint}) with
$b=0$.

The literature on the eigenvector problem also identifies a 
\emph{gap-dependent} convergence regime, where the power and Lanczos 
methods converge linearly with rate depending on the eigen-gap
$\min_k\{\lambda\supind{k}(A) - \lambda\supind{1}(A)
\mid \lambda\supind{k}(A) > \lambda\supind{1}(A)\}$ of $A$.
The parallels here are less immediate;
our paper~\cite{CarmonDu19} shows that
gradient descent exhibits such a convergence regime
for problem~\eqref{eqn:problem-cubic}, 
though we defer
deeper investigation.

\subsection{Comparison of proof strategies}
We return briefly to our discussion of analysis strategies for
nonconvex optimization problems in 
Section~\ref{sec:concurrent-subsequent}.
Our analysis of Krylov subspace methods~\eqref{eqn:krylov-iterates}
leverages the fact that by 
definition they outperform all algorithms with iterates in  
the Krylov subspace~\eqref{eqn:krylov-subspace}; we argue 
\emph{some} (possibly impractical) algorithm does well, and hence so does 
the Krylov subspace method. This allows us to obtain strong convergence 
guarantees, essentially with no assumptions, but occludes the picture of 
how the Krylov subspace iterations behave. 

In contrast, our analysis of gradient descent paints a very detailed picture
of the dynamics of the iterates: their norm is monotonic, growing
exponentially until they are sufficiently far from all saddle points, and
subsequently they converge linearly towards the minimizer. The
iterates remain in a half-space whose only stationary point is the global
solution (see Figure~\ref{fig:simple-gradients}). This description gives
insight into the mechanisms by which nonconvexity affects convergence at
the expense of requiring a particular initialization
(Assumption~\ref{assu:init}) and tailored arguments that are non-trivial to
extend even for the trust-region
problem~\eqref{eqn:problem-tr}. Nonetheless, we hope this analysis may serve
as a prototype for the growing collection ``trajectory-based''
analyses~\cite{LiYu17,MaWaChCh19}.

\subsection{Additional results}
We conclude by briefly mentioning a few results in our original 
works~\cite{CarmonDu19,CarmonDu18} that we omit for brevity. 
In paper~\cite{CarmonDu19}, we consider only gradient descent and provide 
additional convergence guarantees that depend on the eigen-gap of $A$,
as well as giving a line-search procedure within gradient descent.
Our paper~\cite{CarmonDu18} focuses on the Krylov 
subspace solutions; in addition to the results we describe here, we show a 
matching lower proving the sharpness of our analysis for these methods 
and---by a resisting oracle argument---their optimality compared to any 
deterministic algorithm operating sufficiently high dimension.

\arxiv{
  \section*{Acknowledgment}
  YC and JCD were
  partially supported by the SAIL-Toyota Center for AI Research and the
  Office of Naval Research award N00014-19-2288. YC was partially
  supported by the Stanford Graduate Fellowship and the Numerical
  Technologies Fellowship. JCD was partially supported by the National
  Science Foundation award NSF-CAREER-1553086.
  
  \newpage
  \bibliographystyle{abbrvnat}
  \setlength{\bibsep}{3pt}

  \newpage
}

\appendix

\section{Computing Krylov subspace solutions}
\label{sec:lanczos}
Setting $A_\lambda = A + \lambda I$, generic instances of
problems~\eqref{eqn:problem-tr} and~\eqref{eqn:problem-cubic} can be
globally optimized~\cite{ConnGoTo00,CartisGoTo11} via Newton's method to
find the roots (respectively) of the one-dimensional equations
\begin{equation}
  \label{eq:lambda-search}
  \norm{A_{\lambda}^{-1}b} = R, ~
  \lambda  > \hinge{-\lmin}
  ~~~ \mbox{and} ~~~
  \norm{A_\lambda^{-1}b} = \lambda/\rho,
  ~
  \lambda \ge \hinge{-\lmin}.
\end{equation}
For high-dimensional problems where linear system solves $A_\lambda^{-1} b$
become expensive,
a general approach to obtaining approximate solutions is to 
constrain the domain to a linear subspace $\mc{Q}_t \subset \R^d$ of 
dimension $t \ll d$. Let $Q_t \in \R^{d\times t}$ be an orthogonal basis for 
$\mc{Q}_t$ ($Q_t^T Q_t = I$). Finding the global minimizer in $\mc{Q}_t$ 
is then equivalent to solving
\begin{equation*}
  \tilde{x}_t = \argmin_{y \in \R^t}
  \Big\{\half y^T Q_t^T A Q_t y + (Q_t^T b)^T y
  + \regpenalty(\norm{y})\Big \}
\end{equation*}
for $\regpenalty(r) = \infty \cdot \I(r \le R)$ for  
problem~\eqref{eqn:problem-tr} and $\regpenalty(r) = \frac{\rho}{3} r^3$ 
for~\eqref{eqn:problem-cubic}, then setting $x_t = Q_t \tilde{x}_t$.
For sufficiently large $d$, the time to solve such problems is
dominated by the $t$ matrix-vector products required to construct 
$Q_t^T A Q_t$.

Choosing the Krylov subspaces $\mc{Q}_t = \Krylov$ offers a significant
efficiency boost: we can construct a basis $Q_t$ for which
$Q_t^T A Q_t$ is tridiagonal using the Lanczos process~\cite[Part
  VI]{TrefethenBa97}, which beginning from
$q_1 = b / \norm{b}, q_0 = 0$ recurses
\begin{equation*}
  \alpha_t = q_t^T A q_t
  ~,~
  q'_{t+1} = A q_t - \alpha_t q_t - \beta_t q_{t-1}
  ~,~
  \beta_{t+1} = \norms{q'_{t+1}}
  ~,~
  q_{t+1} = q'_{t+1}/\norms{q'_{t+1}}.
\end{equation*}
The vectors $q_1, \ldots, q_t$ give the columns of $Q_t$ while $\alpha_1,
\ldots, \alpha_t$ and $\beta_2, \ldots, \beta_t$, respectively, give the
diagonal and off-diagonal elements of the symmetric tridiagonal matrix
$\tilde{A} = Q_t^T A Q_t$; this makes solving
equations~\eqref{eq:lambda-search} easy.  One straightforward approach is to
compute the eigenvalues of $\tilde{A}$, which for a $t \times t$ symmetric
tridiagonal matrix takes $O(t\log t)$ time~\cite{CoakleyRo13}. A more
efficient and practical approach is to iteratively solve systems of the form
$\tilde{A}_{\lambda}x = -Q_t^T b$ and update $\lambda$ using Newton
steps~\cite[Ch.~7.3.3]{CartisGoTo11,ConnGoTo00}. Every tridiagonal system
solution takes time $O(t)$, and the Newton steps are
linearly convergent (with local quadratic
convergence). In our experience 20
Newton steps generally suffice to reach machine precision, and so the
computational cost is essentially linear in $t$.
To avoid keeping $Q_t$ in memory (if $t\cdot d$ 
storage is too demanding), one may run the Lanczos process twice, once to
find $\tilde{x}$ and again to find $x=Q_t \tilde{x}$.

The Lanczos process produces the same result as Gram-Schmidt 
orthonormalization of the vectors $[b, Ab, \ldots, A^{t-1}b]$ but 
uses the 
special structure of the matrix to avoid computing
structurally zero inner products. When run for many iterations, the Lanczos 
process is unstable~\cite{TrefethenBa97}, but in
our setting we usually seek low to moderate accuracy solutions and will 
usually stop at $t < 100$, for which Lanczos is reasonably numerically
stable with 
floating point arithmetic even when $d$ is large.

\subsection{Computing joint Krylov subspace 
solutions}\label{sub:block-lanczos}

To solve equations~\eqref{eq:lambda-search} 
in subspaces of the form
\begin{equation*}
\Krylov[mt][A, \{v_1, \ldots, v_m\}] \defeq \mathrm{span}\{A^j 
v_i\}_{i\in\{1,\ldots, m\},j\in\{0,\ldots,t-1\}}
\end{equation*}
we may use the block Lanczos method~\cite{CullumDo74,Golub77}, a 
natural generalization of 
the Lanczos method that creates an orthonormal basis for the subspace 
$\Krylov[mt][A, \{v_1, \ldots, v_m\}]$ in which $A$ has a block tridiagonal 
form. Overloading the notation defined above so that now $q_t \in 
\R^{d\times m}$ and $\alpha_t, \beta_t \in \R^{m\times m}$ are matrices, the block 
Lanczos recursion is
\begin{equation*}
  \alpha_t = q_t^T A q_t
  ~,~
  q'_{t+1} = A q_t - q_t \alpha_t -  q_{t-1} \beta_t^T
  ~,~
  (q_{t+1}, \beta_{t+1}) = \mathrm{QR}(q'_{t+1}).
\end{equation*}
where $\mathrm{QR}$ is the QR decomposition,
and the initial conditions are that $q_1$ is an 
orthonormalized version of $[v_1, \ldots, v_m]$ and $q_0=0$. The matrix 
$\tilde{A} = Q_t^T A Q_t$ is now block tridiagonal, with the diagonal and 
sub-diagonal blocks given by $\{\alpha_i\}_{i\in\{1,\ldots,t\}}$ and 
$\{\beta_i\}_{i\in\{2,\ldots,t\}}$ respectively. Since the $\beta$ matrices 
are upper diagonal, $\tilde{A}$ is a symmetric banded matrix with $m$ 
non-zero sub-diagonal bands; such matrices admit fast Cholesky 
decompositions (in time linear in $m^2 t$), and consequently the Newton 
method for the system~\eqref{eq:lambda-search} is efficient
when $m$ is small (e.g.\ $m = 2$).

\section{Proof of Lemma~\ref{lemma:monotone}}
\label{app:prel}
Throughout this section, we let $f(x) = \f(x) + \frac{\rho}{3} \norm{x}^3$
for short.
Before proving Lemma~\ref{lemma:monotone}, we state and prove
two technical lemmas (see Sec.~\ref{sec:finally-proof-monotone-weak}
for the proof conditional on these lemmas).
For the first lemma,
let $\chi \in \R^d$ satisfy $\chi\supind{1} \le \chi\supind{2}
\le \ldots \le \chi\supind{d}$, let
$\nu_t$ be a nonnegative and nondecreasing sequence, $0 \le \nu_1 \le
\nu_2 \le \ldots$, and consider the process
\begin{equation}\label{eq:zt-def}
  z_{t}\supind{i}=(1-\chi\supind{i}-\nu_{t-1})z_{t-1}\supind{i}+1.
\end{equation}
Additionally, assume
$1-\chi\supind{i}-\nu_{t-1}\geq0$ for all $i$ and $t$.
\begin{lemma}
  \label{lem:ordering}
  Let $z_{0}\supind{i}=c_{0}\geq0$ for every $i \in [d]$. Then for
  every $t \in \N$ and $j \in [d]$, the following holds:
  \begin{enumerate}[label=(\roman*)]
  \item \label{item:future-signs}
    If $z_{t}^{\left(j\right)}\leq z_{t-1}^{\left(j\right)}$, then also
    $z_{t'}^{\left(j\right)}\leq z_{t'-1}^{\left(j\right)}$ for every
    $t'>t$.
  \item \label{item:sign-ratio} If
    $z_{t}^{\left(j\right)}\geq z_{t-1}^{\left(j\right)}$, then
    $z_{t}^{\left(j\right)}/z_{t+1}^{\left(j\right)}\geq
    z_{t}\supind{i}/z_{t+1}\supind{i}$ for every $i\le j$.
  \item \label{item:bigger-signs}
    If $z_{t+1}\supind{i}\leq z_{t}\supind{i}$,
    then $z_{t+1}\supind{j}\leq z_{t}\supind{j}$
    for every $j\geq i$.
  \end{enumerate}
\end{lemma}
\begin{proof}
  For shorthand, we define
  $\delta_{t}\supind{i} \defeq \chi\supind{i}+\nu_{t}$. 
  
  We first establish part~\ref{item:future-signs} of the lemma. 
  By~\eqref{eq:zt-def}, we have
  \begin{alignat*}{1}
    z_{t+1}\supind{j} - z_{t}\supind{j}
    =(1-\delta_{t-1}^{\left(j\right)})(z_{t}^{\left(j\right)}-z_{t-1}^{\left(j\right)})- 
    (\delta_{t}^{\left(j\right)}-\delta_{t-1}^{\left(j\right)})z_{t}^{\left(j\right)}
  \end{alignat*}
  By our assumptions that $z_{0}\supind{j}\geq0$ and that
  $1-\delta_{t}\supind{j}\geq0$ for every $t$ we immediately have that
  $z_{t}\supind{j}\geq0$, and therefore also
  $(\delta_{t}\supind{j}-\delta_{t-1}\supind{j})z_{t}\supind{j} = 
  (\nu_{t}-\nu_{t-1})z_{t}\supind{j}\geq0$.
  We therefore conclude that
  \begin{equation*}
  z_{t+1}\supind{j}-z_{t}\supind{j}\leq (1-\delta_{t-1}\supind{j})  
  (z_{t}\supind{j}-z_{t-1}\supind{j} ) \le 0,
  \end{equation*}
  and induction gives part~\ref{item:future-signs}.
  
  To establish part~\ref{item:sign-ratio} of the lemma, first note
  that by the contrapositive of part~\ref{item:future-signs},
  $z_{t}\supind{j}\geq z_{t-1}\supind{j}$
  for some $t$ implies $z_{t'}\supind{j}\geq z_{t'-1}\supind{j}$
  for any $t'\leq t$. We prove by induction that
  \begin{equation}
    z_{t'}\supind{i}-z_{t'}\supind{j}
    \leq (\chi\supind{j}-\chi\supind{i}) z_{t'}\supind{i}z_{t'}\supind{j}
    \label{eqn:kappa-z-inequality}
  \end{equation}
  for any $i\leq j$ and $t'\leq t$. The basis of the induction is
  immediate from the assumption $z_{0}\supind{i}=z_{0}\supind{j}\ge0$.
  Assuming the property holds through time $t' - 1$ for $t' \le t$, we
  obtain
  \begin{alignat*}{1}
    \frac{z_{t'}\supind{i}-z_{t'}\supind{j}}{z_{t'}\supind{i}z_{t'}\supind{j}}
    & =\frac{(1-\delta_{t'-1}\supind{i})(z_{t'-1}\supind{i}-z_{t'-1}\supind{j})+
      (\delta_{t'-1}\supind{j}-\delta_{t'-1}\supind{i})z_{t'-1}\supind{j}}{
      z_{t'}\supind{i}z_{t'}\supind{j}} \\
    & \le \frac{(1 - \delta_{t'-1}\supind{i}) (\chi\supind{j} - \chi\supind{i})
      z_{t'-1}\supind{i} z_{t'-1}\supind{j}}{z_{t'}\supind{i} z_{t'}\supind{j}}
    = (\chi\supind{j} - \chi\supind{i}) \frac{z_{t'-1}\supind{j}}{
      z_{t'}\supind{j}} \le \chi\supind{j} - \chi\supind{i}
  \end{alignat*}
  where the first inequality uses inequality~\eqref{eqn:kappa-z-inequality} 
  (assumed by induction)
  and the second uses $z_{t'-1}\supind{j} \le
  z_{t'}\supind{j}$ for any $t'\le t$, as argued above.
  With the bound
  $z_{t}\supind{i}-z_{t}\supind{j}\leq(\chi\supind{j}-\chi\supind{i})
  z_{t}\supind{i}z_{t}\supind{j}$
  in place, we may finish the proof of part~\ref{item:sign-ratio} by noting 
  that
  \begin{alignat*}{1}
    \frac{z_{t}\supind{j}}{z_{t+1}\supind{j}}-\frac{z_{t}\supind{i}}{z_{t+1}\supind{i}}
    & =\frac{z_{t+1}\supind{i}z_{t}\supind{j}-z_{t+1}\supind{j}z_{t}\supind{i}}{z_{t+1}\supind{j}z_{t+1}\supind{i}}
    =\frac{(\chi\supind{j}-\chi\supind{i})z_{t}\supind{i} 
    z_{t}\supind{j}-(z_{t}\supind{i}-z_{t}\supind{j})}{z_{t+1}\supind{j} 
    z_{t+1}\supind{i}}\geq0.
  \end{alignat*}

  Lastly, we prove part~\ref{item:bigger-signs}. If
  $z_{t}\supind{j}\leq z_{t-1}\supind{j}$ then we have
  $z_{t+1}\supind{j}\leq z_{t}\supind{j}$ by
  part~\ref{item:future-signs}. Otherwise we have
  $z_{t}\supind{j}\ge z_{t-1}\supind{j}$, and so
  $z_{t}\supind{j}/z_{t+1}\supind{j}\geq z_{t}\supind{i}/z_{t+1}\supind{i}$ by
  part~\ref{item:sign-ratio}. As $z_{t+1}\supind{i}\leq z_{t}\supind{i}$,
  this implies
  $z_{t}\supind{j} / z_{t+1}\supind{j} \geq
  z_{t}\supind{i} / z_{t+1}\supind{i} \geq 1$
  and therefore $z_{t+1}\supind{j}\leq z_{t}\supind{j}$
  as required.
\end{proof}

Our second technical lemma provides a lower bound on certain
inner products in the gradient descent iterations.
In the lemma, we recall the definition~\eqref{eq:R-upper-def} of
$\Rupper$.
\begin{lemma}
  \label{lemma:x-transpose-A-nice}
  Assume that $\norm{x_\tau}$ is non-decreasing in $\tau$ for
  $\tau \le t$, that $\norm{x_t} \le \Rupper$, and that
  $x_t^{\T}\nabla f(x_t) \le 0$. Then $x_t^{\T}A\nabla f(x_t) \ge \beta 
  x_t^{\T}\nabla 
  f(x_t)$.
\end{lemma}
\begin{proof}
  If we define $z_t \supind{i}=x_t \supind{i}/(-\eta b\supind{i})$, then
  evidently
  \begin{equation*}
    z_{t + 1}\supind{i}
    = (1 - \underbrace{\eta\lambda\supind{i}(A)}_{
      \eqdef\chi\supind{i}}
    -\underbrace{\eta\rho\norm{ x_{t}}}_{
      \eqdef \nu_{t}}) z_{t}\supind{i} + 1.
  \end{equation*}
  We verify that $z_t\supind{i}$ satisfies
  the conditions of Lemma~\ref{lem:ordering}
  (if $b\supind{i} = 0$ then Assumption~\ref{assu:init} means that 
  $x\supind{i}_t=0$ for all $t$ so you may ignore it):
  \begin{enumerate}[label=(\roman*), leftmargin=*]
  \item By definition $\chi\supind{i}$ are increasing in $i$, and
    $\nu_{0}\leq\nu_{1}\leq\cdots\leq\nu_{t}$ by our assumption
    that $\norm{ x_\tau} $ is non-decreasing for $\tau \le t$.
  \item As
    $\eta\leq1/\left(\beta+\rho \Rupper\right)$
    for $\tau \leq t$, we have that $\chi\supind{i}+\nu_\tau \leq1$ for
    $\tau \le t$ and $i \in [d]$.
  \item As $x_{0}=-r b/\norm{ b}$, $z_{0}\supind{i}=r/(\eta\norm{ b}) \geq0$
    for every $i$.
  \end{enumerate}
  We may therefore apply Lemma~\ref{lem:ordering},
  part~\ref{item:bigger-signs} to conclude that
  $z_{t}\supind{i}-z_{t + 1}\supind{i}\geq0$ implies
  $z_{t}\supind{j}-z_{t + 1}\supind{j}\geq0$ for every $j\geq i$. Since
  $z_{t}\supind{i}\geq0$ for every $i$,
  \begin{equation*}
  \mathrm{sign}\left(x_{t}\supind{i}\left(x_{t}\supind{i}-x_{t + 1}\supind{i}\right)\right)=\mathrm{sign}\left(z_{t}\supind{i}\left(z_{t}\supind{i}-z_{t + 1}\supind{i}\right)\right)=\mathrm{sign}\left(z_{t}\supind{i}-z_{t + 1}\supind{i}\right),
  \end{equation*}
  and there must thus exist some $i^{*}\in[d]$ such that
  $x_{t}\supind{i}(x_{t}\supind{i}-x_{t + 1}\supind{i})\leq0$ for every
  $i\leq i^{*}$ and
  $x_{t}\supind{i}(x_{t}\supind{i}-x_{t + 1}\supind{i})\geq0$ for every
  $i>i^{*}$. We thus have (by expanding
  in the eigenbasis of $A$) that
  \begin{alignat*}{1}
    & x_{t}^{\T}A\grad f\left(x_{t}\right)  
    =\frac{1}{\eta}\sum_{i=1}^{i^{*}}\lambda\supind{i}\left(A\right)x_{t}\supind{i}\left(x_{t}\supind{i}-x_{t
     + 
    1}\supind{i}\right)+\frac{1}{\eta}\sum_{i=i^{*}+1}^{d}\lambda\supind{i}\left(A\right)x_{t}\supind{i}\left(x_{t}\supind{i}-x_{t
     + 1}\supind{i}\right)\\
    & \quad 
    \geq\lambda^{\left(i^{*}\right)}\left(A\right)\frac{1}{\eta}\sum_{i=1}^{i^{*}}x_{t}\supind{i}\left(x_{t}\supind{i}-x_{t
     + 
    1}\supind{i}\right)+\lambda^{\left(i^{*}+1\right)}\left(A\right)\frac{1}{\eta}\sum_{i=i^{*}+1}^{d}x_{t}\supind{i}\left(x_{t}\supind{i}-x_{t
     + 1}\supind{i}\right)\\
    & \quad 
    \geq\lambda^{\left(i^{*}\right)}\left(A\right)\frac{1}{\eta}\sum_{i=1}^{d}x_{t}\supind{i}\left(x_{t}\supind{i}-x_{t
     + 1}\supind{i}\right)=\lambda^{\left(i^{*}\right)}\left(A\right)x_{t}^{\T}\grad 
    f\left(x_{t}\right)\geq \beta x_{t}^{\T}\grad f\left(x_{t}\right)
  \end{alignat*}
  where the first two inequalities use the fact the $\lambda\supind{i}$ is
  non-decreasing with $i$, and the last inequality uses our assumption that
  $x_{t}^{\T}\nabla f\left(x_{t}\right)\leq0$ along with
  $\lambda^{\left(d\right)}\left(A\right)\leq \beta$.  
\end{proof}

\subsection{Proof of Lemma \ref{lemma:monotone}}
\label{sec:finally-proof-monotone-weak}

\newcommand{\Rlow}{R_{\rm low}}
\newcommand{\hessremain}{\Delta}

By definition of the gradient descent iteration~\eqref{eq:grad-iter},
\begin{equation}
  \label{eqn:norm-one-step}
  \norm{ x_{t+1}}^{2}=\norm{ x_{t}}^{2}-2\eta x_{t}^{\T}\grad f\left(x_{t}\right)+\eta^{2}\norm{ \grad f\left(x_{t}\right)}^{2},
\end{equation}
and therefore if we can show that $x_{t}^{\T}\grad f\left(x_{t}\right)\leq0$
for all $t$, the lemma holds.  We give a proof by induction.  The basis of
the induction $x_{0}^{\T}\grad f\left(x_{0}\right)\leq0$ is immediate as
$r \mapsto f(-r b / \norm{b})$ is decreasing until
$r = \Rcauchy$
(recall the definition~\eqref{eq:Rc-def}), and $x_0^T \nabla f(x_0) = 0$ for
$r \in \{0, \Rcauchy\}$.  Our induction assumption is that
$x_{t'-1}^{\T}\grad f\left(x_{t'-1}\right)\leq0$ (and hence also
$\norm{ x_{t'}} \geq\norm{ x_{t'-1}} $) for  $t'\leq t$ and we wish to show
that $x_{t}^{\T}\grad f\left(x_{t}\right)\leq0$.
Note that
\begin{equation*}
  x^{\T}\grad f\left(x\right)=x^{\T}Ax+\rho\norm{ x}^{3}+b^{\T}x\geq\rho\norm{ x}^{3}
  + \lmin \norm{ x}^{2}-\norm{ b} \norm{ x} 
\end{equation*}
and therefore $x^{\T}\grad f(x)>0$ for every
$\norm{x} > \Rlow\defeq
\frac{-\lmin}{2\rho}+\big[(\frac{\lmin}{2\rho})^{2} +
  \frac{\norms{b} }{\rho}\big]^{1/2}$.
Therefore, our induction assumption also implies
$\norm{ x_{t'-1}} \leq \Rlow\leq \Rupper$ for  $t'\leq t$.

Using that $\nabla^2 f$ is $2\rho$-Lipschitz,
a Taylor
expansion immediately
implies~\cite[Lemma~1]{NesterovPo06} that for all vectors $\Delta$, we have
\begin{equation}
  \norm{\nabla f(x + \Delta)
    - (\nabla f(x) + \nabla^2 f(x) \Delta)} \le
  \rho \norm{\Delta}^2.
  \label{eqn:lip-hess-bound}
\end{equation}
Thus, if we define 
$\hessremain_t \defeq \frac{1}{\eta^2} [\nabla f(x_t) - (\nabla f(x_{t-1}) -
  \eta \nabla^2 f(x_{t-1}) \nabla f(x_{t-1}))]$, we have
$\norm{\hessremain_t} \le \rho \norm{\nabla f(x_{t-1})}^2$, and using
the iteration $x_t = x_{t-  1} - \eta \nabla f(x_{t-1})$ yields
\begin{align}
  x_{t}^{\T}\grad f\left(x_{t}\right)
  & =x_{t-1}^{\T}\grad f\left(x_{t-1}\right)
  -\eta\norm{ \grad f\left(x_{t-1}\right)}^2
  - \eta \underbrace{x_{t-1}^{\T}\grad^{2}f\left(x_{t-1}\right)
    \grad f\left(x_{t-1}\right)}_{\eqdef \mc{T}_1} \nonumber \\
  & \qquad+\eta^{2}
  \underbrace{\grad f\left(x_{t-1}\right)^T 
    \grad^{2}f\left(x_{t-1}\right)\grad f\left(x_{t-1}\right)}_{\eqdef \mc{T}_2}
  +\eta^{2} \underbrace{x_{t}^{\T}\hessremain_t}_{\eqdef \mc{T}_3}.
  \label{eqn:three-fun-terms}
\end{align}
We bound each of the terms $\mc{T}_i$ in turn. We have that
\begin{alignat*}{1}
  \mc{T}_1 &= x_{t-1}^{\T}\grad^{2}f\left(x_{t-1}\right)\grad 
  f\left(x_{t-1}\right)=x_{t-1}^{\T}A\grad f\left(x_{t-1}\right)+2\rho\norm{ 
    x_{t-1}} 
  x_{t-1}^{\T}\grad f\left(x_{t-1}\right) \\
  & \qquad \ge (\beta + 2\rho \norm{x_{t-1}})x_{t-1}^{\T}\grad 
  f(x_{t-1}) \ge (\beta + 2\rho \Rupper)x_{t-1}^{\T}\grad 
  f(x_{t-1}),
  \label{eqn:x-transpose-A-expansion}
\end{alignat*}
where both inequalities follow from the induction assumption; the first is 
Lemma~\ref{lemma:x-transpose-A-nice} and the second is due to 
$\norm{x_{t-1}} \le \Rupper$ and $x_{t-1}^{\T}\grad 
f(x_{t-1}) \le 0 $.

Treating the second order term $\mc{T}_2$, we obtain that
\begin{alignat*}{1}
  \mc{T}_2  \le \opnorm{\hess f(x_{t-1})} 
  \norm{\grad f(x_{t-1})}^2 \leq\left(\beta+2\rho \Rupper\right)
  \norm{ \grad f\left(x_{t-1}\right)}^{2},
\end{alignat*}
and, by the Lipschitz bound~\eqref{eqn:lip-hess-bound},
the remainder term $\mc{T}_3$ satisfies
\begin{align*}
  \mc{T}_3 = x_{t}^{\T} \hessremain_t &
  \leq\norm{ x_{t}} \norm{ r} \leq\rho\norm{ x_{t}}
  \norm{ \grad f\left(x_{t-1}\right)}^{2}
  \leq \rho\norm{ x_{t-1}-\eta\grad f\left(x_{t-1}\right)} \norm{ \grad 
    f\left(x_{t-1}\right)}^{2}\\
  & \leq\rho\norm{ x_{t-1}} \norm{ \grad f\left(x_{t-1}\right)}^{2}+\rho\eta\norm{ \grad f\left(x_{t-1}\right)}^{3}.
\end{align*}
Using that $\norm{\nabla f(x)} = \norm{\nabla f(x) - \nabla f(\scu)} \le
(\beta + 2 \Rupper) \norm{x - \scu}
\le \Rupper(\beta + 2 \rho \Rupper)$ for $\norm{x}
\le \Rupper$ and that $\eta\leq1/2\left(\beta+2\rho \Rupper\right)$, our inductive
assumption that $\norm{x_{t-1}} \le \Rupper$ thus guarantees that $\mc{T}_3 \le
2 \rho \Rupper \norm{\nabla f(x_{t-1})}^2$.  Combining our bounds on the terms
$\mc{T}_i$ in expression~\eqref{eqn:three-fun-terms}, we have that
\begin{equation*}
  x_{t}^{\T}\grad f\left(x_{t}\right)
  \leq\left(1-\eta\left(\beta +2\rho \Rupper\right)\right)x_{t-1}^{\T}\grad f\left(x_{t-1}\right)
  -\left(\eta - \eta^{2}(\beta+4\rho \Rupper)\right) \norm{ \grad f(x_{t-1})}^{2}.
\end{equation*}
Using $\eta\leq1/\left(\beta+4\rho \Rupper\right)$ shows that
$x_{t}^{\T}\grad f\left(x_{t}\right)\leq0$, completing our induction.
By the expansion~\eqref{eqn:norm-one-step}, we have
$\norm{x_t} \le \norm{x_{t+1}}$ as desired, and that
$x_t^T \grad f(x_t) \le 0$ for all $t$ guarantees that $\norm{x_t} \le \Rlow 
\le \Rupper$.

It remains to argue that $\lim_{t\to \infty}\norm{x_t}$ (which necessarily 
exists) is at most $\norm{\scu}$. To see this note that $x_t$ converges to 
a stationary point $\hat{x}$: the proof of Proposition~\ref{prop:converge} 
 shows this using only the bound $\norm{x_t} \le \Rupper$ and without  
the assumption $b\supind{1}\ne 0$. By 
Proposition~\ref{prop:characterization} every stationary point can have 
norm at most $\norm{\scu}$, and consequently we have that $\norm{x_t} 
\le \norm{\hat{x}} \le \norm{\scu}$ for all $t$. Finally, we have that $f$ is 
$(\beta+2\rs)$-smooth on a ball of radius $\norm{\scu}$, since 
$\norm{\hess f(x)}\le \beta + 2\rho\norm{x}$. 

\section{Proof of Theorem~\ref{theorem:gradient-descent-cubic}}
\label{sec:proof-gd-cubic}
Throughout the proof, let $f(x) = \half x^T A x + b^T x + \frac{\rho}{3}
\norm{x}^3$ for short.  A number of the steps of the proof of
Theorem~\ref{theorem:gradient-descent-cubic} involve technical lemmas whose
proofs we defer. In all lemma statements, we tacitly
let Assumptions~\ref{assu:init} and~\ref{assu:step-size} hold as in the
theorem. We assume $\eps \le \half
\beta \norm{\sol}^2 + \rho \norm{\sol}^3$ w.l.o.g., as $f$ is $\beta + 2 \rho
\norm{\sol}$ smooth on the set $\{x : \norm{x} \le \norm{\sol}\}$ and
therefore $f(x_0) \le f(\sol) + \eps$ for any $\eps \ge \half \beta
\norm{\sol}^2 + \rho \norm{\sol}^3$.
We divide the proof of Theorem~\ref{theorem:gradient-descent-cubic} into two
main steps: in Section~\ref{sec:linear-convergence-part} we prove
the linear convergence
case of the theorem, and in
Section~\ref{sec:epsilon-convergence-part} we prove the
sublinear convergence result.

\subsection{Linear convergence and exponential growth}
\label{sec:linear-convergence-part}

We first prove that $f(x_t) \le f(\sol) + \eps$ for $t \ge \frac{1}{\eta (\rho
  \norm{\sol} + \lmin)}(\tgrow + \tconv(\eps))$. We begin with two lemmas
that provide regimes in which $x_t$ converges to the solution $\sol$ linearly.

\begin{lemma}
  \label{lem:lin-converge}
  For each $t > 0$, we have
  \begin{equation*}
    \norm{x_t - \sol}^2 \le 
    \left(1-\eta \left[\rho\norm{x_t}-\left(-\lmin - 
      \frac{\rs + \lmin}{2}\right)\right]\right)
    \norm{x_{t-1}-\sol}^2
  \end{equation*}
\end{lemma}
\noindent
We defer the technical proof of this
lemma to Sec.~\ref{sec:proof-lin-converge}.

For nonconvex problem instances (those with $\lmin < 0$), the above
recursion is a contraction (implying linear convergence of $x_t$ to $\sol$)
only when $\rho\norm{x_t}$ is larger than $-\lmin - \half(\rs + \lmin)$.
Using the fact that $\|x_t\|$ is non-decreasing
(Lemma~\ref{lemma:monotone}), Lemma~\ref{lem:lin-converge} immediately
implies the following result.
\begin{lemma}
  \label{lem:really-lin-converge}
  Let $\mu \ge 0$. If $\rho\| x_{t}\| \geq -\lmin - \half(\rs + \lmin) + \mu$
  for some $t\geq0$, then for all $\tau\geq0$,
  \begin{equation*}
    \| x_{t+\tau}-\sol\|^2
    \leq\left(1-\eta\mu\right)^{\tau}\| 
    x_{t}-\sol\|^2\leq2\norm{\sol}^2e^{-\eta\mu\tau}.
  \end{equation*}
\end{lemma}
\begin{proof}
  Lemma~\ref{lem:lin-converge} implies that
  $\norm{x_{t + \tau} - \sol}^2 \le (1 - \eta \mu) \norm{x_{t + \tau - 1} -
    \sol}^2$
  for all $\tau > 1$. Using that
  $\norm{x_t - \sol}^2 \le \norm{x_t}^2 + \norm{\sol}^2 \le 2\norm{\sol}^2$
  by Lemmas~\ref{lemma:signs} and~\ref{lemma:monotone}, and
  $1 + \alpha \le e^\alpha$ for all $\alpha$ gives the result.
\end{proof}

It remains to understand whether the gradient descent iterations satisfy the 
condition $\rho\| x_{t}\| \ge 
-\lmin - \half (\rs + \lmin)+\mu$. Fortunately, as long 
as $\rho\| x_{t}\| $
is below $-\lmin - \smallval$, $|x_{t}\supind{1}|$ grows faster than 
$(1 + \eta \smallval)^t$:
\begin{restatable}{lemma}{lemGrowth}
  \label{lem:growth}
  Let $\smallval>0$. Then $\rho\| x_{t}\| \geq -\lmin-\smallval$ for all
  $t\geq\frac{2}{\eta\smallval}\log(1+\frac{\hinge{-\lmin}^2}{
    4\rho|b\supind{1}|})$.
\end{restatable}
\noindent
See Sec.~\ref{sec:proof-growth} for a proof of this lemma.

We combine the lemmas to give the linear convergence regime of
Theorem~\ref{theorem:gradient-descent-cubic}:
Lemma~\ref{lem:growth} with $\smallval=\frac{1}{3}(\rho\norm{\sol} + \lmin)$
yields $\rho\norm{x_{t}} \geq -\lmin-\frac{1}{3}\left(\rho\norm{\sol} +
\lmin\right)$ for
\begin{equation*}
  t\geq T_{1}\triangleq\frac{6}{\eta\left(\rho\norm{\sol}
    + \lmin\right)}\log\left(1+\frac{\hinge{-\lmin}^2}{
    4\rho|b^{\left(1\right)}|}\right)
  \IndNC
  = \frac{1}{\eta (\rho \norm{\sol} + \lmin)} \tgrow.
\end{equation*}
Therefore, by Lemma~\ref{lem:really-lin-converge} with $\mu = 
\half(\rs+ \lmin)-\smallval=
\frac{1}{6}(\rs+ \lmin)$, for any $t$ we have
\begin{equation}
  \label{eqn:eventual-linear-convergence}
  \norm{x_{T_{1}+t} - \sol}^{2}
  \leq2\norm{\sol}^{2} \exp\Big(-\frac{1}{6}\eta\left(\rho\norm{\sol} 
  + \lmin\right)t\Big).
\end{equation}
As a consequence, for all $t \ge 0$ we may use the
$(\beta + 2 \rho \norm{\sol})$-smoothness of $f$ and the fact that
$\norm{x_t}\le \norm{\sol}$ (by Lemma~\ref{lemma:monotone}) to
obtain
\begin{equation*}
  f\left(x_{t}\right)-f\left(\sol\right)\leq\frac{\beta+2\rho\norm{\sol}
  }{2}\norm{x_{t}-\sol}^{2}
  \le (\beta + 2 \rho \norm{\sol}) \norm{\sol}^2
  e^{-\frac{\eta}{6} (\rho \norm{\sol} + \lmin) (t - T_1)}
\end{equation*}
where we have used that $\nabla f(\sol) = 0$ and the
bound~\eqref{eqn:eventual-linear-convergence}.  Therefore, if we set
\begin{equation*}
  T_2 \defeq
  \frac{6}{\eta\left(\rho\norm{\sol} + \lmin\right)}
  \log \frac{(\beta + 2 \rho \norm{\sol}) \norm{\sol}^2}{\eps}
  = \frac{1}{\eta (\rho \norm{\sol} + \lmin)} \tconv(\eps),
\end{equation*}
then $t \ge T_1 + T_2 = \frac{1}{\eta (\rho \norm{\sol} +
  \lmin)}(\tgrow + \tconv(\eps))$ implies $f(x_t) - f(\sol) \le
\eps$.

\subsection{Sublinear convergence and convergence in subspaces}
\label{sec:epsilon-convergence-part}

We now turn to the sublinear convergence regime in
Theorem~\ref{theorem:gradient-descent-cubic},
which applies when the quantity $\rs + \lmin$ is sufficiently small that
\begin{equation}
  \label{eqn:assumption-bad-condition}
  \rho \norm{\scu} + \lmin \le \frac{\eps}{10 \norm{\scu}^2}~.
\end{equation}
If~\eqref{eqn:assumption-bad-condition} fails to hold, the $(\rs +
\lmin)^{-1}$ term dominates the convergence guarantee in
Theorem~\ref{theorem:gradient-descent-cubic}.  Therefore,
to complete the proof of Theorem~\ref{theorem:gradient-descent-cubic} it
suffices to show that if~\eqref{eqn:assumption-bad-condition} holds, then
$f(x_t) \le f(\scu) + \eps$ whenever
\begin{equation}
\label{eqn:sublinear-main}
t \ge T^{\rm sub}_\eps \defeq \frac{\tgrow + \tconv(\eps)}{\eta}
\cdot\frac{10 \norm{\scu}^2}{\eps}.
\end{equation}

Our proof of the result~\eqref{eqn:sublinear-main} proceeds as
follows: when $\rs + \lmin$ is small, the function $f$ is smooth along
eigenvectors with eigenvalues close to $\lmin$. It is
therefore sufficient to show convergence in the complementary subspace, which
occurs at a linear rate. Appropriately choosing the gap between the
eigenvalues in the complementary subspace and $\lambda\supind{1}(A)$ to trade
between convergence rate and function smoothness yields the
rates~\eqref{eqn:sublinear-main}.

The following analogues of Lemmas~\ref{lem:lin-converge}
and~\ref{lem:really-lin-converge}
establish subspace convergence. (Recall the notation $\As = A + \rs I$.)
\begin{lemma}
  \label{lem:proj-lin-converge}
  Let $\Pi$ be any projection matrix satisfying $\Pi A = A \Pi$ for which
  $\Pi \As\succeq\smallval\Pi$ for some $\smallval>0$.
  For all $t>0$,
  \begin{align*}
    &\norms{ \Pi \As^{1/2}(x_{t}-\scu)}^2
     \leq\left(1-\eta\smallval\right)
      \norms{ \Pi \As^{1/2}(x_{t-1}-\scu)}^2\\
    & \qquad+\sqrt{8}\eta\rho\left(\norm{ \scu} -\norm{ x_{t-1}}
      \right)\left[\rho\left(\norm{ \scu} -\norm{ x_{t-1}} \right)\norm{
      x_{t-1}}^2+\opnorm{ (I -\Pi)\As} \norm{ \scu}^2\right].
  \end{align*}
\end{lemma}
\noindent
See Appendix~\ref{sec:proof-proj-lin-converge} for a proof. Letting
$\Pi_\smallval = \sum_{i : \lambda\supind{i} \ge \smallval + 
\lambda\supind{1}}
v_i v_i^T$
be the projection matrix onto the span of eigenvectors of $A$
with eigenvalues at least $\lambda\supind{1}(A) + \smallval$,
we obtain the following consequence of 
Lemma~\ref{lem:proj-lin-converge},
whose proof we provide in
Appendix~\ref{sec:proof-eventually-proj-lin-converge}.
\begin{lemma}
  \label{lemma:eventually-proj-lin-converge}
  Let $t\geq0$, $\smallval \ge 0$.
  If $\rs\le -\lmin+\smallval$
  and $\rho \norm{x_t} 
  \geq-\lmin-\frac{1}{3}\smallval$, then
  for any $\tau \ge 0$,
  \begin{align*}
    \norm{ \Pi_{\smallval}\As^{1/2}\left(x_{t+\tau}-\scu\right)}^2
    & \leq\left(1-\eta\smallval\right)^{\tau}\norm{
      \Pi_{\smallval}\As^{1/2}\left(x_{t}-\scu\right)}^2+13\norm{ 
      \scu}^2\smallval \\
    & \leq2\left(\beta+\rho\norm{ \scu} \right)\norm{ 
    \scu}^2e^{-\eta\smallval\tau}+13\norm{ \scu}^2\smallval.
  \end{align*}
\end{lemma}

\newcommand{\Tsublinone}{T_1^{\rm sub}}
\newcommand{\Tsublintwo}{T_2^{\rm sub}}
\newcommand{\Tsublin}{T^{\rm sub}}

We use these lemmas to prove the desired bound~\eqref{eqn:sublinear-main}
by appropriate separation of the eigenspaces over which we guarantee
convergence.
To that end, we define
\begin{equation}
  \label{eqn:delta-defs}
  \smallval \defeq \frac{\eps}{10 \norm{\scu}^2}.
\end{equation}
The growth that Lemma~\ref{lem:growth} guarantees shows that
$\rho\norm{x_{t}} \geq -\lmin-\frac{1}{3}\smallval$ 
for every
\begin{equation*}
  t \ge \Tsublinone \defeq
  \frac{6}{\eta\smallval}\log\left(1+\frac{\hinge{-\lmin}^2}{
      4\rho|b^{\left(1\right)}|}\right)
  = \frac{1}{\eta\smallval} \tgrow.
\end{equation*}
Thus,
using $(\beta + 2 \rs \norm{\scu}^2 / \eps \ge 2$ as in
the beginning of Appendix~\ref{sec:proof-gd-cubic}, we may define
\begin{equation*}
  \Tsublintwo \defeq
  \frac{1}{\eta\smallval}
  \log \frac{2 (\beta + \rho \norm{\scu})}{\smallval}
  \le \frac{1}{\eta\smallval}
  \log\left(\left[ \frac{ (\beta + 2\rs)\norm{\scu}^2}{\eps}\right]^6 \right)
  = \frac{\tconv(\eps)}{\eta\smallval}.
\end{equation*}
Thus $2(\beta + \rs)\norm{\scu}^2 e^{-\eta \smallval t} \le 
\norm{\scu}^2 
\smallval$ for $t \ge \Tsublintwo$, and
by Lemma~\ref{lemma:eventually-proj-lin-converge} we have 
\begin{equation}
  \norm{\Pi_{\smallval}\As^{1/2}(x_t - \scu)}^{2}\leq
  \norm{\scu}^2 \smallval + 13\norm{\scu}^2 \smallval = 
  14\norm{\scu}^{2} \smallval,
  \label{eqn:proj-bound}
\end{equation}
for $t \ge \Tsublin = \Tsublinone+\Tsublintwo$.

We now translate the guarantee~\eqref{eqn:proj-bound} on the distance from $x_t$ to
$\scu$ in the subspace of ``large'' eigenvectors of $A$ to
a guarantee on the solution quality $f(x_t)$. Using
the expression~\eqref{eq:fx-s-expression} for $f(x)$, the
orthogonality of $I - \Pi_\smallval$ and $\Pi_\smallval$ and 
$\norm{x_t}\le\norm{\scu}$, we have
\begin{flalign*}
 & f(x_t) \le  f(\scu)+ \half \norm{(I - \Pi_{\smallval})\As^{\frac{1}{2}}(x_t - 
 \scu)}^{2}
    + \half\norm{\Pi_{\smallval}\As^{\frac{1}{2}}(x_t - \scu)}^{2}
    \\ & \quad \quad \quad
    +
    \frac{\rho\norm{\scu} }{2}(\norm{\scu} -\norm{x_t} )^{2}.
\end{flalign*}
Now we note that
\begin{equation}\label{eq:subpsace-opnorm-bound}
  \opnorm{(I - \Pi_\smallval)\As}
  = \max_{i : \lambda\supind{i} < \lambda\supind{1} + \smallval}
  |\lambda\supind{i} + \rho \norm{\scu}|
  \le \lmin + \smallval + \rho \norm{\scu}
  \le 2 \smallval,
\end{equation}
where we have used our assumption~\eqref{eqn:assumption-bad-condition}
that $\rho \norm{\scu} + \lmin \le \frac{\eps}{10 \norm{\scu}^2} = \smallval$.
Using this gives
\begin{equation*}
  f(x_t) \le f(\scu) + \smallval \norm{x_t - \scu}^2
  + 7 \norm{\scu}^2 \smallval
  + \frac{\rho \norm{\scu}}{2} (\norm{\scu} - \norm{x_t})^2,
\end{equation*}
where we use inequality~\eqref{eqn:proj-bound}.  Because
$\rho \norm{x_t} \ge -\lmin - \frac{1}{3}\smallval$
for $t \ge \Tsublinone$, we obtain
\begin{equation*}
  0 \le \rho (\norm{\scu} - \norm{x_t})  \le
  \rho \norm{\scu} + \lmin - (\rho \norm{x_t} + \lmin)
  \le \frac{4}{3}\smallval.
\end{equation*}
Substituting back into~\eqref{eq:subpsace-opnorm-bound} and using  
$\norm{x_t}\le\norm{\scu}$ 
(Lemma~\ref{lemma:monotone}) gives
\begin{equation*}
  f(x_t)\le f(\scu) + 9 \norm{\scu}^2 \smallval
  \le  f(x_t) + \eps,
\end{equation*}
where we substitute $\smallval = \frac{\eps}{10 \norm{\scu}^2}$.
Summarizing, if $\rho \norm{\scu} + \lmin \le \smallval = \frac{\eps}{10
  \norm{\scu}^2}$, then $x_t$ is $\eps$-suboptimal for
\eqref{eqn:problem-cubic} whenever $t \ge \Tsublinone +
\Tsublintwo$, i.e., inequality~\eqref{eqn:sublinear-main} holds.
\subsection{Proof of Lemma~\ref{lem:lin-converge}}
\label{sec:proof-lin-converge}

Expanding $x_t = x_{t - 1} - \eta \nabla f(x_{t-1})$, we have
\begin{equation}
  \label{eqn:lyapunov-distance}
  \norm{x_{t}-\sol}^2
  =\norm{x_{t-1}-\sol}^2 - 2\eta\left(x_{t-1}-\sol\right)^{\T} \grad f(x_{t-1})
  + \eta^2\norm{\grad f\left(x_{t-1}\right)}^2.
\end{equation}
Using the equality
$\grad f\left(x\right)=\As\left(x-\sol\right)-\rho\left(\norm{\sol} -\norm{x}
\right)x$, we rewrite the cross-term  $\left(x_{t-1}-\sol\right)^{\T}
\grad
f\left(x_{t-1}\right)$ as
\begin{flalign}
  \nonumber
  \left(x_{t-1}-\sol\right)^{\T}&\As 
  \left(x_{t-1}-\sol\right)+\rho\left(\norm{x_{t-1}}
  -\norm{\sol} \right)(\norm{x_{t-1}}^2 - x_{t-1}^T \sol) \\ \nonumber
  = & 
  \left(x_{t-1}-\sol\right)^{\T}\left(\As+\frac{\rho}{2}\left(\norm{x_{t-1}}
    -\norm{\sol} \right) I \right) \left(x_{t-1}-\sol\right) \\
   & +\frac{\rho}{2}\left(\norm{\sol} -\norm{x_{t-1}} 
   \right)^2\left(\norm{x_{t-1}} 
    +\norm{\sol} \right).
    \label{eqn:cross-term-lyapunov}
\end{flalign}
Moving to the second order term $\norm{\grad f\left(x_{t-1}\right)}^2$
from the expansion~\eqref{eqn:lyapunov-distance},
we find
\begin{alignat*}{1}
  \norm{\grad f\left(x_{t-1}\right)}^2 & 
  =\norm{\As\left(x_{t-1}-\sol\right)+\rho\left(\norm{x_{t-1}} -\norm{\sol} 
  \right)x_{t-1}}^2\\
  & 
  \leq2\left(x_{t-1}-\sol\right)^{\T}\As^2\left(x_{t-1}-\sol\right)+2\rho^2\left(\norm{x_{t-1}}
   -\norm{\sol} \right)^2\norm{x_{t-1}}^2.
\end{alignat*}
Combining this inequality with the cross-term
calculation~\eqref{eqn:cross-term-lyapunov}
and the squared distance~\eqref{eqn:lyapunov-distance} we obtain
\begin{alignat*}{1}
  \norm{x_{t}-\sol}^2 & \le
  (x_{t-1}-\sol)^{\T}(I-2\eta \As(I-\eta \As)-\eta\rho(\norm{x_{t-1}} 
  -\norm{\sol} ) I )
  (x_{t-1}-\sol) \\
  & \qquad-\eta\rho\left(\norm{\sol} -\norm{x_{t-1}} \right)^2\left(\norm{x_{t-1}} \left(1-2\eta\rho\norm{x_{t-1}} \right)+\norm{\sol} \right).
\end{alignat*}
Using $\eta\leq\frac{1}{4\left(\beta+\rho 
R\right)}\leq\frac{1}{4\opnorm{\As}}$ yields 
$2\eta \As\left(1-\eta \As\right)\succeq\frac{3}{2}\eta
\As\succeq\frac{3}{2}\eta \left(\lmin + \rho\norm{\sol} \right)I$, so
\begin{align*}
  \norm{x_{t}-\sol}^2 & \leq
  \left(1-\frac{\eta}{2}\left[3\lmin +
    \rho\left(\norm{\sol} +2\norm{x_{t-1}} \right)\right]\right)\norm{x_{t-1}-\sol}^2 \\
  & \qquad ~ -\eta\rho\left(\norm{\sol} -\norm{x_{t-1}} \right)^2\norm{\sol}.
\end{align*}

\subsection{Proof of Lemma~\ref{lem:growth}}
\label{sec:proof-growth}
 
The claim is trivial when $\lmin \ge 0$, as clearly $\rho \norm{x_t} 
\ge 0$, so we assume $\lmin < 0$. Using
Proposition~\ref{prop:converge} that gradient descent is
convergent, we may define $t\opt = \max\{t : \rho \norm{x_t} \le -\lmin
- \smallval\}$.
Then for every $t\leq t\opt$, the gradient descent 
iteration~\eqref{eq:grad-iter} 
satisfies
\begin{align*}
  \frac{x_{t}^{\left(1\right)}}{-\eta b^{\left(1\right)}}
  & =\left(1 - \eta\lmin - \eta\rho\norm{x_{t-1}} 
  \right)\frac{x_{t-1}^{\left(1\right)}}{-\eta b^{\left(1\right)}}+1 \\ & 
  \geq\left(1+\eta\smallval\right)\frac{x_{t-1}^{\left(1\right)}}{-\eta 
  b^{\left(1\right)}}+1
  \geq \cdots 
  \geq\frac{1}{\eta\smallval}\left(\left(1+\eta\smallval\right)^{t}-1\right).
\end{align*}
Multiplying both sides of the equality by $\eta |b\supind{1}|$ and
using that $x_t\supind{1} b\supind{1} \le 0$, we have
\begin{equation*}
  \frac{-\lmin-\smallval}{\rho}\geq\norm{x_{t\opt}} \geq 
  |x_{t\opt}^{\left(1\right)}|\geq\frac{|b^{\left(1\right)}|}{\smallval}\left(\left(1+\eta\smallval\right)^{t\opt}-1\right).
\end{equation*}
Consequently,
\begin{equation*}
  t\opt\leq\frac{\log\left(1+\frac{(-\lmin-\smallval)\smallval}{\rho|b\supind{1}|}\right)}{
    \log(1+\eta\smallval)}
  \leq\frac{2}{\eta\smallval}\log\left(1+\frac{\hinge{-\lmin}^2}{4\rho|b^{\left(1\right)}|}\right),
\end{equation*}
where we used
$\eta\smallval \leq -\eta\lmin \leq -\lmin/\beta \leq1$, whence
$\log(1 + \eta \smallval) \ge \frac{\eta \smallval}{2}$, and
$-\lmin \smallval - \smallval^2 \le \sup_{x\ge0}\{- x(\lmin + x)\} \le 
\frac{ \hinge{-\lmin}^2 }{ 4} $.

\subsection{Proof of Lemma~\ref{lem:proj-lin-converge}}
\label{sec:proof-proj-lin-converge}

For typographical convenience, we prove the result with $t + 1$ replacing $t$.
Using the commutativity of $\Pi$ and $A$, we have $\Pi \As=\As\Pi$, so
\begin{flalign}
  \nonumber \norm{\Pi \As^{1/2}\left(x_{t + 1}-\scu\right)}^2
  = & \norm{\Pi
    \As^{1/2}\left(x_t-\scu\right)}^2\\ &-2\eta\left(x_t-\scu\right)^T \!\!
  \As\Pi
  \grad f\left(x_t\right)+\eta^2\norm{\Pi \As^{1/2}\grad
    f\left(x_t\right)}^2.
  \label{eqn:proj-lyapunov}
\end{flalign}
We substitute
$\grad f\left(x\right)=\As\left(x-\scu\right)-\rho\left(\norm{\scu} -\norm{x}
\right)x$ in the cross term to obtain
\begin{align*}
  \lefteqn{\left(x_t-\scu\right)^{\T} \Pi \As\grad f\left(x_t\right)} \\
  & \qquad ~ = \left(x_t-\scu\right)^{\T}\Pi \As^2\Pi\left(x_t-\scu\right)
  - \rho\left(\norm{\scu} -\norm{x_t} \right)x_t^{\T}\Pi \As\left(x_t-\scu\right).
\end{align*}
Substituting
$\As\left(x-\scu\right)=\grad f\left(x\right)+\rho\left(\norm{\scu} -\norm{x}
\right)x$ in the last term yields
\begin{equation}
  x_t^{\T}\Pi \As\left(x_t-\scu\right)=x_t^{\T}\Pi\grad
  f\left(x_t\right)+\rho\left(\norm{\scu} -\norm{x_t} \right)\norm{\Pi
    x_t}^2.
  \label{eqn:proj-cross-term}
\end{equation}
Invoking Lemma~\ref{lemma:monotone} and the fact that
$x_t^{\T}\grad f\left(x_t\right)\leq0$, we get
\begin{alignat*}{1}
  x_t^{\T}\Pi\grad f\left(x_t\right) & =x_t^{\T}\grad f\left(x_t\right)-x_t^{\T}\left(I-\Pi\right)\grad f\left(x_t\right)\\
  & \leq-x_t^{\T}\left(I-\Pi\right)\grad f\left(x_t\right)\\
  & =-x_t^{\T}\left(I-\Pi\right)\As\left(x_t-\scu\right)+\rho\left(\norm{\scu} 
  -\norm{x_t} \right)\norm{\left(I-\Pi\right)x_t}^2\\
  & \leq\opnorm{\left(I-\Pi\right)\As} \norm{x_t} \norm{x_t-\scu} 
  +\rho\left(\norm{\scu} -\norm{x_t} \right)\norm{\left(I-\Pi\right)x_t}^2\\
  & \leq\sqrt{2}\opnorm{\left(I-\Pi\right)\As} 
  \norm{\scu}^2+\rho\left(\norm{\scu} -\norm{x_t} 
  \right)\norm{\left(I-\Pi\right)x_t}^2,
\end{alignat*}
where in the last line we used $x_t^{\T}\scu \ge 0$ (by
Lemma~\ref{lemma:signs}). Combining this with
the cross terms~\eqref{eqn:proj-cross-term}, we find that
\begin{subequations}
\begin{alignat}{1}\label{eq:subspace-bound-1}
  x_t^{\T}\Pi \As\left(x_t-\scu\right) & \leq\sqrt{2}
  \opnorm{\left(I-\Pi\right)\As} \norm{\scu}^2
  +\rho\left(\norm{\scu} -\norm{x_t} \right)\norm{x_t}^2.
\end{alignat}
Moving on to the second order term in the expansion~\eqref{eqn:proj-lyapunov},
we have
\begin{alignat}{1}
  \norm{\Pi \As^{1/2}\grad f\left(x_t\right)}^2 & =\norm{\Pi 
  \As^{3/2}\left(x_t-\scu\right)+\rho\left(\norm{x_t} -\norm{\scu} 
  \right)\As^{1/2}\Pi x_t}^2 \nonumber\\
  & \leq2\norm{\Pi \As^{3/2}\left(x_t-\scu\right)}^2+2\rho^2\opnorm{\Pi 
  \As} 
  \left(\norm{x_t} -\norm{\scu} \right)^2\norm{x_t}^2. 
  \label{eq:subspace-bound-2}
\end{alignat}
\end{subequations}
Substituting the bounds~\eqref{eq:subspace-bound-1} 
and~\eqref{eq:subspace-bound-2} into the
expansion~\eqref{eqn:proj-lyapunov},
we have
\begin{align*}
 \norm{\Pi \As^{1/2}\left(x_{t + 1}-\scu\right)}^2 
\leq & \left(x_t-\scu\right)^{\T}\left(I-2\eta\Pi \As\left(I-\eta\Pi 
\As\right)\right)\Pi 
\As\left(x_t-\scu\right)\\
&  +2\eta\rho\left(\norm{\scu} -\norm{x_t} 
  \right)\left[\sqrt{2}\opnorm{\left(I-\Pi\right)\As}\norm{\scu} 
  ^2\right. \\ & \qquad\qquad
  + \left. \left(1+\eta\opnorm{\Pi \As} \right)\rho\left(\norm{x_t} 
  -\norm{\scu} \right)\norm{x_t}^2\right].
\end{align*}
Using $\eta\leq1/(4\left(\beta+\rho R\right))$, which guarantees
$0 \preceq \eta\Pi \As\preceq I/4 \prec I/2$, together with the 
assumption that $\Pi \As\succeq\smallval\Pi$ gives
\begin{equation*}
0 \preceq I-2\eta\Pi \As\left(I-\eta\Pi 
\As\right) \preceq (1-\eta\smallval)I
\end{equation*} 
and therefore
\begin{alignat*}{1}
&\norm{\Pi \As^{1/2}\left(x_{t + 1}-\scu\right)}^2  
\leq\left(1-\eta\smallval\right)\norm{\Pi 
	\As^{1/2}\left(x_t-\scu\right)}^2\\
& \quad\quad\quad\quad+\sqrt{8}\eta\rho\left(\norm{\scu} -\norm{x_t} 
\right)\left[\rho\left(\norm{\scu} -\norm{x_t} \right)\norm{x_t}^2
+\opnorm{\left(I-\Pi\right)\As}\norm{\scu}^2\right].
\end{alignat*}

\subsection{Proof of Lemma~\ref{lemma:eventually-proj-lin-converge}}
\label{sec:proof-eventually-proj-lin-converge}

The conditions of the lemma imply that for $\tau\geq0$,
\begin{equation*}
  \rho(\norm{\scu} -\norm{x_{t+\tau}}) \leq 4\smallval/3
\end{equation*}
and also that
$\opnorm{\left(I-\Pi_{\smallval}\right)\As}\leq2\smallval$ 
(Eq.~\eqref{eq:subpsace-opnorm-bound}) 
and
$\Pi_{\smallval}\As\succeq\smallval \Pi_{\smallval}$. Substituting these 
bounds 
into
Lemma~\ref{lem:proj-lin-converge} along with $\norm{x_{t-1}} \leq\norm{\scu} $
(Lemma~\ref{lemma:monotone}), we get
\begin{alignat*}{1}
  \norm{\Pi_{\smallval}\As^{1/2}\left(x_{t+\tau}-\scu\right)}^2 & \leq 
  \left(1-\eta\smallval\right) 
  \norm{\Pi_{\smallval}\As^{1/2}\left(x_{t+\tau-1}-\scu\right)}^2+ 
  13\eta\smallval\smallval\norm{\scu}^2.
\end{alignat*}
Iterating this $\tau$ times gives
\begin{alignat*}{1}
  \norm{\Pi_{\smallval}\As^{1/2}(x_{t+\tau}-\scu)}^2
  & \leq
  (1-\eta\smallval)^{\tau}
  \norm{\Pi_{\smallval}\As^{1/2}(x_{t}-\scu)}^2
  + 13\smallval\norm{\scu}^2(1-(1-\eta\smallval)^{\tau})\\
  & \leq
  2 (\beta+\rho\norm{\scu})\norm{\scu}^2
  e^{-\eta\smallval\tau}+13\norm{\scu}^2\smallval
\end{alignat*}
where the last transition uses that
\begin{equation*}
  \norm{\Pi_{\smallval}\As^{1/2}\left(x_{t}-\scu\right)}^2\leq\opnorm{\As} 
  \norm{x_{t}-\scu}^2\leq\left(\beta+\rho\norm{\scu} \right)2\norm{\scu}^2.
\end{equation*}

\newcommand{\Agam}{A_{0}}
\newcommand{\sgam}{\tilde{x}^\star_0}

\section{Proof of Theorem~\ref{theorem:tr-krylov}}
\label{sec:proof-krylov}
We begin with a few building blocks on polynomial approximation and
the convex trust region problem; see \cite[Appendix 
C.1]{CarmonDu18} for full
proofs, though the results are essentially standard
polynomial approximations.

\begin{restatable}[Approximate matrix inverse]{lemma}{lemLinCheby}
  \label{lem:lin-cheby}
  Let $\alpha,\beta$ satisfy $0 < \alpha \le \beta$, and let $\kappa = 
  \beta/\alpha$. For $t \ge 1$ there exists a polynomial $p$ 
  of degree at most $t-1$, such that for every $M$ satisfying $\alpha I 
  \preceq M \preceq \beta I$,
  \begin{equation*}
    \opnorm{I - Mp(M)} \le 2e^{-2t/\sqrt{\kappa}}.
  \end{equation*}
\end{restatable}
\begin{lemma}[Finding eigenvectors~{\cite[Thm.~4.2]{KuczynskiWo92}}]
  \label{lem:eigenvec-tight}
  Let $u \in \R^d$ be a unit vector and $M\succeq 0$ be such that $u^T M u = 0$,
  and let 
  $v\in\R^d$. For $t\ge1$ there exists $z_t \in \Krylov[t][M,v]$ such 
  that
  \begin{equation*}
    \norm{z_t} = 1
    ~~\mbox{and}~~
    z_t^T M z_t \le \frac{\opnorm{M}}{16(t-\half)^2}
    \log^2\left(-2+4\frac{\norm{v}^2}{(u^T v)^2}\right).
  \end{equation*}
\end{lemma}
\noindent
The final preliminary result we require is based on a variant
of Nesterov's accelerated gradient method due to Tseng~\cite{Tseng08},
whose iterates lie in the Krylov subspace.
\begin{lemma}[Convex trust-region problem]
  \label{lem:agd}
  Let $t\ge1$, $M\succeq 0$, $v\in \R^d$ and $r\ge 0$, and let 
  $\f[M,v](x) 
  =\half x^T M x + v^T x$. There exists $x_t \in 
  \Krylov[t][M,v]$ such that 
  \begin{equation*}
    \norm{x_t} \le r
    ~~\mbox{and}~~
    \f[M,v](x_t) - \min_{\norm{x}\le r} \f[M,v](x) 
    \le \frac{4\lambda_{\max}(M) \cdot r^2}{(t+1)^2}.
  \end{equation*}
\end{lemma}

We can now provide the proof of Theorem~\ref{theorem:tr-krylov}. In the
proof, we let $\mc{P}_t$ denote all polynomials of degree
at most $t - 1$.

\subsection{Linear convergence}\label{sec:upper-lin-proof}
Recalling the notation $\As = A + \ltr I$, let $y_t = -p(\As)b = p(\As)\As 
\str$, for the $p\in\mc{P}_t$ which 
Lemma~\ref{lem:lin-cheby} guarantees to satisfy $\opnorm{p(\As)\As - 
  I}\le 2e^{-2t/\sqrt{\kappa(\As)}}$. Let
\begin{equation*}
  x_t = (1-\alpha) y_t,
  ~\mbox{where}~
  \alpha = \frac{\norm{y_t} - \norm{\str}}{\max\{\norm{\str},\norm{y_t}\}},
\end{equation*}
so that we are guaranteed $\norm{x_t}\le\norm{\str}$ for any value of 
$\norm{y_t}$. Moreover
\begin{equation*}
|\alpha| = 
\frac{|\norm{y_t}-\norm{\str}|}{\max\{\norm{\str},\norm{y_t}\}} \le 
\frac{\norm{y_t - \str}}{\norm{\str}} =
\frac{\norm{(p(\As)\As-I)\str}}{\norm{\str}} 
\le 2e^{-2t/\sqrt{\kappa(\As)}},
\end{equation*}
where the last transition used $\opnorm{p(\As)\As - 
	I}\le 2e^{-2t/\sqrt{\kappa(\As)}}$.

Since $b = -\As\str$, we have 
$\f[\As,b](x) = \f[\As,b](\str) + \half\norms{\As^{1/2}(x-\str)}^2$. 
The equality~\eqref{eq:tr-gamma-pivot-outline} with $\lambda = \ltr$ and 
$\norm{x_t}\le 
\norm{\str}$ therefore implies
\begin{equation}\label{eq:tr-lin-subopt-bound}
\f(x_t) - \f(\str) \le \half\norm{\As^{1/2}(x_t - \str)}^2 + 
\ltr\norm{\str}(\norm{\str} - \norm{x_t}).
\end{equation}
When $\norm{y_t} \ge \norm{\str}$ we have $\norm{x_t} = \norm{\str}$ 
and the second term vanishes. When $\norm{y_t} < \norm{\str}$,
\begin{flalign}\label{eq:tr-lin-norm-diff}
\norm{\str} - \norm{x_t} & = \norm{\str} - \norm{y_t} - 
\frac{\norm{y_t}}{\norm{\str}}\cdot (\norm{\str} - \norm{y_t} )
\nonumber\\& = \norm{\str}\alpha^2 \le 
4e^{-4t/\sqrt{\kappa(\As)}}\norm{\str}.
\end{flalign}
We also have
\begin{flalign}
  \label{eq:tr-linear-quad}
  \norm{\As^{1/2}(x_t - \str)}  &= 
    \norm{\left([1-\alpha]p(\As)\As-I\right)\As^{1/2}\str}
\nonumber \\
  & \le (1+|\alpha|)\norm{\left(p(\As)\As-I\right)\As^{1/2}\str} + 
  |\alpha| \norm{\As^{1/2}\str} %
  \nonumber \\& \le 6 \norm{\As^{1/2}\str}  e^{-2t/\sqrt{\kappa(\As)}},
\end{flalign}
where in the final transition we used our upper bounds on $\alpha$ 
and $\opnorm{p(\As)\As - I}$, as well as $|\alpha|\le 1$.
Substituting the bounds~\eqref{eq:tr-lin-norm-diff}
and~\eqref{eq:tr-linear-quad} into
inequality~\eqref{eq:tr-lin-subopt-bound}, we have
\begin{equation}
  \label{eq:tr-lin-time-bound-stronger}
  \f(x_t) - \f(\str) \le \left(18 (\str)^T \As \str +  
  4\ltr\norm{\str}^2\right) 
  e^{-4t/\sqrt{\kappa(\As)}},
\end{equation}
and the final bound follows from recalling that 
$\f(0)-\f(\str) = \half (\str)^T \As \str + \frac{\ltr}{2}\norm{\str}^2$ 
and substituting $\kappa(\As) = (\lmax + \ltr)/(\lmin + \ltr)$. 
To conclude the proof we note 
that 
$(1-\alpha)p(\As) = (1-\alpha)p(A + \ltr I) = \tilde{p}(A)$ for some 
$\tilde{p} \in \mc{P}_t$, so that $x_t \in \mc{K}_t(A, b)$ and 
$\norm{x_t}\le R$, 
and therefore $\f(\itertr_t) \le \f(x_t)$.

\subsection{Sublinear convergence}\label{sec:upper-sub-proof}
Let $\Agam \defeq  A - \lmin I \succeq 0$ and apply 
Lemma~\ref{lem:agd} with $M=\Agam$, $v=b$ and $r=\norm{\str}$ to 
obtain $y_t \in \Krylov[t][\Agam, b] = \Krylov$ such that 
$\norm{y_t}\le 
\norm{\str}$ and
\begin{equation}
  \label{eq:tr-sublin-yt}
  \f[\Agam,b](y_t) - \f[\Agam,b](\str) \le 
  \f[\Agam,b](y_t) - \min_{\norm{x}\le\norm{\str}}\f[\Agam,b](x) \le 
  \frac{4\opnorm{\Agam}\norm{\str}^2}{(t+1)^2}.
\end{equation}
If $\lmin \ge 0$, equality~\eqref{eq:tr-gamma-pivot-outline} with 
$\lambda=-\lmin$ along with~\eqref{eq:tr-sublin-yt} means we are done, 
recalling that $\opnorm{\Agam} = \lmax-\lmin$. 
For $\lmin< 0$, apply 
Lemma~\ref{lem:eigenvec-tight} with $M=\Agam$ and $v=b$ to obtain 
$z_t\in\Krylov$ such that
\begin{equation}\label{eq:tr-sublin-zt}
\norm{z_t} = 1
~~\mbox{and}~~
z_t^T \Agam z_t \le \frac{\opnorm{\Agam}}{16(t-\half)^2}
\log^2\left(4\frac{\norm{b}^2}{(\vmin^T b)^2}\right).
\end{equation}
We form the vector
\begin{equation*}
x_t = y_t + \alpha \cdot z_t\in\mc{K}_t(A,b),
\end{equation*}
and choose $\alpha$ to satisfy
\begin{equation*}
\norm{x_t} = \norm{\str}
~~\mbox{and}~~
\alpha \cdot z_t ^T (\Agam y_t  + b) = 
\alpha \cdot z_t ^T  \grad \f[\Agam,b](y_t) \le 0.
\end{equation*}
We may 
always choose such an $\alpha$, as $\norm{y_t}\le\norm{\str}$ and 
therefore $\norm{y_t + \alpha z_t} = \norm{\str}$ has both a 
non-positive and a non-negative solution in $\alpha$. Moreover because 
$\norm{z_t}=1$ 
we have that $|\alpha| \le 2\norm{\str}$. 
The property 
$\alpha \cdot z_t ^T  \grad \f[\Agam,b](y_t) \le 0$ of our 
construction of $\alpha$ along with $\hess \f[\Agam,b] = \Agam$
gives
\begin{equation*}
  \f[\Agam,b](x_t) = \f[\Agam,b](y_t) + \alpha \cdot z_t ^T  \grad 
  \f[\Agam,b](y_t) + \frac{\alpha^2}{2} z_t^T \Agam z_t
  \le \f[\Agam,b](y_t) + \frac{\alpha^2}{2} z_t^T \Agam z_t.
\end{equation*}
Substituting this bound along with $\norm{x_t}=\norm{\str}$ and 
$\alpha^2 \le 4\norm{\str}^2$ into~\eqref{eq:tr-gamma-pivot-outline} 
with $\lambda=-\lmin$ gives
\begin{equation}
  \label{eqn:last-line-sublinear}
  \f(x_t) - \f(\str) \le \f[\Agam,b](y_t) - \f[\Agam,b](\str) + 
  2\norm{\str}^2 
  z_t^T \Agam z_t.
\end{equation}
Substituting the bounds~\eqref{eq:tr-sublin-yt} 
and~\eqref{eq:tr-sublin-zt} concludes the proof for the case $\lmin < 
0$.
\section{Proofs of randomization strategies}

\subsection{Proof of Corollary~\ref{corr:gradient-pert}}
\label{sec:proof-gradient-pert}
Throughout this proof, we use the notational shorthand $f = \fcu$.
Corollary~\ref{corr:gradient-pert} follows from three basic observations 
about
the effect of adding a small uniform perturbation to $b$, which we summarize
in the following lemma (see Section~\ref{sec:proof-rand-prop} for a proof).

\newcommand{\rhosol}{\rho \norms{\sol}}

\begin{lemma}
  \label{lem:rand-prop}
  Set $\tilde{b}=b+\sigma \univar$, where $\univar \sim
  \uniform(\sphere^{d-1})$ and $\sigma>0$.  Let
  $\tilde{f}\left(x\right)=\frac{1}{2}x^{\T}Ax+\tilde{b}^{\T}x+\frac{1}{3}\rho\norm{
    x}^3$ and let $\soltilde$ be a global minimizer of $\tilde{f}$. Then, the 
    following holds 
  for any $\delta > 0$:
  \begin{enumerate}[label=(\roman*)]
  \item For $d > 2$, $\P( |\tilde{b}\supind{1}|
    \le \sqrt{\pi}\sigma\delta/\sqrt{2 d}) \leq\delta$.
  \item $| f(x)-\tilde{f}(x) | \le \sigma\norm{x}$
    for all $x\in\mathbb{R}^{d}$.
  \item $\big| \norm{\sol}^2- \norm{\soltilde}^2 \big| \leq2\sigma/\rho$.
  \end{enumerate}
\end{lemma}

With Lemma~\ref{lem:rand-prop} in hand, our proof proceeds in three parts:
in the first two, we provide bounds on the iteration complexity of each of
the modes of convergence that Theorem~\ref{theorem:gradient-descent-cubic}
exhibits in the perturbed problem with vector $\tilde{b}$. The final part
shows that the quality of the (approximate) solutions $\tilde{x}_t$ and
$\soltilde$ is not much worse than $\sol$.

Let $\tilde{f}, \tilde{b}$ and $\soltilde$ be as in
Lemma~\ref{lem:rand-prop}. By Theorem~\ref{theorem:gradient-descent-cubic},
$\tilde{f}(\tilde{x}_{t}) \le \tilde{f}(\soltilde) + \eps$ for
all
\begin{equation}
  \label{eqn:perturbed-time}
  \begin{split}
    t & \geq \frac{6}{\eta}
    \left(\log\left(1+\frac{\hinge{-\lmin}^{2}}{4\rho 
      |\tilde{b}\supind{1}|}\right)
    +\log\frac{ (\beta  + 
      2\rho\norm{\soltilde})\norm{\soltilde}^2}{\eps}\right) 
    \\ & \quad \quad  \times 
    \min\left\{\frac{1}{\rho 
      \norm{\soltilde} + \lmin},
    \frac{10 \norm{\soltilde}^2}{\eps}\right\}.
  \end{split}
\end{equation}
We now turn to bounding expression~\eqref{eqn:perturbed-time}.

\paragraph{Part 1: bounding terms outside the logarithm}
Recalling that $\sigma = \frac{\rho\sigbar\eps}{12(\beta+2\rhosol)}$ and 
$\eps \le (\half\beta + \rhosol)\norm{\sol}^2$, we have $\sigma \le 
\frac{\rho}{24}\sigbar\norm{\sol}^2$. Part (iii) of 
Lemma~\ref{lem:rand-prop} gives
\begin{equation*}
  |\norm{\sol}^2 - \norm{\soltilde}^2| \le 2\sigma/\rho \le \sigbar 
  \norm{\sol}^2/12,
  ~~ \mbox{so} ~~
  \norm{\soltilde}^2 \in(1 \pm \sigbar/12) \norm{\sol}^2.
\end{equation*}
Consequently, using $\sigbar\le1$ we have
\begin{equation*}
  \big| \norm{\sol} - \norm{\soltilde}\big|
  \le \frac{2 \sigma}{\rho (\norm{\sol} + \norm{\soltilde})}
  \le \frac{2 \sigbar \eps }{12(1+\sqrt{11/12})\norm{\sol}(\beta+2\rhosol)}
  \le \frac{ \sigbar \eps }{20\rho\norms{\sol}^2}~.
\end{equation*}
Now, suppose that $\frac{\eps}{10 \norm{\sol}^2} \le \rho \norm{\sol} + \lmin$.
Substituting this above yields
$|\norm{\sol} - \norm{\soltilde}|
\le \frac{\sigbar}{2\rho}(\rhosol + \lmin)$, and
rearranging, we obtain
\begin{equation*}
  \rho\norm{\soltilde}+ \lmin \ge \left(1-0.5\sigbar\right)(\rhosol + \lmin) 
  \ge 
  \frac{\rhosol + \lmin}{1+\sigbar}
\end{equation*}
because $\sigbar \le 1$.
We combine the preceding bounds to obtain
\begin{equation}
  \label{eqn:upper-bound-min-stuff}
  \min\left\{\frac{1}{\rho \norm{\soltilde} + \lmin},
    \frac{10 \norm{\soltilde}^2}{\eps}\right\}
  \le (1+\sigbar)\min\left\{\frac{1}{\rho \norm{{\sol}} + \lmin},
  \frac{10 \norm{{\sol}}^2}{\eps}\right\}
\end{equation}
where we have used $\norm{\soltilde} \le (1+\sigbar)\norms{\sol}^2$ and
$\norm{\soltilde} \ge \sqrt{1-\sigbar/12}\norms{\sol}^2 \ge
\norms{\sol}^2/(1+\sigbar)$.

\paragraph{Part 2: bounding terms inside the logarithm}
  Fix a confidence 
level  $\delta \in (0, 1)$. By Lemma~\ref{lem:rand-prop}(i),
$1 / |\tilde{b}^{\left(1\right)}| \leq \sqrt{2 d} / 
(\sqrt{\pi}\sigma\delta)\leq
\sqrt{d}/(\sigma\delta)$ with 
probability at least $1-\delta$, so
\begin{flalign*}
  6\log \left(1+\frac{\hinge{-\lmin}^2}{4 \rho |\tilde{b}\supind{1}|}\right)
  & \le
  6 \log \left(1 + \frac{\hinge{-\lmin}^2 \sqrt{d}}{4 \rho \sigma \delta}\right)
  \stackrel{(\star)}{\le} 6 \log \left(1 +\I_{\{\lmin<0\}}\frac{3\sqrt{d}}{ 
    \sigbar 
	\delta}\right) 
\\
 & \quad + 6\log\frac{ (\beta + 2\rhosol) \norms{\sol}^2}{\eps}
  = 6 \tiltgrow(\delta,\sigbar) + 6 \tiltconv(\eps),
\end{flalign*}
where inequality $(\star)$ uses that $\rhosol \ge \hinge{-\lmin}$ and 
$\eps \le (\beta + \half \rhosol)\norms{\sol}^2$.
Using $\norm{\soltilde} \le \sqrt{1+\sigbar/12}\norm{\sol}$ yields the upper 
bound
\begin{equation*}
  6\log \frac{ (\beta + 2\rho \norm{\soltilde}) \norm{\soltilde}^2}{\eps}
  \le 
  6\log \frac{ (\beta + 2\rhosol) \norm{{\sol}}^2}{\eps}
  + 9\log(1+\sigbar/12)
  \le 8 \tiltconv(\eps),
\end{equation*}
where the second inequality follows as $9\log(1+\sigbar/12)<2\log2\le 
2\log\frac{ (\beta + 2\rhosol) \norm{{\sol}}^2}{\eps}$.

Substituting the above bounds and
the upper bound~\eqref{eqn:upper-bound-min-stuff} 
into expression~\eqref{eqn:perturbed-time} gives
the iteration bounds
in Corollary~\ref{corr:gradient-pert}. To complete
the proof we need only bound the quality of the solution $\tilde{x}_t$.

\paragraph{Part 3: solution quality}
 We recall that $\sigma = 
\frac{\rho\sigbar\eps}{12(\beta+2\rhosol)}\le \frac{\sigbar\eps}{24\norm{\sol}}$ and 
$\norm{\soltilde} \le \sqrt{1+\sigbar/12} \norm{\sol} \le 
\sqrt{2}\norm{\sol}$, so $\sigma \le 
\frac{\sigbar\eps}{\norm{\sol}+\norm{\soltilde}}$. Thus, whenever 
$\tilde{f}(\tilde{x}_t) \le \tilde{f}(\soltilde) + \eps$,
\begin{align*}
  f(\tilde{x}_{t})
  & \overset{\text{(a)}}{\leq} \tilde{f}(\tilde{x}_{t}) + \sigma\norm{\tilde{x}_{t}}
    \leq\tilde{f}(\soltilde) + \eps +\sigma\norm{\tilde{x}_{t}}
    \overset{\text{(b)}}{\leq} \tilde{f}(\soltilde) + \eps
    +\sigma\norm{\soltilde} \\
  & \overset{\text{(c)}}{\leq}\tilde{f}(\sol) + \eps +\sigma\norm{\soltilde}
    \overset{\text{(d)}}{\leq} f(\sol) + \sigma(\norm{\soltilde} + \norm{\sol}) + 
    \eps
  \le f(\sol) + (1+\sigbar)\eps,
\end{align*}
where transitions (a) and (d) follow from part (ii) of  
Lemma~\ref{lem:rand-prop}, transition (b) follows from $\norm{\tilde{x}_t} \le 
\norm{\soltilde}$ (Lemma~\ref{lemma:monotone}), and transition (c) 
follows from  $\tilde{f}(\soltilde) = \inf_{z\in \R^d}\tilde{f}(z)$.

\subsection{Proof of Lemma~\ref{lem:rand-prop}}
\label{sec:proof-rand-prop}

To establish part (i) of the lemma, note that marginally
$[\univar^{\left(1\right)}]^2\sim \betadist(\half,\frac{d-1}{2})$ and that
$\univar^{\left(1\right)}$ is symmetrically distributed about 0. Therefore, 
for $d>2$ the density of
$\tilde{b}^{\left(1\right)}=b^{\left(1\right)}+\sigma \univar^{\left(1\right)}$ is
maximal at $b^{\left(1\right)}$ and is monotonically decreasing in the
distance from $b^{\left(1\right)}$. Therefore we have
\begin{align}\label{eq:beta-tail-bound}
  \P\left(|\tilde{b}\supind{1}| \le \sigma\sqrt{\pi}\delta/\sqrt{2 d}\right) \le  
  \P\left(|\univar\supind{1}| \le \sqrt{\pi}\delta / \sqrt{2 d} \right) \le \delta,
\end{align}
where the bound $p_1(u) \le \sqrt{d / (2\pi u)}$
on the density $p_1$ of $\univar\supind{1}$ yields the last inequality.

Part (iii) of the lemma is immediate, as
\begin{equation*}
  |f\left(x\right)-\tilde{f}\left(x\right) | =
  |(b-\tilde{b})^{\T}x|\leq\sigma\norm{\univar} \norm{
    x} =\sigma \norm{x}.
\end{equation*}

Part (ii) of the lemma follows by viewing $\norms{\sol}^2$ as a function of
$b$ and noting that $b \mapsto \norms{\sol}^2$ is $2/\rho$-Lipschitz
continuous. To see this claim, we use the inverse function theorem.  First,
$\norms{\sol}^2$ is a function of $b$, because $\sol$
may be non-unique only when $\norms{\sol} = \hinge{-\lmin}/\rho$ (see
Proposition~\ref{prop:characterization}).  Next, from the relation $b=-\As\sol$,
the inverse mapping $\sol \mapsto b$ is smooth with
Jacobian
\begin{equation*}
  \frac{\del b}{\del \sol}=-\As-\rho\frac{\sol(\sol)^{\T}}{\norm{\sol}}
    =-\grad^2f\left(\sol\right).
\end{equation*}
Let us now evaluate $\del\norm{\sol}^2/\del b$ when
the mapping $\sol \mapsto b(\sol) = -(A + \rho \norm{\sol} I) \sol$
is invertible, i.e.\ when $\norm{\sol} > \hinge{-\lmin}/\rho$;
the inverse function theorem yields
\begin{equation*}
  \frac{\del\norm{\sol}^2}{\del 
  b}=\frac{\del\left((\sol)^{\T}\sol\right)}{\del 
  b}=2\frac{\del \sol}{\del 
  b}\sol=-2\left(\grad^2f\left(\sol\right)\right)^{-1}\sol.
\end{equation*}
The mapping $\sol \mapsto (\grad^2f(\sol))^\dag \sol$ is continuous in $\sol$ even when
$\As\succeq 0$ is singular, and therefore the preceding expression is valid
(as the natural limit) when $\norm{\sol} \to (-\lmin)_+ / \rho$.  Moreover, 
since
$\grad^2f\left(\sol\right)\succeq\rho \sol (\sol)^T / \norm{\sol}$, we have
\begin{equation*}
  \normbigg{\frac{\del\norm{\sol}^2}{\del b}}
  =2\norm{\left(\grad^2f\left(\sol\right)\right)^{\dag}\sol}
  \leq2\norm{\left(\rho \sol(\sol)^{\T} /\norm{\sol} \right)^{\dag}\sol} 
  =\frac{2}{\rho}.
\end{equation*}
We thus conclude that $b \mapsto \norm{\sol}^2$ is a $2/\rho$-Lipschitz
continuous function of $b$, and therefore 
$|\norm{\sol}^2-\norm{\soltilde}^2 |
\leq\left(2/\rho\right) \| b-\tilde{b}\| =2\sigma/\rho$.

\subsection{Proof of Corollary~\ref{cor:tr-rand-joint}}
\label{sec:proof-krylov-random}

We prove only the trust region guarantee; the other is then an immediate
consequence of the fact that optimality gaps in the trust region problem
bound those in the cubic-regularized problem (recall
Sec.~\ref{sec:krylov-cubic}). We revisit the proof of sublinear convergence
in Sec.~\ref{sec:upper-sub-proof}, noting that if
$\lmin \ge 0$, the corollary is immediate, so we need consider only the
case that $\lmin < 0$. Let $z \in
\Krylov[t][A,\univar]$ be the vector~\eqref{eq:tr-sublin-zt} that
Lemma~\ref{lem:eigenvec-tight} guarantees and let $y \in
\Krylov[t][A,b]$ be the vector~\eqref{eq:tr-sublin-yt} that
Lemma~\ref{lem:agd} guarantees. Then for $\jittertr_t$ in the corollary, we
have as in the final inequality~\eqref{eqn:last-line-sublinear} of the
sublinear convergence proof that
\begin{align*}
  \f(\jittertr_t) - \f(\str)
  & \le \f[\Agam,b](y) - \f[\Agam](\str)
  + 2 \norm{\str}^2 z^T \Agam z \\
  & %
    \le \frac{4 \opnorm{\Agam}R^2}{
    (t + 1)^2}
  + \frac{\opnorm{\Agam} R^2 }{8
    (t - \half)^2}
  \log^2 \left(4 \frac{1}{(\univar\supind{1})^2}\right),
\end{align*}
where $\Agam = A - \lmin I$.
Now we recognize that $\opnorm{\Agam} = \opnorm{A - \lmin I} = \lmax 
- \lmin$
and that by the rotational symmetry of $\univar$, we have
$(\univar\supind{1})^2
\sim \betadist(\half, \frac{d-1}{2})$. Thus
$(\univar\supind{1})^2\ge \frac{\pi}{2}\cdot\frac{\delta^2}{d}\ge 
\frac{\delta^2}{d}$
with probability at least $1-\delta$ (recall Eq.~\eqref{eq:beta-tail-bound}).
\section{Numerical experiment details}\label{sec:exp-details}
We provide details on the random problem 
instances for the experiments in Section~\ref{sec:exp-rates} 
and~\ref{sec:exp-rand}.

\paragraph{Random problem generation, $\kappa < \infty$}
We generate random cubic 
regularization instances $(A, b, \rho)$ as 
follows. We take $\lmax = 1$ and draw $\lmin\sim \uniform[-1,-0.1]$.
We then fix two 
eigenvalues of $A$ to be $\lmin,\lmax$ and draw the other $d-2$ 
eigenvalues i.i.d.\ $\uniform[\lmin, \lmax]$. We take $A$
diagonal with said eigenvalues; this is without much loss of generality (as 
the methods are rotationally invariant), and it allows us to 
quickly compute matrix-vector products.

For a desired condition number $\kappa$, we let
\begin{equation*}
\ltr \defeq \frac{\lmax - \kappa\lmin}{\kappa - 1}
\end{equation*} 
and as usual denote $\Atr = A + \ltr I$. 
To generate $b$, $\rho$, we draw a standard normal $d$-dimensional 
vector 
$v\sim \mc{N}(0; I)$ and let
\begin{equation*}
b =  \sqrt{\frac{2}{v^T \Atr^{-1} v + 
		\frac{\ltr}{3} v^T\Atr^{-2} v} }\cdot  v
~,~
\rho = \frac{\ltr}{\norms{\Atr^{-1}b}},
\end{equation*}
The above choice of $b$ and $\rho$ guarantees that $\rho 
\norm{\Atr^{-1}b} = \ltr$, so $\scu = -\Atr^{-1}b$ is the unique 
solution and the problem condition number satisfies
\begin{equation*}
\frac{\lmax+\rs}{\lmin+\rs} = \frac{\lmax + \ltr}{\lmin + \ltr} = \kappa
\end{equation*}
as desired. Moreover, our scaling of $b$ guarantees that
\begin{equation*}
\fcu(0)-\fcu(\scu) = \half (\scu)^T \Atr \scu + \frac{\rho}{6}\norm{\scu}^3
=
\half \left(b^T \Atr^{-1} b + \frac{\ltr}{3}b^T \Atr^{-2} b\right) =1.
\end{equation*}

For every value of $\kappa$, we generate 5,000 independent problem instances.

\paragraph{Random problem generation, $\kappa = \infty$} We let 
$A=\diag(\lambda)$ where $\lambda_1 = \lmin = -0.5$, $\lambda_d = 
\lmax = 0.5$ and $\lambda_2, \ldots, \lambda_{d-1}$ are i.i.d.\
$\uniform[\lmin + \gamma, \lmax]$ where we take the eigen-gap 
$\gamma=10^{-4}$ and $d=10^6$. 
As $\kappa=\infty$, we let
\begin{equation*}
\ltr = -\lmin
\end{equation*}
and denote $\hat{A}_{\ltr} \defeq \diag(\lambda_2+\ltr, \ldots, 
\lmax+\ltr)$. 
We generate $b$ and $\rho$ by drawing a standard normal 
$(d-1)$-dimensional vector $v$, and letting
\begin{equation*}
b_1 = 0
~,~
b_{2:d} = \sqrt{\frac{2}{v^T \hat{A}_{\ltr}^{-1} v + 
		(1+\tau^2)\frac{\ltr}{3} v^T\hat{A}_{\ltr}^{-2} v} } v
~,~
\rho = \frac{\ltr}{\norms{\hat{A}_{\ltr}^{-1}b_{2:d}}\sqrt{1+\tau^2}},
\end{equation*}
where $\tau$ is a parameter that determines the weight of the eigenvector 
corresponding to $\lmin$ in the solution (when $\tau = \infty$ we have a 
pure eigenvector instance); we take $\tau=10$. A global minimizer $\scu$ 
of this problem instance $(A,b,\rho)$ has the form
\begin{equation*}
[\scu]_1 = \pm \tau \norms{\hat{A}_{\ltr}^{-1}b_{2:d}}
~,~
[\scu]_{2:d} = -\hat{A}_{\ltr}^{-1}b_{2:d}.
\end{equation*}
As in the case $\kappa < \infty$, the scaling of $b$ 
guarantees 
$\fcu(0)-\fcu(\scu) = 1$.

When $\kappa=\infty$, the eigen-gap $\gamma = \lambda_2 - \lmin$ strongly
affects optimization performance. We explore this in
Figure~\ref{fig:exp-gap}, which repeats the experiment above with different
values of $\gamma$ (and $d=10^5$). As the figure shows, the non-randomized
Krylov subspace solution becomes more suboptimal as $\gamma$ increases,
which is expected: when $\gamma$ is large, finding the components of $\scu$
in the direction $\vmin$ becomes more important. Randomization ``kicks-in''
with linear convergence after roughly $\log d / \sqrt{\gamma}$ iterations.

To create each plot, we draw 10 independent problem instances from the 
distribution described above, and for each problem instance we run each 
randomization approach with 50 different random seeds; we observe that 
sampling problem instances and randomization seeds each contribute 
similar variation in the final ensemble of results.

\begin{figure}
	\centering
	\includegraphics[width=0.9\columnwidth]
	{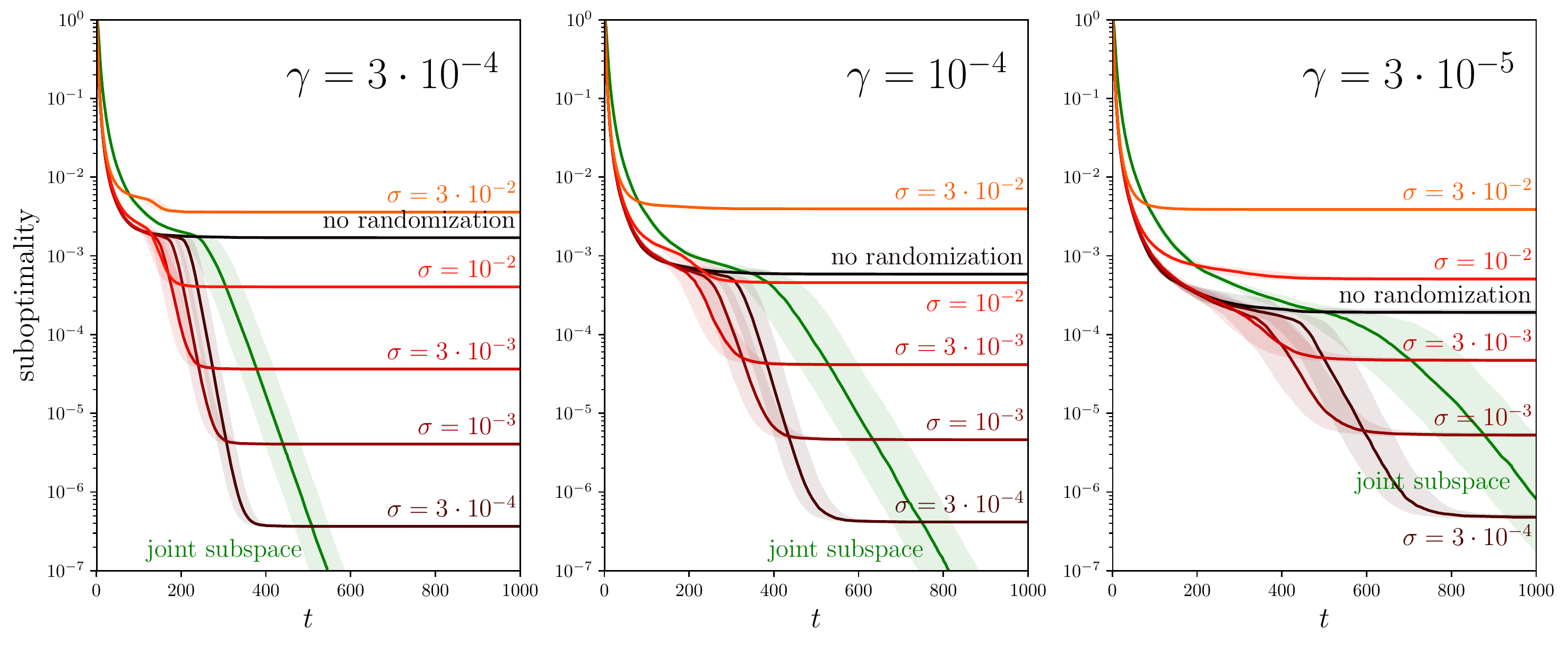}
	\caption{\label{fig:exp-gap}Optimality gap of 
		Krylov subspace solutions on random 
		cubic-regularization problems, versus subspace 
		dimension $t$. Each plot shows result for problem instances with a 
		different eigen-gap $\gamma = (\lmax - \lmin)/(\lambda_2 - 
		\lmin)$, where $\lambda_2$ is the smallest eigenvalue larger than 
		$\lmin$. Each line represents median suboptimality, and shaded 
		regions represent inter-quartile range. Different lines 
		correspond to different randomization settings.
	}
\end{figure}

\paragraph{Computing Krylov subspace solutions}
We use the Lanczos process to obtain a tridiagonal representation of $A$ as 
described in Section~\ref{sec:lanczos}. To obtain full optimization traces 
we solve equation~\eqref{eq:lambda-search} after every Lanczos 
iteration, warm-starting $\lambda$ with the solution from the previous 
step and the minimum eigenvalue of the current tridiagonal matrix. 
We use the Newton method described by~\citet[Algorithm 
6.1]{CartisGoTo11} to solve the equation~\eqref{eq:lambda-search} in 
the Krylov subspace. For the $\kappa < \infty$ experiment, we stop 
the process when
$|\norm{A_{\lambda}^{-1}b}-\lambda/\rho| < 10^{-12}$ or after 25 
tridiagonal system solves are computed. For the $\kappa = \infty$ 
experiment we allow up to 100 system solves.

\section{Proofs from Section \ref{sec:majorization}}

\subsection{Proof of Lemma~\ref{lem:SSP}}
\label{sec:proof-SSP}
Recalling the notation $\As = A + \rho \norm{\scu} I$ and that
$(\scu)^T \As \scu = -b^T \scu$,
the minimal value of $\fcu$ admits the
bound
\begin{equation*}
  \fcu(\scu) 
  = - \half ({\scu})^{\T}\As\scu - \frac{\rho\norm{\scu}^3}{6}
  \leq-\frac{\rho\norm{\scu}^{3}}{6}.
\end{equation*}
Corollary~\ref{cor:tr-rand-joint} thus implies that the output of $x = 
\callSSP{A,b,\rho,\beta,\delta}$  satisfies
\begin{flalign*}
\fcu(x) &\le  -\frac{1}{6}\rho \norm{\scu}^3 + 
 \frac{\beta \norm{\scu}^2}{\Tinner^2}
\left[ 
4 + \frac{\I_{\{\lmin < 0\}}}{2}
\log^2\left(\frac{4 d}{\delta^2}\right)
\right] 
\\ &
\le -\frac{1}{6}\rho \norm{\scu}^3 + \frac{1}{12} \rho 
r\norm{\scu}^2,
\end{flalign*}
where we substituted $\Tinner \ge \sqrt{\frac{12 \beta}{\rho \slb} \left(4 +
  \half \log^2 \frac{4d}{\delta^2}\right)}$, with
probability at least $1 - \delta$. Consequently, $\norm{\scu}\ge 
r$ implies $\fcu(x) \le -\frac{1}{12}\rho \norm{\scu}^3 \le  
-\frac{1}{12}\rho r^3$. 
	
\subsection{Proof of Lemma~\ref{lem:SFSP}}
\label{sec:proof-SFSP}

Let $f(z)=\half z^T  (A+2\rho r I)z + b^T z$, let $x\subopt = -(A+2\rho r 
I)^{-1} b$ be the minimizer of $f$, and let $x= 
\callSFSP{A,b,\rho,\beta,\epsgrad}$. Note that $x$ is the output of 
$\Tfinal$ conjugate gradient (CG) steps for minimizing $f$.  That 
$\norm{x} \le \norm{x\subopt}$ holds as CG iterates have 
nondecreasing norm and converge to $x\subopt$~\cite[Theorem 
2.1]{Steihaug83}. Moreover, bounds on $f(x)-f(x\subopt)$ follow from 
standard convergence analysis of CG; for convenience we simply apply the 
first bound of Corollary~\ref{cor:cu} with $\rho=0$, obtaining
\begin{equation*}
f(x)-f(x\subopt) \le 36[ f(0) -f(x\subopt)] \exp\left\{-4\Tfinal 
\sqrt{\frac{\rho r}{\beta + 2\rho r}}\right\} \le 
\frac{\epsgrad^2}{2(\beta+2\rho r)},
\end{equation*}
where we substituted $\Tfinal \ge \frac{1}{4}\sqrt{\frac{\beta+2\rho r}{\rho 
r}} \log \frac{36(\beta+2\rho r)^2 
	r^2}{\epsgrad^2}$ and $f(0) -f(x\subopt) = \half 
x\subopt^T (A+2\rho r I) x\subopt \le \half (\beta + 2\rho r) r^2$. Using 
also
\begin{equation*}
\norm{(A+2\rho r I)x + b} = \norm{\grad f(x)} \le \sqrt{2(\beta + 2\rho 
r)(f(x)-f(x\subopt))},
\end{equation*}
we get $\norm{(A+2\rho r I)x + b} \le \epsgrad$, whence the final bound 
$\norm{Ax + b} \le \epsgrad + 2\rho r^2$ follows by substituting 
$\norm{x}\le r$. 

\subsection{Proof of Proposition~\ref{prop:SP}}
\label{sec:proof-SP}

The second result is
inequality~\eqref{eqn:nesterov-polyak-progress}.
For the first,
we argue three facts: first, that the consequences of
Lemma~\ref{lem:SSP} hold in each call to \callSSP{};
second, that when the algorithm terminates it returns an approximate
second-order stationary point; and third, that the total number of
Hessian-vector products is bounded.
We begin with the first. The conclusions of
Lemma~\ref{lem:SSP} fail in iteration $k$ of Alg.~\ref{alg:SP}
with probability at most $\delta / 2k^2$, and so a union bound gives
\begin{equation*}
  \P(\mbox{any failure}) \le \sum_{k = 1}^\infty \frac{\delta}{2k^2}
  < \delta.
\end{equation*}
We perform our analysis deterministically in the event that no
failures occur.

To prove that the algorithm terminates with a second-order stationary 
point, let $K$ be the iteration at which
Alg.~\ref{alg:SP} fails to make enough progress, that is, $g(y_K) >
g(y_{K-1}) - \frac{1}{12}\rho r^3$.  Let $\Delta\opt$ minimize
model~\eqref{eqn:majorization} at $y = y_{K-1}$, and let $\Delta\subfin$ 
be the output of the call to \callSFSP{}, so that $y\subfin  = y_{K-1} + 
\Delta\subfin$ is the output of \callSP{}. The fact that $g(y_K) >
g(y_{K-1}) - \frac{1}{12}\rho r^3$ implies $\norm{\Delta\opt} \le r$
(Lemma~\ref{lem:SSP}), and 
since $\hess g(y_{K-1}) + \rho \norm{\Delta\opt}I \succeq 0$ by 
Proposition~\ref{prop:characterization}, the condition $-\rho r 
I \preceq \hess g(y_{K-1}) \preceq \beta I$ of Lemma~\ref{lem:SFSP} holds 
for the call \callSFSP{}. Similarly, Proposition~\ref{prop:characterization} 
requires that $\norm{(A+\rho \norm{\Delta\opt})^\dagger b} \le 
\norm{\Delta\opt}$ and consequently
\begin{equation*}
  \norm{(A+2\rho r I)^{-1} b} \le \norm{(A+\rho \norm{\Delta\opt})^\dagger 
    b} \le \norm{\Delta\opt} \le r,
\end{equation*}
so that the second condition of Lemma~\ref{lem:SFSP} holds.
Applying the lemma, we obtain
\begin{equation*}
  \norm{\Delta\subfin} \le \slb~~\mbox{and}~~\norm{\hess g(y_{K-1}) 
    \Delta\subfin + \grad g(y_{K-1})} \le \epsgrad + 2\rho \slb^2.
\end{equation*}

Now we demonstrate approximate stationarity.  Using that $\hess g$ is
$2\rho$-Lipschitz continuous, the bounds $\hess g(y_{K-1}) \succeq -\rho
\slb I$ and $\norms{\Delta\subfin} \le \slb$, where $\slb = \sqrt{\epsilon /
  9 \rho}$ imply
\begin{equation*}
  \hess g(y\subfin) \succeq \grad^2 g(y_{k-1}) - 2\rho 
  \norm{\Delta\subfin} I \succeq -3 \rho r I = -\sqrt{\rho \epsilon} I.
\end{equation*}
To control $y\subfin$, let $v = \grad \f[\hess g(y_{K-1}), \grad
  g(y_{K-1})](\Delta\subfin) = \hess g(y_{K-1}) \Delta\subfin + \grad
g(y_{K-1})$, noting that $\norm{v} \le \epsgrad + 2\rho \slb^2$ as
above. Moreover, the $2\rho$-Lipschitz continuity of $\hess g$ implies that
$\norm{\grad g(y\subfin) -v} \le \rho \norm{\Delta\subfin}^2 \le \rho
\slb^2$. Putting these two observations together and using $r =
\sqrt{\epsilon/(9\rho)}$ and $\epsgrad = 2\epsilon/3$, we have
the desired stationarity~\eqref{eqn:second-order-crit}:
\begin{equation*}
  \norm{\grad g(y\subfin)} \le \norm{\grad g(y\subfin)-v} + \norm{v} \le 
  \epsgrad + 3\rho r^2 = \eps.
\end{equation*}

For the final component of the proposition, we bound the total number
of Hessian-vector products the method requires.
The total number of gradient 
computation and calls to \callSSP{} is $\Touter = K = O(1)  \frac{\sqrt{\rho} 
(g(y_0) - \glb)}{\epsilon^{3/2}}$. The number of Hessian-vector products 
in  each call to \callSSP{} is at most
\begin{equation*}\label{eq:tinner-bound}
  \ceil{\sqrt{\frac{24 \beta}{\rho \slb}
    \left(4 + \half \log^2 \frac{16 \Touter^4 d}{\delta^2}\right)}}
  = O(1)
  \frac{\beta^{1/2} }{\rho^{1/4}\epsilon^{1/4}}
  \log \left[\frac{d}{\delta^2}\cdot \frac{\sqrt{\rho} (g(y_0) - 
  \glb)}{\epsilon^{3/2}}\right],
\end{equation*}
where we have used $\epsilon \le
\min\{\beta^2/\rho,\rho^{1/3}(g(y_0)-\glb)^{2/3}\}$.
Similarly simplifying the number of Hessian-vector product evaluations
in the call to \callSFSP{} gives
\begin{equation*}
\frac{1}{4}\sqrt{\frac{\beta+2\rho r}{\rho 
		r}} \log \frac{36(\beta+2\rho r)^2 
	r^2}{\epsgrad^2} = O(1)
\frac{\beta^{1/2} }{\rho^{1/4}\epsilon^{1/4}}
\log \left[\frac{\beta^{1/2} }{\rho^{1/4}\epsilon^{1/4}}\right],
\end{equation*}
and multiplying the last two displays by  $\Touter \le O(1)  
\frac{\sqrt{\rho} 
  (g(y_0) - \glb)}{\epsilon^{3/2}}$ implies the proposition.

\siam{
\bibliographystyle{abbrvnat}

}

\end{document}